\documentclass[draft]{compositio}
\usepackage{stmaryrd}
\usepackage{amsmath, amscd, amssymb, mathrsfs, accents, amsfonts}
\usepackage{url}
\usepackage[all]{xy}
\usepackage{tikz}
\def\nicecolourscheme{\shadedraw[top color=blue!0, bottom color=blue!20]}
\def\nicecolourschemedash{\shadedraw[top color=blue!0, bottom color=blue!20, draw=blue!50!black, dashed]}

\SelectTips{cm}{}

\newtheorem{theorem}{Theorem}[section]
\newtheorem{proposition}[theorem]{Proposition}
\newtheorem{lemma}[theorem]{Lemma}
\newtheorem{corollary}[theorem]{Corollary}
\newtheorem*{theoremn}{Theorem}

\theoremstyle{definition}
\newtheorem{definition}[theorem]{Definition}
\newtheorem{example}[theorem]{Example}
\newtheorem{remark}[theorem]{Remark}

\numberwithin{equation}{section}

\def\res{\operatorname{Res}}
\def\sg{\operatorname{sg}}
\def\Inj{\operatorname{Inj}}

\def\Proj{\operatorname{Proj}}

\def\Ker{\operatorname{Ker}}
\def\Im{\operatorname{Im}}

\def\can{\operatorname{can}}

\def\HH{\operatorname{HH}}
\def\K{\mathbf{K}}

\def\sing{\operatorname{Sg}}
\def\Hom{\operatorname{Hom}}
\def\uHom{\underline{\Hom}}

\def\Modd{\operatorname{Mod}}

\def\Ker{\operatorname{Ker}}

\def\straightK{\operatorname{K}}
\def\straightC{\operatorname{C}}
\def\holim{\operatorname{hocolim}}

\DeclareMathOperator{\Flat}{Flat}
\DeclareMathOperator{\qc}{qc}

\DeclareMathOperator{\Qco}{Qco}
\DeclareMathOperator{\Tr}{Tr}
\DeclareMathOperator{\End}{End}

\DeclareMathOperator{\skos}{K}

\DeclareMathOperator{\tr}{tr}
\DeclareMathOperator{\ch}{ch}
\DeclareMathOperator{\str}{str}
\DeclareMathOperator{\hmf}{hmf}
\DeclareMathOperator{\HMF}{HMF}
\DeclareMathOperator{\hf}{HF}

\DeclareMathOperator{\At}{At}

\begin{document}

\def\dbsing{\mathbf{D}_{\sing}}
\def\dbsinge{\mathbf{D}'_{\sing}}
\def\dcpx{D}
\def\trans{\tau}
\def\Res{\res\!}
\newcommand{\Ress}[1]{\res_{#1}\!}
\newcommand{\kosz}[1]{E[#1]}
\newcommand{\ud}{\mathrm{d}}
\newcommand{\kflat}[1]{\K(\Flat #1)}
\newcommand{\kflatl}[2]{\K_{#1}(\Flat #2)}
\newcommand{\koszul}[2]{\operatorname{E}[#1]}
\newcommand{\ukoszul}[2]{\straightK[#1]}
\newcommand{\cech}[2]{\check{\straightC}[#1]}
\newcommand{\cat}[1]{\mathcal{#1}}
\newcommand{\qder}[1]{\mathbf{D}(#1)}
\newcommand{\qderl}[2]{\mathbf{D}_{#1}(#2)}
\newcommand{\qderu}[2]{\mathbf{D}^{#1}(#2)}
\newcommand{\lto}{\longrightarrow}
\newcommand{\xlto}[1]{\stackrel{#1}\lto}
\newcommand{\mf}[1]{\mathfrak{#1}}
\newcommand{\md}[1]{\mathscr{#1}}
\newcommand{\kinj}[1]{\K(\Inj #1)}
\newcommand{\kinju}[2]{\K^{#1}(\Inj #2)}
\newcommand{\kinjl}[2]{\K_{#1}(\Inj #2)}
\newcommand{\kprof}[1]{\mathbf{N}(\Flat #1)}
\newcommand{\kprofu}[2]{\mathbf{N}^{#1}(\Flat #2)}
\newcommand{\kprofl}[2]{\mathbf{N}_{#1}(\Flat #2)}
\newcommand{\kproful}[3]{\mathbf{N}^{#1}_{#2}(\Flat #3)}
\newcommand{\homd}[2]{[#1,#2]_{\mathbb{D}}}
\newcommand{\homk}[2]{[#1,#2]_{\mathbb{K}}}
\newcommand{\homn}[2]{[#1,#2]_{\mathbb{N}}}
\newcommand{\lift}[1]{\dot{#1}}
\newcommand{\dlift}[1]{\ddot{#1}}
\def\qhom{\mathscr{H}\! om_{\qc}}
\def\rdev{\mathbb{R}}
\def\dual{*}
\def\uotimes{\underaccent{\;\;\:=}{\otimes}}
\newcommand{\qcmod}[1]{\Qco(#1)}
\def\shom{\operatorname{Hom}}
\def\l{\,|\,}
\def\sgn{\textup{sgn}}
\def\nablagr{\nabla_{\operatorname{gr}}}
\newcommand{\kproj}[1]{\K(\Proj #1)}
\newcommand{\kprojl}[2]{\K_{#1}(\Proj #2)}
\newcommand{\kproju}[2]{\K^{#1}(\Proj #2)}
\newcommand{\kprojul}[3]{\K^{#1}_{#2}(\Proj #3)}
\def\dpart{\mathbb{D}}
\def\homot{\simeq}
\def\ceta{\vartheta}
\def\ctr{\Tr^{\sg}}
\def\cf{\bold{cf}}
\def\bx{\bold{x}}
\def\by{\bold{y}}
\def\totimes{\otimes}
\def\di{Q}
\def\rint{R_{int}}
\def\rext{R_{ext}}
\def\intvar{n}
\def\idem{e}
\newcommand{\cotimes}[1]{\,\widehat{\otimes}_{#1}\,}

\title{Pushing forward matrix factorisations}
\author{Tobias Dyckerhoff}
\email{tobias.dyckerhoff@yale.edu}
\address{Department of Mathematics, Yale University}

\author{Daniel Murfet}
\email{daniel.murfet@math.ucla.edu}
\address{Department of Mathematics, UCLA}

\begin{abstract}
We describe the pushforward of a matrix factorisation along a ring morphism in terms of an idempotent defined using relative Atiyah classes, and use this construction to study the convolution of kernels defining integral functors between categories of matrix factorisations. We give an elementary proof of a formula for the Chern character of the convolution generalising the Hirzebruch-Riemann-Roch formula of Polishchuk and Vaintrob.
\end{abstract}

\maketitle

\section{Introduction}

A \emph{linear factorisation} of an element $W$ in a ring $R$ (all our rings are commutative) is a $\mathbb{Z}/2$-graded $R$-module $X$ together with an odd $R$-linear map $d: X \lto X$, called the differential, which squares to $W \cdot 1_X$. The pair $(X,d)$ is called a \emph{matrix factorisation} if each $X^i$ is free over $R$, and a \emph{finite rank matrix factorisation} if each $X^i$ is free of finite rank. We write $\HMF(R, W)$ for the homotopy category of matrix factorisations and $\hmf(R,W)$ for the full subcategory of finite rank matrix factorisations. These objects were introduced by Eisenbud \cite{eisenbud} as a way of representing a special class of modules over local rings of hypersurface singularities, and they have played an important role in singularity theory and more recently in homological mirror symmetry and string theory. Because of this latter connection $W$ is now often referred to as the potential.

Suppose that $\varphi: S \lto R$ is a ring morphism and that $W = \varphi(V)$. Viewing $X$ as an $S$-module via $\varphi$ we obtain a $\mathbb{Z}/2$-graded $S$-module $\varphi_*(X)$ with a differential squaring to $V \cdot 1_{\varphi_*(X)}$, and we call this linear factorisation of $V$ over $S$ the \emph{pushforward} along $\varphi$. In this paper we study the pushforward of finite rank matrix factorisations along ring morphisms. Even when $X$ is finite rank the pushforward is typically not a finitely generated $S$-module, but we can ask: when is $\varphi_*(X)$ homotopy equivalent to a finite rank matrix factorisation, and can such a finite model for the pushforward be described concretely? We refer to this as the \emph{construction of finite pushforwards.}

Before stating our results we want to mention some interesting examples of pushforwards. Our first motivation was to understand the composition of integral functors between categories of matrix factorisations; let us begin with a special case, which involves the Hirzebruch-Riemann-Roch theorem for matrix factorisations.
\vspace{0.2cm}

\textbf{Hirzebruch-Riemann-Roch.} Let $k$ be a field of characteristic zero, $W \in R = k\llbracket x_1,\ldots,x_n \rrbracket$ a polynomial and $X,Y$ two finite rank matrix factorisations of $W$ over $R$. The $R$-module $\Hom_R(X,Y)$ has a natural $\mathbb{Z}/2$-grading and a differential making it into a matrix factorisation of zero and if the zero locus of $W$ has an isolated singularity at the origin then $\Hom_R(X,Y)$ has finite-dimensional cohomology over $k$. One then defines the Euler characteristic
\[
\chi(\Hom_R(X,Y)) = \dim_k H^0 \Hom_R(X,Y) - \dim_k H^1 \Hom_R(X,Y).
\]
The Hirzebruch-Riemann-Roch theorem for matrix factorisations, recently proven by Polishchuk and Vaintrob \cite{polishchuk}, expresses this Euler characteristic in terms of the Chern characters of $X$ and $Y$ using the general Riemann-Roch theorem for dg-algebras of Shklyarov \cite{shklyarov}. Our point of view is that pushing $\Hom_R(X,Y)$ forward along $k \lto R$ yields an infinite rank matrix factorisation of zero over $k$, which is nonetheless homotopy equivalent to a finite rank matrix factorisation: namely, its cohomology viewed a finite-dimensional $\mathbb{Z}/2$-graded vector space.

The Polishchuk-Vaintrob theorem tells us something about this finite model for the pushforward, as it expresses the difference between the rank in even and odd degrees in terms of invariants of $X$ and $Y$, and this gives us some idea of what to expect from a general theorem on pushforwards.
\vspace{0.2cm}

\textbf{Convolution.} We can view the above example as a special case of the pushforward that arises when we compose integral functors between homotopy categories of matrix factorisations. Consider polynomial rings\footnote{One can of course work with integral functors defined over other rings, for example power series rings, but there are some subtleties we want to avoid here; see \S\ref{section:convolution_powerseries}.} $k[\bold{x}] = k[x_1,\ldots,x_n]$ and $k[\bold{y}] = k[y_1,\ldots,y_m]$ and $W \in k[\bold{x}], V \in k[\bold{y}]$. We view a matrix factorisation $E \in \hmf(k[\bold{x},\bold{y}], V- W)$ as the defining \emph{kernel} of an integral functor
\begin{gather*}
\Phi_E: \HMF(k[\bold{x}],W) \lto \HMF(k[\bold{y}],V),\\
\Phi_E(X) = \varphi_*( E \otimes_{k[\bold{x}]} X )
\end{gather*}
where $\varphi: k[\bold{y}] \lto k[\bold{x},\bold{y}]$ is the canonical ring morphism. Here the tensor product of two matrix factorisations \cite{yoshino98} is a matrix factorisation of the \emph{sum} of their potentials; in particular, $E \otimes_{k[\bold{x}]} X$ is a matrix factorisation of $V$ over $k[\bold{x},\bold{y}]$ and the definition of $\Phi_E$ makes sense.

Certainly a nonzero finite rank free module over $k[\bold{x},\bold{y}]$ is not of finite rank over $k[\bold{y}]$. Nonetheless, if the critical locus of $W$ in $\mathbb{A}^n_k$ is equal to the origin and $X$ is free of finite rank, then it is known that the pushforward $\Phi_E(X)$ is a direct summand in $\HMF(k[\bold{y}], V)$ of a finite rank matrix factorisation. The reader can find proofs in \cite[\S 4.2]{brunnerdefect}, \cite[Prop. 13]{khovanov} and Theorem \ref{theorem:defectfusion} below, but the statement is probably best understood in To\"en's framework of derived Morita theory of dg categories \cite{toen.morita} where it follows from the smoothness and properness statements established in \cite{dyck4} for dg categories of matrix factorisations. In this situation we address the problem of effectively calculating this summand, that is, the problem of understanding $\Phi_E$ as a functor between the idempotent completions of the subcategories of finite rank matrix factorisations.

More generally, if $k[\bold{z}] = k[z_1,\ldots,z_p]$ and we take a potential $U \in k[\bold{z}]$, we can consider a second kernel $F \in \hmf(k[\bold{y},\bold{z}], U - V)$. With $\psi: k[\bold{x},\bold{z}] \lto k[\bold{x},\bold{y},\bold{z}]$ canonical, the composite $\Phi_F \circ \Phi_E$ is the integral functor with kernel $F \star E = \psi_*( F \otimes_{k[\bold{y}]} E )$, as depicted in the diagram
\begin{equation}\label{eq:intro_pushforward_psi}
\xymatrix@C+1pc{
\HMF(k[\bold{x}],W) \ar@/_2pc/[rr]_{\Phi_{F \star E}}\ar[r]^{\Phi_E} & \HMF(k[\bold{y}],V) \ar[r]^{\Phi_F} & \HMF(k[\bold{z}],U)
}
\end{equation}
which commutes up to natural isomorphism. The infinite rank matrix factorisation $F \star E$ is called the \emph{convolution} of $E$ and $F$ and, if the critical locus of $V$ in $\mathbb{A}^m_k$ is equal to the origin, this convolution is a direct summand in $\HMF(k[\bold{x},\bold{z}], U - W)$ of a finite rank matrix factorisation. In order to understand the operation of convolution, we are therefore led to a study of pushforwards. 

There are various reasons why one might be interested in the convolution. To mention just two, convolution is the mathematical avatar of the fusion of defects in Landau-Ginzburg models, and it is also the basic ingredient in the construction of Khovanov and Rozansky's link homology \cite{khovanov}. We discuss these connections in more detail at the end of this introduction.
\vspace{0.2cm}

In the above examples the matrix factorisation to be pushed forward is acted on null-homotopically by a regular sequence in the ring of variables ``integrated out'' by the pushforward (the $\bold{y}$-variables, in the case of convolution) and there is a natural choice of null-homotopy contracting the action of each polynomial in this regular sequence. This motivates the setting for our main theorem.
\\

\textbf{Results.} Let $\varphi: S \lto R$ be a ring morphsim with $S$ a $\mathbb{Q}$-algebra, and let $W \in S$ be a potential. Writing $W$ for $\varphi(W)$ we let $(X,d_X)$ be a finite rank matrix factorisation of $W$ over $R$. Suppose that there is a quasi-regular sequence $\bold{t} = \{t_1,\ldots,t_n\}$ in $R$ with each $t_i$ acting null-homotopically on $X$, such that $R/\bold{t} R$ is a finite rank free $S$-module\footnote{A regular sequence is quasi-regular, but over arbitrary rings the latter condition is better behaved; see \S\ref{section:fundamentals}.}. With the natural $\mathbb{Z}/2$-grading and the differential inherited from $X$ the tensor product
\[
X/\bold{t} X = X \otimes_R R/\bold{t} R
\]
is a finite rank matrix factorisation of $W$ over $S$ and we prove that $\varphi_*(X)[n]$ is a direct summand of $X/\bold{t}X$ in the homotopy category of linear factorisations of $W$ over $S$. When $\hmf(S,W)$ is idempotent complete (e.g. $S = k\llbracket \bold{x}\rrbracket$ and $W$ is a polynomial with an isolated singularity at the origin) it follows that the pushforward is homotopy equivalent to a finite rank matrix factorisation. Our construction of this finite pushforward is not direct, in the sense that we do not give a formula for the differential on a finite rank matrix factorisation homotopy equivalent to $\varphi_*(X)$. Instead, we give an explicit idempotent on $X/\bold{t} X$ corresponding to the direct summand $\varphi_*(X)[n]$.

We construct this idempotent using a kind of differential calculus in the homotopy category of matrix factorisations, with the relative Atiyah class playing the role of the differential. To introduce the Atiyah class, let $S[\bold{t}] = S[t_1,\ldots,t_n]$ be the formal polynomial ring in the $t_i$ and make $R$ into a $S[\bold{t}]$-algebra in the obvious way. We view $X$ as a linear factorisation of $W$ over $S[\bold{t}]$ and, passing to the $\bold{t}R$-adic completion of $R$ if necessary (this does not affect the pushforward), we may assume that as an $S[\bold{t}]$-module $X$ admits a connection, i.e. a $S$-linear degree zero map of $\mathbb{Z}/2$-graded modules
\[
\nabla: X \lto X \otimes_{S[\bold{t}]} \Omega^1_{S[\bold{t}]/S}
\]
satisfying the Leibniz rule. We equip $X \otimes_{S[\bold{t}]} \Omega^1_{S[\bold{t}]/S}$ with the differential $d_X \otimes 1$. Although $\nabla$ is only $S$-linear the commutator $[d_X, \nabla]$ is $S[\bold{t}]$-linear, and it is easy to see that this is in fact a morphism of linear factorisations of $W$ over $S[\bold{t}]$
\[
\At_{S[\bold{t}]/S}(X) = [d_X, \nabla]: X \lto X \otimes_{S[\bold{t}]} \Omega^1_{S[\bold{t}]/S}[1].
\]
Up to homotopy this morphism is independent of the choice of connection, and we call it the \emph{Atiyah class} of $X$ relative to $S[\bold{t}]/S$. One should think of the Atiyah class as a differential $\ud_{S[\bold{t}]/S} X$. Indeed, if $R$ is free of finite rank over $S[\bold{t}]$ (which often happens in examples) then the Atiyah class is, as a matrix of $1$-forms, the entrywise differential $-\ud_{S[\bold{t}]/S}(d_X)$ of $d_X$ as a matrix over $S[\bold{t}]$.

Iterating the Atiyah class produces morphisms
\[
\At_{S[\bold{t}]/S}(X)^k: X \lto X \otimes_{S[\bold{t}]} \Omega^k_{S[\bold{t}]/S}[k],
\]
and if we trivialise $\Omega^n_{S[\bold{t}]/S}$ using the $n$-form $\ud t_1 \wedge \cdots \wedge \ud t_n$ then $\At_{S[\bold{t}]/S}(X)^n$ is a degree $n$ endomorphism of $X$, which we view as an $S$-linear endomorphism of $X/\bold{t} X$. By assumption the $t_i$ act null-homotopically on $X$, so there exists for each $i$ an odd $R$-linear map $\lambda_i$ with $\lambda_i d_X + d_X \lambda_i = t_i \cdot 1_X$. Each of these homotopies defines an odd endomorphism of $X/\bold{t} X$ over $S$, and so
\[
e = \frac{1}{n!} (-1)^{\binom{n+1}{2}}\lambda_1 \cdots \lambda_n \circ \At_{S[\bold{t}]/S}(X)^n
\]
is a well-defined $S$-linear endomorphism of the matrix factorisation $X/\bold{t} X$.

\begin{theoremn} The $n$-fold suspension of $\varphi_*(X)$ is a direct summand of $X/\bold{t} X$ in the homotopy category of linear factorisations of $W$ over $S$, and the associated idempotent is $e$. That is, there are morphisms in the homotopy category
\begin{equation}\label{eq:theorem_main_intro}
\xymatrix@+3pc{
X/\bold{t} X \ar@<-0.8ex>[r]_(0.5){\psi} & \varphi_*(X)[n] \ar@<-0.8ex>[l]_(0.5){\vartheta}
}
\end{equation}
such that $\psi \circ \vartheta = 1$ and $\vartheta \circ \psi = e$.
\end{theoremn}

When $\hmf(S,W)$ is idempotent complete $e$ can be split in this category to obtain a finite rank matrix factorisation homotopy equivalent to $\varphi_*(X)$. It is not always feasible to perform this splitting by hand, but it can be done on a computer. An implementation in Singular is described in \cite{carqmurf}, where the technqiue is used to compute examples of the link homology theory defined by Khovanov and Rozansky using matrix factorisations \cite{Khovanov07, khovanov}.

Derived categories of matrix factorisations have recently been introduced over non-affine schemes, and under some conditions the usual derived pushforward functor gives rise to a pushforward operation on matrix factorisations \cite{pv11}. In this paper we consider only the simplest kind of pushforward between affine schemes, but we expect that the techniques work more globally.

The construction of our idempotent is natural in the context of the theory of residues. To explain, let us assume that $R$ is separated and complete in its $\bold{t}R$-adic topology, so that there exists a flat $S$-linear connection $\nabla^0: R \lto R \otimes_{S[\bold{t}]} \Omega^1_{S[\bold{t}]/S}$ (such connections are constructed in Appendix \ref{section:derhamsplit}) which can be extended to an $S$-linear map $\nabla$ on $R \otimes_{S[\bold{t}]} \Omega_{S[\bold{t}]/S}$, where $\Omega_{S[\bold{t}]/S}$ is the exterior algebra $\wedge \Omega^1_{S[\bold{t}]/S}$. We can view elements $s,r_1,\ldots,r_n$ in $R$ as acting on $R \otimes_{S[\bold{t}]} \Omega_{S[\bold{t}]/S}$ by multiplication, so the commutators $[\nabla, r_i]$ are $S[\bold{t}]$-linear maps on $R \otimes_{S[\bold{t}]} \Omega_{S[\bold{t}]/S}$ and the product
\begin{equation}\label{eq:form_as_map}
s [\nabla, r_1] \cdots [\nabla, r_n]
\end{equation}
therefore induces an $S$-linear endomorphism of $R/\bold{t} R$. Since this module is finitely generated and free over $S$ the operator in (\ref{eq:form_as_map}) has a trace, and one defines the \emph{residue symbol} \cite{ResiduesDuality,Lipman87}
\begin{equation}\label{eq:defn_intro_ressym}
\Ress{R/S}\begin{bmatrix} s \cdot \ud r_1 \cdots \ud r_n \\ t_1, \ldots, t_n \end{bmatrix} = \tr_S\left( s [\nabla, r_1] \cdots [\nabla, r_n] \right) \in S,
\end{equation}
which agrees with the usual analytic residue symbol \cite[Chapter V]{Griffiths} when $S = \mathbb{C}$ and $R = \mathbb{C}[\bold{x}]$.

The Atiyah class is also defined in terms of a commutator with a connection, so it is not surprising to find that the supertrace of endomorphisms involving the Atiyah class can be expressed, via (\ref{eq:defn_intro_ressym}), in terms of residues. To see why this might be interesting, let us take $S = k\llbracket \bold{x}\rrbracket$ to be a power series ring over a field $k$ of characteristic zero, and suppose that $W$ is a polynomial whose zero locus has an isolated singularity at the origin.

The Chern character $\ch(Z)$ of $Z \in \hmf(S,W)$ is the image of the identity under the boundary-bulk map $\beta$ introduced in \cite{kapli} and formalised mathematically in \cite{Segal09, murfet-2009, polishchuk,dyckmurf}. In many ways this map $\beta$ behaves like a trace; for example, $\beta(fg) = \beta(gf)$. Ignoring suspensions, signs and some other details, (\ref{eq:theorem_main_intro}) leads to the following simple calculation of the Chern character of the pushforward
\begin{align*}
\ch(\varphi_*(X)) &= \beta( 1_{\varphi_*(X)} )\\
&= \beta( \psi \circ \vartheta )\\
&= \beta( \vartheta \circ \psi )\\
&= \beta( e )\\
&= \frac{1}{n!} \beta( \lambda_1 \cdots \lambda_n \circ \At_{S[\bold{t}]/S}(X)^n ).
\end{align*}
Since $\beta$ is defined using the supertrace, (\ref{eq:defn_intro_ressym}) allows us to rewrite $\ch(\varphi_*(X))$ in terms of residues, and in this way we recover the Hirzebruch-Riemann-Roch theorem of Polishchuk and Vaintrob (Section \ref{section:hrr}) and more generally we find a formula for the Chern character of the convolution of matrix factorisation kernels (Section \ref{section:fusingtop}). 
\vspace{0.2cm}

\textbf{Defects.} Matrix factorisations appear in string theory as boundary conditions in two-dimensional $\cat{N} = 2$ supersymmetric Landau-Ginzburg models \cite{Kapustin03,brunnerherbst,lazaroiu} with morphisms of matrix factorisations corresponding to boundary fields. This connection has stimulated much of the recent work on matrix factorisations, most notably by Orlov \cite{Orlov04,orlov-2005}. Beyond boundary conditions there is another natural structure in two-dimensional field theories, known as defects. To explain, recall that a two-dimensional topological quantum field theory assigns a value, the correlator, to oriented surfaces (possibly with boundary) called \emph{worldsheets}, which may be decorated with field insertions both in the interior and on components of the boundary labelled with boundary conditions. In the topologically B-twisted Landau-Ginzburg model with potential $W \in \mathbb{C}[x_1,\ldots,x_n]$, the correlator of a disc whose boundary is labelled by a finite rank matrix factorisation $(X,Q)$ of $W$, with insertions of a boundary field $\phi \in \underline{\End}(X)$ and bulk field $\psi \in \mathbb{C}[x_1,\ldots,x_n]$, is given by the Kapustin-Li formula \cite{kapli}
\begin{equation}\label{eq:picture_kaplicorrelator}
\left\langle
\begin{tikzpicture}[scale=0.6,baseline]
\draw[very thick] (0,0) circle (2cm);
\nicecolourscheme (0,0) circle (2cm);
\draw (0,0.5) node {{$\psi$}};
\draw (180:2.5cm) node {{$\phi$}};
\filldraw [black] (0,0) circle (3pt);
\filldraw [black] (180:2cm) circle (3pt);
\draw (2.5,0) node {{$X$}};
\end{tikzpicture}
\right\rangle = (-1)^{\binom{n}{2}} \Res \Bigg[ \begin{matrix} \str( \phi (\ud Q)^{\wedge n}) \psi \\ \partial_{x_1} W, \ldots, \partial_{x_n} W \end{matrix} \Bigg] \in \mathbb{C}.
\end{equation}
To introduce defects we take this piece of worldsheet and draw a smooth line separating it into two halves, on either side of which we consider field theories corresponding to potentials $W$ and $V$. The boundary is divided into two pieces, which we label by finite rank matrix factorisations $X$ of $W$ and $Y$ of $V$. The separating line, called a defect, is also given a label $D$, and with insertions of boundary fields $\phi \in \underline{\operatorname{End}}(X),\phi' \in \underline{\operatorname{End}}(Y)$ and additional fields $\varepsilon, \delta$ which we leave unexplained for now, the correlator should assign to this worldsheet a number
\begin{equation}\label{eq:correlator_kapli_withdefect}
\left\langle
\begin{tikzpicture}[scale=0.6,baseline]
\draw[very thick] (0,0) ellipse (4cm and 2cm);
\nicecolourscheme (0,0) ellipse (4cm and 2cm);
\draw (-4.5,0) node {{$\phi$}};
\filldraw [black] (-4,0) circle (3pt);

\draw (4.5,0) node {{$\phi'$}};
\filldraw [black] (4,0) circle (3pt);

\draw (-4,-1.5) node {{$X$}};
\draw (4,1.5) node {{$Y$}};

\filldraw [black] (-1.25,1.925) circle (3pt);
\draw (-1.25, 2.425) node {{$\varepsilon$}};

\filldraw [black] (1.25,-1.925) circle (3pt);
\filldraw [black] (1.25, -2.425) node {{$\delta$}};

\filldraw [white] (0.3, -1.4) node {{$W$}};
\filldraw [white] (1.9, -1.2) node {{$V$}};

\draw[very thick, in=90, out = 320] (0,0) to (1.25,-1.925);
\draw[->, very thick, in=140, out= 270] (-1.25,1.925) to (0,0);
\draw (0.5,0.5) node {{$D$}};
\end{tikzpicture}
\right\rangle \in \mathbb{C}.
\end{equation}
This correlator should obey certain consistency conditions, one of which says that the value should remain unchanged if we deform the path of the defect line on the worldsheet, as long as the defect line is not taken across field insertions or other defect lines. We can therefore bring the defect $D$ together with the part of the boundary labelled $X$, and in order for the correlator to make sense the resulting \emph{fusion} $D \star X$ should be an object of $\hmf(V)$ and $\varepsilon, \delta$ should be morphisms between $D \star X$ and $Y$.  In this way $D$ defines a function sending matrix factorisations of $W$ to matrix factorisations of $V$, and in fact this is a \emph{functor}, because boundary fields $\phi$ should determine boundary fields on the fusion. We already know how to describe functors $\hmf(W) \lto \hmf(V)$ in terms of objects of $\hmf(V - W)$, so it is natural to label defect lines with such kernels. The full justification of the representation of defects by kernels in the physics literature is via the folding trick; see \cite{wongaffleck,frohlichfuchs,brunnerdefect}. 

We are of course free to consider more complicated worldsheets; the boundary could be labelled by more than two boundary conditions and between these there may be inserted boundary condition changing operators (elements of $\uHom$ rather than $\underline{\End}$), there could be multiple defects separating more than two field theories, and the defect lines need not contact the boundary. A general formula for correlators (\ref{eq:correlator_kapli_withdefect}) in the setting of topological Landau-Ginzburg models defined on world-sheet foams was derived by Khovanov and Rozansky in \cite{Khovanov07}.

One structure that makes defects more interesting than boundary conditions is that the former have a natural operation of fusion
\begin{equation}
\left\langle
\begin{tikzpicture}[scale=0.6,baseline]
\clip (0,0) ellipse (4cm and 2cm);
\nicecolourschemedash (0,0) ellipse (4cm and 2cm);
\draw[->, very thick] (-1.25,1.925) to (-1.25,0);
\draw[very thick] (-1.25,0) to (-1.25,-1.925);
\draw (-2,0) node {{$D_2$}};

\draw[->, very thick] (1.25,1.925) to (1.25,0);
\draw[very thick] (1.25,0) to (1.25,-1.925);
\draw (0.5,0) node {{$D_1$}};
\end{tikzpicture}
\right\rangle
=
\left\langle
\begin{tikzpicture}[scale=0.6,baseline]
\clip (0,0) ellipse (4cm and 2cm);
\nicecolourschemedash (0,0) ellipse (4cm and 2cm);
\draw[->, very thick] (0,1.925) to (0,0);
\draw[very thick] (0,0) to (0,-1.925);
\draw (-1.5,0) node {{$D_1 \star D_2$}};
\end{tikzpicture}
\right\rangle
\end{equation}
where two defect lines $D_1, D_2$ on a worldsheet are moved together and in the limit form a new fused defect $D_1 \star D_2$ (in this picture we depict only part of the worldsheet). This naturally corresponds to composition of functors, or what is the same, convolution of integral kernels. Mathematically, 2D TQFTs with boundary conditions are well-understood in terms of Calabi-Yau categories, and in particular the correlator (\ref{eq:picture_kaplicorrelator}) is understood rigorously \cite{murfet-2009,dyckmurf}. However, when we allow our worldsheets to be decorated with defect lines the situation is much more complicated \cite{runkel09}. In order to make sense of, say, the topological Landau-Ginzburg model with defect lines as a functor, and to derive rigorously the correlators (\ref{eq:correlator_kapli_withdefect}), one should first understand the fusion of defects with boundary conditions and other defects. In simple examples one can do this directly \cite{brunnerdefect,engerreck} but in general this approach is not practical; resolving this problem was our motivation for studying finite models for pushforwards and convolutions in the present paper.

From the point of view of categorification, it is natural to think of worldsheets with defects in terms of planar algebra \cite{khov10} in the bicategory $\cat{LG}$ whose objects are potentials with isolated singularities, whose $1$-morphisms are matrix factorisations of differences of potentials, and whose $2$-morphisms are ordinary morphisms of matrix factorisations. Some of the detailed monoidal structure of this bicategory has been worked out in the physics literature \cite{nils09,nils10} in connection with the LG/CFT correspondence. A discussion of defects from the point of view of three-dimensional field theories can be found in \cite{kapustintft}.
\vspace{0.1cm}

\emph{Acknowledgements.} It is a pleasure to thank Nils Carqueville for introducing the authors to the world of defects and for numerous helpful discussions during the evolution of this project, Amnon Neeman for patiently pointing out many improvements that should be made, Dmytro Shklyarov for explaining his thesis \cite{shklyarov}, and Ed Segal, Ragnar Buchweitz, Igor Burban, Daniel Huybrechts for helpful discussions. The second author thanks the Hausdorff Center for Mathematics, where most of this work was carried out.

\section{Fundamentals}\label{section:fundamentals}

Throughout all rings are commutative. In this section $R$ denotes a ring. First we recall the basic definitions in the theory of matrix factorisations; for further background see \cite{yoshino,khovanov}.

\begin{definition} A \emph{linear factorisation} of $W \in R$ is a $\mathbb{Z}/2$-graded $R$-module $X = X^0 \oplus X^1$ together with an odd endomorphism
\[
d = \begin{pmatrix} 0 & d^1 \\ d^0 & 0 \end{pmatrix}: X \lto X
\]
with the property that $d \circ d = W \cdot 1_X$. Given two linear factorisations $X,Y$ of $W$ the $R$-module $\Hom_R(X,Y)$ is a linear factorisation of zero with differential $d( \alpha ) = d_Y \circ \alpha - (-1)^{|\alpha|} \alpha \circ d_X$. A linear factorisation of zero is the same thing as a $\mathbb{Z}/2$-graded complex. A \emph{morphism} $X \lto Y$ is a cocycle in degree zero of $\Hom_R(X,Y)$, and a morphism is \emph{null-homotopic} if it is a coboundary. We write $\hf(W)$ for the homotopy category of factorisations of $W$, or $\hf(R,W)$ if we want to be specific about the ring, with morphism spaces denoted
\[
\uHom(X,Y) = H^0 \Hom_R(X,Y).
\]
A \emph{matrix factorisation} of $W$ is a linear factorisation whose underlying modules $X^0, X^1$ are free, and we say that $X$ is \emph{finite rank} if further $X$ is finitely generated. We denote by $\HMF(W)$ or $\HMF(R,W)$ the homotopy category of all matrix factorisations, and by $\hmf(W)$ or $\hmf(R,W)$ the full subcategory of finite rank matrix factorisations. These are all triangulated categories.
\end{definition}

\begin{remark}\label{remark:partial_gives_homotopy} Let $k$ be a field and suppose $R = k\llbracket x_1,\ldots,x_n\rrbracket$. Let $(X,d_X)$ be a finite rank matrix factorisation of $W \in R$ and choose a homogeneous basis for $X$ in order to write $d_X$ as a matrix. Applying the derivation $\partial_{x_i} = \partial/\partial x_i$ to the identity $d_X^2 = W \cdot 1_X$ we obtain a null-homotopy
\[
\partial_{x_i}(d_X) \cdot d_X + d_X \cdot \partial_{x_i}(d_X) = \partial_{x_i} W \cdot 1_X
\]
of the action of $\partial_{x_i} W$ on $X$.
\end{remark}

\begin{definition}\label{defn:koszul_signs} If $X, Y$ are $\mathbb{Z}$ or $\mathbb{Z}/2$-graded $R$-modules and $f, g: X \lto Y$ are homogeneous maps, then $f \otimes g$ is the map defined on homogeneous tensors by $x \otimes y \mapsto (-1)^{|x|||g|} f(x) \otimes g(y)$.
\end{definition}

\begin{definition} Given potentials $W,V \in R$ and linear factorisations $X \in \hf(R,W), Y \in \hf(R,V)$, the tensor product $X \otimes_R Y$ acquires a natural $\mathbb{Z}/2$-grading and a differential $d_{X \otimes Y} = d_X \otimes 1 + 1 \otimes d_Y$ making it into a linear factorisation of $W + V$. For more on this construction, see \cite{yoshino98}.
\end{definition}

We will always use the notation $- \otimes -$ to mean this $\mathbb{Z}/2$-graded tensor product of factorisations. If one of the factors in the tensor product is simply an $R$-module, or more generally a complex of $R$-modules, then we view it as a factorisation of zero using the $\mathbb{Z}/2$-folding:

\begin{definition}[(Folding)] Given any complex $P$ of $R$-modules we denote by $P_{\mathbb{Z}/2}$ the \emph{$\mathbb{Z}/2$-folding}, which has $\oplus_{i \in 2\mathbb{Z}} P^i$ in degree zero and $\oplus_{i \in 2\mathbb{Z} + 1} P^i$ in degree one, together with the obvious differential. We view $P_{\mathbb{Z}/2}$ as a linear factorisation of zero over $R$.
\end{definition}

In particular when we write $X \otimes \skos(\bold{t})$ for a Koszul complex $\skos(\bold{t})$, we actually mean $X \otimes \skos(\bold{t})_{\mathbb{Z}/2}$.

\begin{definition} Given $X \in \hmf(R,W)$ we define the \emph{dual} $X^{\lor}$ to be the matrix factorisation with even component $(X^0)^* = \Hom_R(X^0,R)$, odd component $(X^1)^*$ and differential
\[
d_{X^{\lor}} = \begin{pmatrix} 0 & (d^0_X)^* \\ -(d^1_X)^* & 0 \end{pmatrix}.
\]
Given $Y \in \hf(R,W)$ there is a canonical isomorphism $X^{\lor} \otimes Y \cong \shom(X,Y)$ of linear factorisations of zero, natural in both variables, which sends a homogeneous tensor $\nu \otimes y$ to $x \mapsto (-1)^{|\nu||y|} \nu(x) \cdot y$.
\end{definition}

\begin{definition} If $X$ is a $\mathbb{Z}/2$-graded finitely generated projective $R$-module and $\alpha: X \lto~X$ a homogeneous $R$-linear map, the \emph{supertrace} of $\alpha$ is the ordinary trace of $G \alpha$, where $G$ is the grading operator $G(x) = (-1)^{|x|} x$.
\end{definition}

\begin{definition}\label{defn:connections} Let $k$ be a ring, $A$ a $k$-algebra and set $\Omega_{A/k} = \wedge \Omega^1_{A/k}$. If $\ud^0: A \lto \Omega^1_{A/k}$ denotes the K\"ahler differential then there is a unique graded $k$-derivation $\ud: \Omega_{A/k} \lto \Omega_{A/k}$ of degree one such that $\ud^2 = 0$ and $\ud|_A = \ud^0$, and this map is called the \emph{de Rham differential}. A $k$-linear \emph{connection} on an $A$-module $M$ is a $k$-linear map
\[
\nabla^0: M \lto M \otimes_A \Omega^1_{A/k}
\]
satisfying the Leibniz rule
\[
\nabla^0(a m) = a \nabla^0(m) + m \otimes \ud^0(a), \quad \forall a \in A, m \in M.
\]
Any such connection extends uniquely to a $k$-linear map $\nabla$ of degree one on $M \otimes_A \Omega_{A/k}$ satisfying $\nabla(x \omega) = (\nabla x) \omega + (-1)^p x (\ud \omega)$ for $x \in M \otimes_A \Omega^p_{A/k}$ and $\omega \in \Omega_{A/k}$ and $\nabla^0$ is said to be \emph{flat} if $\nabla^2 = 0$. The K\" ahler differential gives a flat connection on $A$, and this induces a flat connection on any free $A$-module. Splitting idempotents, we get a connection on any projective $A$-module (not necessarily a \emph{flat} connection). For details we refer the reader to \cite[\S 2.10]{bourbaki} and \cite[Ch.~8]{loday}
\end{definition}

We will make extensive use of quasi-regular sequences; see \cite[Chapitre $0$ \S 15.1]{EGA4} or \cite{Matsumura}. Every regular sequence $t_1, \ldots, t_n$ in a ring $R$ is quasi-regular, and the converse holds if $R$ is separated and complete in the $I = (t_1,\ldots,t_n)$-adic topology. If $R$ is noetherian then $\bold{t}$ is quasi-regular if and only if the Koszul complex over $R$ of $\bold{t}$ is exact except in degree zero.

\section{Outline}\label{section:integrating}

Let $\varphi: S \lto R$ be a morphism of rings, $W \in S$ a potential, and $X$ a finite rank matrix factorisation of $W$ over $R$. Giving $\mathbb{Z}/2$-graded $R$-modules the $S$-module structure determined by the morphism $\varphi$ defines the pushforward (or restriction of scalars) functor
\[
\varphi_*: \hf(R, W) \lto \hf(S, W).
\]
Throughout we will often write $X$ for $\varphi_*(X)$ and leave the pushforward implicit. If $R$ is a finite rank free $S$-module then $\varphi_*$ sends finite rank matrix factorisations to finite rank matrix factorisations. In general, we look for factorisations $X$ on which $R$ acts up to homotopy like a finitely generated projective $S$-module, in the sense that there is a regular sequence $\bold{t} = \{t_1,\ldots,t_n\}$ in $R$ such that $R/\bold{t} R$ is a finitely generated projective $S$-module, and each $t_j \cdot 1_X$ is null-homotopic. Using tools of local cohomology and perturbation techniques we relate $X$ to a factorisation defined over $R/\bold{t} R$, which can be pushed down to $S$.

In order to explain the intuition we introduce the stable Koszul complex. Our convention is that the Koszul complex $\skos(t)$ on an element $t \in R$ lives in degrees $-1$ and $0$. Given $t \in R$ there is a map of Koszul complexes $\skos(t) \lto \skos(t^2)$,
\begin{equation}\label{eq:koszul_map}
\xymatrix{
R \ar[r]^1\ar[d]_t & R\ar[d]^{t^2}\\
R \ar[r]_t & R
}
\end{equation}
Set $\bold{t}^i = \{ t_1^i, \ldots, t_n^i \}$. Tensoring together maps of the form (\ref{eq:koszul_map}) one defines a direct system
\begin{equation}\label{eq:koszul_sequence}
\skos(\bold{t}) \lto \skos(\bold{t}^2) \lto \skos(\bold{t}^3) \lto \cdots
\end{equation}
of Koszul complexes, and the direct limit is called the \emph{stable Koszul complex} $\skos_\infty(\bold{t})$. This complex is used to define the local cohomology functor $\rdev \Gamma_Z( Y ) = Y \otimes \skos_\infty(\bold{t})$. Since the direct limit of (\ref{eq:koszul_sequence}) is a direct sum of limits of sequences of the form
\[
R \xlto{t} R \xlto{t} R \xlto{t} R \lto \cdots
\]
which has limit $R[t^{-1}]$, $\skos_\infty(\bold{t})$ is isomorphic to the tensor product $\bigotimes_{j=1}^n(R \lto R[t_j^{-1}])$. Observe that the canonical morphism $\varepsilon_\infty: \skos_\infty(\bold{t}) \lto R[n]$ has a mapping cone which is an extension of localisations of $R$ to open subsets of the complement of $Z$.

Let $X$ be a finite rank matrix factorisation of $W$ over $R$, with each endomorphism $t_j \cdot 1_X$ null-homotopic. Combining this hypothesis with the previous observation, we see that the tensor product of $X$ with the cone of the map $\varepsilon_\infty$ is contractible. Hence we have a homotopy equivalence
\[
X \otimes \skos_\infty(\bold{t}) \lto X[n].
\]
In the homotopy category we therefore have
\begin{align*}
\uHom(X[n], X[n]) &\cong \uHom(X[n], X \otimes \skos_\infty(\bold{t})) \cong \varinjlim_i \uHom(X[n], X \otimes \skos(\bold{t}^i)),
\end{align*}
so the identity on $X[n]$ must factor through some $X \otimes \skos(\bold{t}^i)$. Indeed, since $t_j \cdot 1_X = 0$ the tensor product $X \otimes \skos(\bold{t})$ is homotopy equivalent to a direct sum of $X$ and its suspension $X[1]$, and in Section \ref{section:koszulcpxs} we give a particular embedding of $X[n]$ into $X \otimes \skos(\bold{t})$. In Section \ref{section:perturbation} we give some background on the perturbation lemma, which we use in Sections \ref{section:homotopies_to_morphisms} and \ref{section:pertfrompoint} to prove that under some mild hypotheses the map induced by the augmentation from the Koszul complex to its cohomology
\[
X \otimes \skos(\bold{t}) \lto X \otimes R/\bold{t} R = X/\bold{t} X
\]
is a homotopy equivalence over $S$. At this point we will have proven that $X$ is a direct summand of $X/\bold{t} X$ in $\hf(S,W)$, but in order to have a satisfactory description of the corresponding idempotent $e$ on $X/\bold{t} X$ we need to recall the theory of residues and traces (Section \ref{section:residuesandtraces}) and Atiyah classes (Section \ref{section:atiyah_class}). Then finally we prove the main theorem, as stated in the introduction, in Section \ref{section:idempotentformula}.

Then we turn to applications. In Section \ref{section:chern_char_over} we derive some residue formulas for Chern characters, in Section \ref{section:fusingtop} we study the convolution of kernels, and in Section \ref{section:knorrer} we compute an example of the idempotent appearing in Kn\"orrer periodicity.

\section{Koszul complexes on null-homotopic maps}\label{section:koszulcpxs}

Let $R$ be a ring and $W \in R$ a potential. Let $(X,d)$ be a linear factorisation of $W$ and suppose we are given a sequence $\bold{t} = \{ t_1,\ldots, t_n \}$ of elements of $R$ which acts null-homotopically on $X$, and for $1 \le i \le n$ let $\lambda_i$ be a degree one $R$-linear map on $X$ with
\[
\lambda_i \cdot d + d \cdot \lambda_i = t_i \cdot 1_X.
\]
Let $\skos(t_i)$ denote the cone of the morphism $R \xlto{t_i} R$ of complexes concentrated in degree zero. In the triangulated category of linear factorisations there is a triangle (tensor products are $R$-linear)
\[
X \xlto{t_i} X \xlto{\theta} X \otimes \skos(t_i) \lto X[1].
\]
Since $t_i \cdot 1_X$ is null-homotopic, $\theta$ is actually a split monomorphism (up to homotopy), and $X \otimes \skos(t_i)$ is homotopy equivalent to $X \oplus X[1]$. Writing $\skos(\bold{t}) = \skos(t_1) \otimes \cdots \otimes \skos(t_n)$ and iterating this argument we find that $X \otimes \skos(\bold{t})$ is homotopy equivalent to $(X \oplus X[1])^{\oplus 2^{n-1}}$. Of all the ways of embedding $X$ and $X[1]$ as direct summands in this factorisation, we want to highlight one in particular.

The model for the Koszul complex $\skos(\bold{t})$ that we use is the following: let $F$ be the free $R$-module on the formal symbols $\ud t_1, \ldots, \ud t_n$ and make $\wedge F$ a graded $R$-algebra by setting $|\ud t_i| = -1$. We equip $\wedge F$ with the differential $\delta = (\sum_{j=1}^n t_i (\ud t_i)^*) \,\neg\, (-)$ and set $\skos(\bold{t}) = (\wedge F, \delta)$. Then $X \otimes \skos(\bold{t})$ is a linear factorisation of $W$ with differential $d + \delta = d \otimes 1 + 1 \otimes \delta$. Let $\varepsilon: X \otimes \skos(\bold{t}) \lto X[n]$ be the morphism which projects onto the degree $-n$ component of $\skos(\bold{t})$, with a sign:
\begin{equation}\label{eq:augmentation_epsilon}
\varepsilon( x \cdot \ud t_1 \cdots \ud t_n ) = (-1)^{n|x|} x.
\end{equation}
This sign is ``correct'' in the sense that it arises from $X \otimes (R[n]) \cong X[n]$ and the sign that accompanies the removal of a suspension from the second variable of a tensor product. Multiplying $X \otimes \skos(\bold{t})$ on the left by $\ud t_i$ also involves a Koszul sign
\begin{gather*}
\ud t_i \wedge -: X \otimes \wedge F \lto X \otimes \wedge F,\\
x \otimes \omega \mapsto (-1)^{|x|} x \otimes \ud t_i \wedge \omega
\end{gather*}
for homogeneous $x \in X$ and $\omega \in \wedge F$. We will write $\ud t_i$ for this map, and $\lambda_i$ for $\lambda_i \otimes 1$, and set
\[
\vartheta' = (\ud t_1 - \lambda_1) \circ \cdots \circ (\ud t_n - \lambda_n).
\]
We claim that this gives a right inverse to the projection $\varepsilon$.

\begin{lemma} $\vartheta'$ is a morphism of linear factorisations $X[n] \lto X \otimes \skos(\bold{t})$, and $\varepsilon \circ \vartheta' = 1$.
\end{lemma}
\begin{proof}
It suffices to prove that $d + \delta$ and $\ud t_i - \lambda_i$ anticommute. Since $\delta$ is a graded derivation on $\skos(\bold{t})$ we have, for $\omega$ homogeneous, $\delta \ud t_i(\omega) = \delta( \ud t_i \wedge \omega ) = t_i \omega - \ud t_i \wedge \delta(\omega)$. That is, $\delta \ud t_i + \ud t_i \delta = t_i$. Since $\delta$ anticommutes with $\ud t_i$ and $\lambda_i$ the first claim follows. It is clear that $\varepsilon \circ \vartheta' = 1$.
\end{proof}

Let us explain why the definition of the splitting $\vartheta'$ is a natural one. We may think of the Koszul complex as the tensor product of
complexes
\[
\skos(\bold{t}) = \bigotimes_{i=1}^n \left( R \ud t_i \lto R \right) \text{.}
\]
Similarly, we can describe the stable Koszul complex as
\[
\skos_\infty(\bold{t}) = \bigotimes_{i=1}^n \left( R \ud t_i \lto R[t_i^{-1}] \right)
\]
where the differential is given by 
$\delta' = (\sum_{j=1}^n \ud t_i^*) \,\neg\, (-)$.
As explained above, the complex $\skos_\infty(\bold{t})$ can be obtained as a colimit over
the Koszul complexes $\skos(\bold{t}^i)$ and hence we have a natural map
$\iota:\; \skos(\bold{t}) \to \skos_\infty(\bold{t})$
which on the $i$-th component of the tensor products is given by
\[
\xymatrix{
R \ud t_i \ar[r]^1 \ar[d] & R \ud t_i \ar[d] \\
R \ar[r]^{t_i^{-1}} & R[{t_i}^{-1}] \text{.}
}
\]
For a moment we restrict attention to one variable $t_i$. We claim that the canonical map 
\[
X \otimes \skos_\infty(t_i) \to X[1]
\]
is a homotopy equivalence with homotopy inverse given by $\alpha_i = \ud t_i -
\frac{\lambda_i}{t_i}$ with homotopy $\frac{\lambda_i}{t_i}$. This follows from an application of the homological
perturbation lemma (see Lemma \ref{prop:perturb_to_morphism} below), and can of course be verified
explicitly by direct calculation.

In fact this map $\alpha_i$ factors explicitly via the map $\beta_i = \ud t_i - \lambda_i$ to yield:
\[
\xymatrix{
X \otimes \skos(t_i) \ar[r] &  X \otimes \skos_\infty(t_i)\\
& X[1]. \ar[u]_{\alpha_i} \ar[ul]^{\beta_i}
}
\]
Applying this argument iteratively to all variables, we obtain the commutative
diagram
\[
\xymatrix{
X \otimes \skos(\bold{t}) \ar[r]^{\iota} &  X \otimes \skos_\infty(\bold{t})\\
& X[n] \ar[u]_{\alpha} \ar[ul]^{\vartheta'}
}
\]
where $\vartheta'$ is as above, and $\alpha$ is inverse to the canonical map $\varepsilon_{\infty}: X \otimes \skos_\infty(\bold{t}) \to X[n]$, which composes with $\iota$ to give $\varepsilon$. We see again that $\varepsilon \circ \vartheta' = \varepsilon_{\infty} \circ \iota \circ \vartheta' = \varepsilon_{\infty} \circ \alpha = 1$.

We refer to the canonical morphism of complexes $\skos(\bold{t}) \lto R/\bold{t} R$
as the augmentation. Tensoring with $X$ yields a morphism of
linear factorisations
\[
\pi:\; X \otimes \skos(\bold{t}) \lto X \otimes R/\bold{t} R =
X/\bold{t} X \text{.}
\]
We write $\vartheta$ for the composite $\pi \circ \vartheta'$, and it is clear that $\vartheta = (-1)^n \lambda_1 \circ \cdots \circ \lambda_n$.

Assuming that $S \lto R$ is as in the introduction with $R/\bold{t} R$ a
finite free $S$-module, the map $\pi$ will, under mild assumptions, turn out to be an isomorphism in
$\hf(S,W)$. In this case the
following commutative diagram summarises the contents of the foregoing discussion:

\[
\xymatrix@C+3pc{
X/\bold{t} X \ar@<-0.8ex>[r]_{\pi^{-1}} \ar@/_5pc/[rr]_{\psi = \varepsilon \circ
\pi^{-1}}& 
X \otimes K(\bold{t}) \ar[d]_{\iota} \ar@<-0.8ex>[l]_{\pi} \ar@<-0.8ex>[r]_{\varepsilon}
& X[n]. \ar@<-0.8ex>[l]_{\vartheta'} \ar@/_3pc/[ll]_{\vartheta = \pi \circ \vartheta'}\\
 & X \otimes K_\infty(\bold{t}) \ar[ur]_{\simeq} & 
}
\]

By construction we have $\psi \circ \vartheta = 1$ in $\HMF(S,W)$ and
the idempotent $e = \vartheta \circ \psi$ therefore realises $X[n]$ as a summand of the
finite rank matrix factorisation $X/\bold{t}X$. In order to explicitly describe the idempotent $e$ we have to calculate $\pi^{-1}$,
which we will do in the subsequent sections.

\section{Perturbation}\label{section:perturbation}

Throughout this section $R$ is a ring, $W$ is an element of $R$, and all linear factorisations are over $R$.

\begin{definition}\label{defn:2perp} A \emph{deformation retract datum} of linear factorisations of $W$ consists of
\[
\xymatrix@C+2pc{
(L,b) \ar@<-0.8ex>[r]_i & (M,b), \ar@<-0.8ex>[l]_p
} \quad h
\]
where $(L,b)$ and $(M,b)$ are linear factorisations of $W$, $p$ and $i$ are morphisms of factorisations, and $h$ is a degree one map $M \lto M$, which together satisfy the following two conditions:
\begin{itemize}
\item[(i)] $pi = 1$,
\item[(ii)] $ip = 1 + bh + hb$.
\end{itemize}
Notice that in particular $p$ is a homotopy equivalence with inverse $i$.
\end{definition}

\begin{remark} A \emph{deformation retract datum} of ordinary (i.e. $\mathbb{Z}$-graded) complexes over some ring is defined in the same way, with all morphisms and homotopies $\mathbb{Z}$-graded. Note that if we have such a $\mathbb{Z}$-graded deformation retract datum then the $\mathbb{Z}/2$-folding is a deformation retract datum of linear factorisations of zero.
\end{remark}

Suppose we are given a deformation retract datum. We want to perturb the differential on $M$ and see if we can induce a perturbation of the differential on $L$, together with a homotopy equivalence between these two perturbed factorisations. We say that a degree one $R$-linear map $\mu: M \lto M$ is a \emph{small perturbation} if $(\mu h)^n = 0$ for all sufficiently large integers $n$. In this case
\[
(1 - \mu h)\sum_{n \ge 0} (\mu h)^n = 1
\]
so $1-\mu h$ is an isomorphism of $\mathbb{Z}/2$-graded $R$-modules, and we set
\[
A = (1 - \mu h)^{-1} \mu = \sum_{n \ge 0} (\mu h)^n \mu.
\]
Consider the following collection of data:
\begin{equation}\label{eq:new_perturb_fac}
\xymatrix@C+2pc{
(L,b_\infty) \ar@<-0.8ex>[r]_{i_\infty} & (M,b + \mu), \ar@<-0.8ex>[l]_{p_\infty}
} \quad h_\infty
\end{equation}
where
\begin{equation}\label{eq:new_perturb_fac2}
i_\infty = i + hAi, \quad p_\infty = p + pAh, \quad h_\infty = h + hAh, \quad b_\infty = b + pAi.
\end{equation}
We are going to give conditions under which (\ref{eq:new_perturb_fac}) is again a deformation retract datum, possibly for some different potential $W' \in R$.

\begin{theorem}\label{theorem:2perp} Given $W' \in R$ let $\mu$ be a small perturbation such that $(b + \mu)^2 = W' \cdot 1_M$. If either of the following conditions hold:
\begin{itemize}
\item[(1)] $p \mu = 0, ph = 0$ and we assume $W' = W$, or
\item[(2)] $h i = 0, ph = 0, h^2 = 0$ and we allow $W' \neq W$
\end{itemize}
then (\ref{eq:new_perturb_fac}) is a deformation retract datum of linear factorisations of $W'$.
\end{theorem}
\begin{proof}
This is the $\mathbb{Z}/2$-graded variant of the perturbation lemma given in \cite[Theorem 2.3]{crainic}.
\end{proof}

\begin{remark} In type $(1)$ it follows from the assumptions that $p_\infty = p$ and $b_\infty = b$.
\end{remark}

\section{From homotopies to morphisms}\label{section:homotopies_to_morphisms}

Let $\varphi: S \lto R$ be a ring morphism, $W \in S$ a potential and $(X,d)$ a finite rank matrix factorisation of $W$ over $R$. Fix a homogeneous $R$-basis $\{ \xi_i \}_{i \in I}$ for $X$, and suppose that we are given a deformation retract datum of linear factorisations of zero over $S$
\begin{equation}\label{eq:pre_perturb_htm2}
\xymatrix@C+2pc{
(L,b) \ar@<-0.8ex>[r]_\sigma & (M,b), \ar@<-0.8ex>[l]_p
} \quad h.
\end{equation}
We assume that $(L,b)$ and $(M,b)$ are linear factorisations of zero over $R$ and that $p$ is $R$-linear, but we \emph{do not} assume that either $h$ or $\sigma$ is $R$-linear. We can tensor with $X$, considered as a $\mathbb{Z}/2$-graded $R$-module with no differential, to obtain a diagram of linear factorisations of zero over $S$ (all tensor products are $\mathbb{Z}/2$-graded and over $R$)
\begin{equation}\label{eq:pre_perturb_htm}
\xymatrix@C+2pc{
(X \otimes L,1 \otimes b) \ar@<-0.8ex>[r]_{1 \otimes \sigma} & (X \otimes M,1 \otimes b), \ar@<-0.8ex>[l]_{1 \otimes p}
} \quad 1 \otimes h.
\end{equation}
The $S$-linear map $1 \otimes h$ is defined using our chosen basis for $X$ (and the Koszul sign convention) by
\begin{equation}\label{eq:pre_perturb_htm11}
1 \otimes h: X \otimes M \lto X \otimes M, \qquad \xi_i \otimes m \mapsto (-1)^{|\xi_i|} \xi_i \otimes h(m).
\end{equation}
The map $1 \otimes \sigma$ is defined similarly, but without signs; where it will not cause confusion we will also write $h$ and $\sigma$ for these extensions. With all this in mind, it is clear that (\ref{eq:pre_perturb_htm}) is a deformation retract datum of linear factorisations of zero over $S$.

We will now insert as a perturbation the differential $d$ on $X$.

\begin{proposition}\label{prop:perturb_to_morphism} Suppose that the perturbation $\mu = d \otimes 1$ on $X \otimes M$ in the original deformation retract datum (\ref{eq:pre_perturb_htm2}) is small, and that we have $h^2 = 0, hi = 0$ and $ph = 0$. Then there is a deformation retract datum of linear factorisations of $W$ over $S$
\begin{equation}\label{eq:pre_perturb_htm3}
\xymatrix@C+2pc{
(X \otimes L, d \otimes 1 + 1 \otimes b) \ar@<-0.8ex>[r]_{\sigma_\infty} & (X \otimes M,d \otimes 1 + 1 \otimes b), \ar@<-0.8ex>[l]_{1 \otimes p}
} \quad h_\infty
\end{equation}
where 
\begin{align*}
\sigma_\infty = \sum_{m \ge 0} ( h d )^{m} \sigma,\qquad
h_\infty = \sum_{m \ge 0} ( h d )^{m} h.
\end{align*}
\end{proposition}
\begin{proof}
It is clear from the hypotheses that $(1 \otimes h)^2 = 0, (1 \otimes h)(1 \otimes \sigma) = 0$ and $(1 \otimes p)(1 \otimes h) = 0$, and also that $(d \otimes 1 + 1 \otimes b)^2 = d^2 \otimes 1 = W \cdot 1_{X \otimes M}$ (here we use that $b$ is $R$-linear). It follows from Theorem \ref{theorem:2perp}(2) that we have a deformation retract datum of linear factorisations of $W$ over $S$
\begin{equation}\label{eq:pre_perturb_htm12}
\xymatrix@C+2pc{
(X \otimes L, b_\infty) \ar@<-0.8ex>[r]_(0.4){\sigma_\infty} & (X \otimes M,d \otimes 1 + 1 \otimes b), \ar@<-0.8ex>[l]_(0.6){p_\infty}
} \quad h_\infty
\end{equation}
where
\begin{align*}
A &= \sum_{m \ge 0} ((d \otimes 1)(1 \otimes h))^m (d \otimes 1),\\
\sigma_\infty &= 1 \otimes \sigma + (1 \otimes h) A (1 \otimes \sigma),\\
p_\infty &= 1 \otimes p + (1 \otimes p)A(1 \otimes h),\\
h_\infty &= 1 \otimes h + (1 \otimes h)A(1 \otimes h),\\
b_\infty &= 1 \otimes d + (1 \otimes p)A(1 \otimes \sigma).
\end{align*}
But we notice that because $p$ is $R$-linear, $(1 \otimes p)(d \otimes 1) = (d \otimes 1)(1 \otimes p)$ and since $(1 \otimes p)(1 \otimes h) = 0$ we find that $(1 \otimes p)A(1 \otimes \sigma) = d \otimes 1$ hence $b_\infty = d \otimes 1 + 1 \otimes b$. Similarly one checks that $p_\infty = p$ and some simple manipulations change $i_\infty, h_\infty$ into the desired form, completing the proof.
\end{proof}

\section{The idempotent}\label{section:pertfrompoint}

Let $\varphi: S \lto R$ be a ring morphism, $W \in S$ a potential, and $(X,d)$
a finite rank matrix factorisation of $W$ over $R$. Let $\bold{t} = \{t_1,\ldots,t_n\}$ be a quasi-regular sequence in $R$ such that $t_j \cdot 1_X$ is null-homotopic for $1 \le j \le n$. We assume that there is a deformation retract datum of $\mathbb{Z}$-graded complexes over $S$
\begin{equation}\label{eq:koszul_perturb1}
\xymatrix@C+2pc{
(R/\bold{t} R,0) \ar@<-0.8ex>[r]_\sigma & (\skos(\bold{t}), \delta), \ar@<-0.8ex>[l]_\pi
} \quad h
\end{equation}
satisfying the identities $h^2 = 0, h\sigma = 0$ and $\pi h = 0$. We emphasise that while $\pi$ is $R$-linear, the maps $h$ and $\sigma$ are only $S$-linear. The Koszul complex $\skos(\bold{t})$ is defined as in Section \ref{section:koszulcpxs} and $\pi: \skos(\bold{t}) \lto R/\bold{t} R$ is the augmentation. If $R$ is noetherian then $\pi$ is a quasi-isomorphism, and a deformation retract datum (\ref{eq:koszul_perturb1}) satisfying these identities will exist, for example, if $R$ and $R/\bold{t} R$ are projective over $S$ since in this case $\pi$ is a homotopy equivalence over $S$.

We prove that $X[n]$ is a direct summand of $X/\bold{t} X$ in the homotopy category of linear factorisations $\hf(S,W)$, and describe the corresponding idempotent. Consider the augmentation
\[
1 \otimes \pi: X \otimes \skos(\bold{t}) \lto X \otimes R/\bold{t} R = X/\bold{t} X
\]
where as usual we $\mathbb{Z}/2$-fold the second tensor component, and all tensor products are over $R$. We prove next that $1 \otimes \pi$ is a homotopy equivalence of factorisations of $W$ over $S$. The point being that by hypothesis $\pi$ is a $S$-linear homotopy equivalence, and we show that by perturbation any choice of homotopy $h$ in (\ref{eq:koszul_perturb1}) gives rise to an $S$-linear homotopy inverse to $1 \otimes \pi$. Since it is not likely to cause confusion, from now on we simply write $\pi$ for $1 \otimes \pi$.

We fix a homogeneous $R$-basis $\{ \xi_i \}_{i \in I}$ for $X$, used implicitly whenever we introduce an $S$-linear map, say $1 \otimes h$ on $X \otimes \skos(\bold{t})$, where $h$ is only $S$-linear, see (\ref{eq:pre_perturb_htm11}). 

\begin{proposition}\label{prop:fromhomotopytomap} The morphism $\pi: X \otimes \skos(\bold{t}) \lto X/\bold{t} X$ is a homotopy equivalence, and there is a deformation retract datum of linear factorisations of $W$ over $S$
\begin{equation}\label{eq:koszul_perturb2}
\xymatrix@C+2pc{
(X/\bold{t} X,d) \ar@<-0.8ex>[r]_(0.45){\sigma_\infty} & (X \otimes \skos(\bold{t}), d + \delta), \ar@<-0.8ex>[l]_(0.55)\pi
} \quad h_\infty
\end{equation}
where
\begin{align*}
\sigma_\infty = \sum_{m \ge 0} ( h d )^{m} \sigma,\qquad
h_\infty = \sum_{m \ge 0} ( h d )^{m} h.
\end{align*}
\end{proposition}
\begin{proof}
This follows from the perturbation argument of Proposition \ref{prop:perturb_to_morphism}.
\end{proof}

We are mainly interested in the composite of the homotopy inverse $\sigma_\infty$ of $\pi$ with the projection $\varepsilon: X \otimes \skos(\bold{t}) \lto X[n]$ discussed in Section \ref{section:koszulcpxs}. 

\begin{definition} Let $\psi$ denote the following composite in $\hf(S,W)$
\[
\psi: X/\bold{t} X \xlto{\pi^{-1}} X \otimes \skos(\bold{t}) \xlto{\varepsilon} X[n]
\]
of the homotopy inverse of $\pi$ with the projection $\varepsilon$. For any choice of deformation retract (\ref{eq:koszul_perturb1}) the proposition gives us an explicit representative $\sigma_\infty$ for $\pi^{-1}$, so that $\psi = \varepsilon \circ \sigma_\infty$.
\end{definition}

For each $1\le j \le n$ let $\lambda_j$ denote a null-homotopy over $R$ for the action of $t_j$ on $X$. Recall from Section \ref{section:koszulcpxs} that there is a morphism $\vartheta': X[n] \lto X \otimes \skos(\bold{t})$ in the category of linear factorisations of $W$ over $R$ with $\varepsilon \circ \vartheta' = 1$, and we write $\vartheta$ for the composite $\vartheta = \pi \circ \vartheta' = (-1)^n \lambda_1 \cdots \lambda_n$. Consider the following diagram of morphisms of linear factorisations of $W$ over $S$
\[
\xymatrix@+3pc{
X/\bold{t} X \ar@<-0.8ex>[r]_{\sigma_\infty} \ar@/_3pc/[rr]_{\psi} & X \otimes \skos(\bold{t}) \ar@<-0.8ex>[l]_{\pi} \ar@<-0.8ex>[r]_(0.55){\varepsilon} & X[n]. \ar@/_3pc/[ll]_{\vartheta} \ar@<-0.8ex>[l]_(0.45){\vartheta'}
}
\]

\begin{definition} Let $e$ denote the endomorphism $\vartheta \circ \psi$ of $X/\bold{t} X$.
\end{definition}

\begin{theorem}\label{theorem:first_attempt_e} We have $\psi \circ \vartheta = 1$ and hence $e$ is an idempotent endomorphism of $X/\bold{t} X$ in $\hf(S,W)$. Given a deformation retract (\ref{eq:koszul_perturb1}) and homotopies $\lambda_i$ on $X$ such that $\lambda_i \cdot d + d \cdot \lambda_i = t_i \cdot 1_X$ we have
\[
e = (-1)^n \lambda_1 \cdots \lambda_n \varepsilon (hd)^n \sigma.
\]
\end{theorem}
\begin{proof}
We need only observe that in $\psi = \varepsilon \circ \sigma_\infty$ only $(h d)^n \sigma$ survives the projection.
\end{proof}

It is clear that $e$ is independent, up to homotopy, of the deformation retract datum (\ref{eq:koszul_perturb1}) and our chosen homogeneous basis of $X$, as these only enter the construction via $\sigma_\infty$ which is well-defined up to homotopy as the inverse of $\pi$. Clearly if $R/\bold{t} R$ is a finite rank free $S$-module then $X/\bold{t} X$ is a finite rank matrix factorisation of $W$ over $S$. It is also worth noting that by remarks in Section \ref{section:koszulcpxs} there is a homotopy equivalence over $S$
\[
X/\bold{t} X \cong (\varphi_*(X) \oplus \varphi_*(X)[1])^{\oplus 2^{n-1}}.
\]
Note that the above description of $e$ holds without any assumption on the characteristic. If we settle for the case where $S$ is a $\mathbb{Q}$-algebra, and make some other mild assumptions about $R$, we obtain the formula for $e$ given in the introduction. This is the subject of the next several sections.

\begin{remark} Suppose that for $i > 0$ there are deformation retracts (\ref{eq:koszul_perturb1}) with $\bold{t}^i = \{t_1^i, \ldots, t_n^i\}$ in place of $\bold{t}$, and that $R/\bold{t}^i R$ is a finite rank free $S$-module. If $Y$ is a finite rank matrix factorisation of $W$ over $R$ then it is compact in $\HMF(R,W)$ and hence $\uHom(Y,-)$ commutes with homotopy colimits \cite[Lemma 2.8]{Neeman96}. With reference to the system (\ref{eq:koszul_sequence}), it follows that
\begin{align*}
\uHom(Y, \holim_i X \otimes \skos(\bold{t}^i)) &\cong \operatorname{colim}_i \uHom(Y, X \otimes \skos(\bold{t}^i))\\
&\cong \uHom(Y, \operatorname{colim}_i X \otimes \skos(\bold{t}^i))\\
&\cong \uHom(Y, X \otimes \skos_\infty(\bold{t}))\\
&\cong \uHom(Y, X).
\end{align*}
If the finite rank matrix factorisations compactly generate $\HMF(R,W)$ (say $R$ is a complete regular local ring and $R/W$ has an isolated singularity, see \cite{dyck4}) then we may deduce from this calculation that $X$ is the homotopy colimit of the $X \otimes \skos(\bold{t}^i)$, and since the pushforward functor commutes with homotopy colimits we conclude that $X$ is the homotopy colimit
\[
X = \holim_i X/\bold{t}^i X = \holim_i\left( \xymatrix@C+1pc{ X/\bold{t} X \ar[r]^{t_1 \cdots t_n} & X/\bold{t}^2 X \ar[r]^{t_1 \cdots t_n} & X/\bold{t}^3 X \ar[r] & \cdots} \right)
\]
of finite rank matrix factorisations in $\HMF(S,W)$. As we have already seen, the morphism from the first term $X/\bold{t} X$ in this system to the homotopy colimit is a split epimorphism.
\end{remark}

Let us write $\vartheta_X$ and $\psi_X$ for the morphisms constructed above. It is clear from the construction that $\psi_X$ is natural, since both $\pi$ and $\varepsilon$ are natural. Consider a second finite rank matrix factorisation $Y$ of $W$ on which the $t_j$ act null-homotopically, suppose with null-homotopies $\mu_j$ giving rise to $\vartheta_{Y}, \psi_{Y}$.

\begin{lemma}\label{lemma:nate} Given a morphism $\alpha: X \lto Y$ of matrix factorisations over $R$ the diagram
\[
\xymatrix@C+1pc{
X[n] \ar[d]_{\vartheta_X} \ar[r]^{\alpha[n]} & Y[n]\\
X/\bold{t} X \ar[r]_{\alpha} & Y/\bold{t} Y \ar[u]_{\psi_{Y}}
}
\]
commutes up to homotopy.
\end{lemma}
\begin{proof}
We have $\psi_{Y} \circ \alpha \circ \vartheta_X = \alpha[n] \circ \psi_X \circ \vartheta_X = \alpha[n]$.
\end{proof}

It is useful to know that pushing forward commutes with completion:

\begin{remark}\label{remark:completionandpush} Let $\psi: R \lto R'$ be a flat ring morphism inducing an isomorphism $R/\bold{t} R \cong R'/\bold{t} R'$. For example $R$ could be noetherian and $\psi$ the $\bold{t} R$-adic completion. The image of $\bold{t}$ is a quasi-regular sequence in $R'$ and $X' = X \otimes_R R'$ is a finite rank matrix factorisation of $W$ over $R'$ with homotopies $\lambda_i' = \lambda_i \otimes_R R'$. Assume there is a deformation retract of $\mathbb{Z}$-graded complexes over $k$
\begin{equation}\label{eq:koszul_perturb12}
\xymatrix@C+2pc{
(R'/\bold{t} R',0) \ar@<-0.8ex>[r]_{\sigma'} & (\skos_{R'}(\bold{t}), \delta), \ar@<-0.8ex>[l]_{\pi'}
} \quad h
\end{equation}
with $h^2 = 0, h\sigma' = 0$ and $\pi' h = 0$. Run the above construction with $R'$ in place of $R$ and let $\vartheta_{R'}, \psi_{R'}$ and $e' = \vartheta_{R'} \circ \psi_{R'}$ be the morphisms produced. With $X \lto X'$ the canonical morphism it is clear from the construction that the two squares implicit in the diagram
\[
\xymatrix@C+3pc@R+1pc{
X'/\bold{t} X' \ar@<-0.8ex>[r]_{\psi_{R'}} & X'[n] \ar@<-0.8ex>[l]_{\vartheta_{R'}}\\
X/\bold{t} X \ar[u]^{\cong}\ar@<-0.8ex>[r]_{\psi} & X[n] \ar@<-0.8ex>[l]_{\vartheta} \ar[u]
}
\]
commute in $\hf(S,W)$. Hence the two constructions produce the same idempotent $e = e'$ on $X/\bold{t} X$ and the canonical map $\varphi_*(X) \lto (\psi\varphi)_*(X')$ is a homotopy equivalence over $S$.
\end{remark}

\section{Residues and Traces}\label{section:residuesandtraces}

The residue symbol was introduced by Grothendieck as part of his theory of duality for algebraic varieties \cite[III.9]{ResiduesDuality}. The properties of residues are stated in \cite{ResiduesDuality} but there are no proofs; in the same setting of global duality more detailed treatments were given later by several authors including Lipman \cite{Lipman84} and Conrad \cite{Conrad00}. In \cite{Lipman87} Lipman gives an elementary approach to the residue symbol via Hochschild (co)homology, and this is the definition best suited to our needs. We slightly rephrase his definition using connections, in order to make the relationship to the Atiyah class transparent in the sequel.

In this section $\varphi: S \lto R$ is a ring morphism and $\bold{t} = \{t_1,\ldots,t_n\}$ is a quasi-regular sequence in $R$ such that $R/\bold{t} R$ is a finitely generated projective $S$-module. Set $S[\bold{t}] = S[t_1,\ldots,t_n]$ with the $t_i$ as formal variables, so that $R$ is a $S[\bold{t}]$-algebra in the obvious way. We write $\Omega_{S[\bold{t}]/S} = \wedge \Omega^1_{S[\bold{t}]/S}$ and grade the exterior algebra by $| \ud t_i| = -1$, so that $\Omega^i_{S[\bold{t}]/S}$ sits in degree $-i$. We assume that $R$ admits a flat $S$-linear connection
\begin{equation}
\nabla^0: R \lto R \otimes_{S[\bold{t}]} \Omega^1_{S[\bold{t}]/S}
\end{equation}
with the property that
\begin{equation}\label{eq:condition_a3}
\Ker(\nabla^0) + \bold{t} R = R.
\end{equation}
This hypothesis, that the sections constant with respect to $\nabla^0$ generate $R/\bold{t} R$, is very mild and is satisfied in the natural examples. We use it in order to prove that the residue symbol defined here agrees with the one in \cite{Lipman87}. In the following we denote  the extension of $\nabla^0$ to a $S$-linear operator on $R \otimes_{S[\bold{t}]} \Omega_{S[\bold{t}]/S}$ by $\nabla$. We write $\partial/\partial t_j$ for the $S$-linear map
\begin{equation}\label{eq:partialpartialf}
R \xlto{\nabla^0} R \otimes_{S[\bold{t}]} \Omega^1_{S[\bold{t}]/S} \xlto{(\ud t_j)^*} R,
\end{equation}
and with this notation $\nabla$ is given for $r \in R$ and $\omega \in \Omega_{S[\bold{t}]/S}$ by
\begin{equation}\label{eq:partialf_jprop2}
\nabla(r \cdot \omega) = \sum_{j=1}^n \partial /\partial t_j(r) \cdot \ud t_j \wedge \omega.
\end{equation}
Observe that, in contrast to Definition \ref{defn:connections} where connections are maps of degree $+1$, the opposite convention adopted here for the grading on $\Omega_{S[\bold{t}]/S}$ means that $\nabla$ now has degree $-1$. This choice is dictated by the fact that $\nabla$ is going to play the role of a homotopy in what follows.

\begin{example}\label{example:globalringnabla} If $R$ is a free $S[\bold{t}]$-module of finite rank then it admits a flat connection $\nabla^0$ satisfying (\ref{eq:condition_a3}), constructed by extending the K\" ahler differential on $S[\bold{t}]$. For example, suppose $k = S$ is a field and let $R = \oplus_{n \ge 0} R_n$ be a finitely generated $k$-algebra with $R_0 = k$ and let $\bold{t}$ be a homogeneous system of parameters, that is, suppose that each $t_i$ is homogeneous and that the Krull dimension of $R/\bold{t} R$ is zero. Then by Noether normalisation $R$ is a finitely generated $k[\bold{t}]$-module, and if $R$ is further Cohen-Macaulay (e.g. a polynomial ring) then $R$ is free over $k[\bold{t}]$.
\end{example}

\begin{example}\label{example:localringnabla} If $R$ is separated and complete in its $\bold{t} R$-adic topology then $\nabla^0$ satisfying (\ref{eq:condition_a3}) can be constructed from any $S$-linear section of the canonical map $R \lto R/\bold{t} R$, see Appendix \ref{section:derhamsplit}.
\end{example}

For any element $r \in R$ we also write $r$ for the endomorphism $r \cdot 1_M$ of an $R$-module $M$. %

\begin{lemma} For $r \in R$ the map $[\nabla, r]$ is $S[\bold{t}]$-linear on $R \otimes_{S[\bold{t}]} \Omega_{S[\bold{t}]/S}$.
\end{lemma}
\begin{proof}
One proves that $[\nabla,r]$ commutes with each $t_i$ using the Leibniz rule.
\end{proof}

Given $s,r_1,\ldots,r_n \in R$ the $S[\bold{t}]$-linear map $\gamma = s [\nabla, r_1] \cdots [\nabla, r_n]$ on $R \otimes_{S[\bold{t}]} \Omega_{S[\bold{t}]/S}$ has degree $-n$ and therefore defines a map from $R$ to $R \otimes_{S[\bold{t}]} \Omega^n_{S[\bold{t}]/S}$. Trivialising $\Omega^n_{S[\bold{t}]/S}$ using the basis $\ud t_1 \wedge \cdots \wedge \ud t_n$, we view $\gamma$ as a $S[\bold{t}]$-linear endomorphism of $R$, which induces a $S$-linear endomorphism of $R/\bold{t} R$.

\begin{definition}\label{defn:residuesymbol} For $s,r_1,\ldots,r_n \in R$ the \emph{residue symbol} is a $S$-linear trace over $R/\bold{t} R$ 
\begin{equation}\label{eq:residuesymbol1}
\Res \Bigg[ \begin{matrix} s \cdot \ud r_1 \cdots \ud r_n \\ t_1, \ldots, t_n \end{matrix} \Bigg] = \tr_S\big( s [\nabla, r_1] \cdots [\nabla, r_n] \big) \in S.
\end{equation}
If we want to emphasise the rings, we will write $\Ress{R/S}\big[ - \big]$ for the residue symbol.
\end{definition}

\begin{lemma}\label{lemma:reformulateressym} For $s,r_1,\ldots,r_n \in R$ we have
\begin{equation}\label{eq:lemmaaltres}
\Res \Bigg[ \begin{matrix} s \cdot \ud r_1 \cdots \ud r_n \\ t_1, \ldots, t_n \end{matrix} \Bigg] = \sum_{\tau \in S_n} \sgn(\tau) \tr_S\big( s[\partial/\partial t_{\tau(1)},r_1] \cdots [\partial/\partial t_{\tau(n)},r_n] \big).
\end{equation}
\end{lemma}
\begin{proof}
This follows by direct calculation from (\ref{eq:partialf_jprop2}). Notice that for $r \in R$ the commutator $[\partial/\partial t_i, r]$ is $S[\bold{t}]$-linear on $R$, and hence defines a $S$-linear map on $R/\bold{t} R$.
\end{proof}

It is easy to see that this definition agrees with Lipman's in \cite{Lipman87}, and for reader's convenience we explain the details in Appendix \ref{section:lipmanres}. In particular, the residue symbol we define is independent of the choice of connection $\nabla^0$, which we can therefore omit from the notation. We will use standard properties of the residue symbol, particularly the transition formula of \cite[(2.8)]{Lipman87}. In the elegant approach of  \emph{ibid.} residues are defined via the action of Hochschild cohomology on Hochschild homology, and before proceeding we want to briefly sketch how this works. 

If $M$ is a left module over $R^e = R \otimes_S R$ (i.e. an $R$-$R$-bimodule with compatible $S$-actions) then evaluating cocyles on cycles yields a natural $R$-linear action for $q \ge 0$
\begin{equation}
\rho^q: \HH^q(R, M) \otimes_R \HH_q(R,R) \lto \HH_0(R,M).
\end{equation}
If we take the bimodule $M = \Hom_S(P,P)$ and set $I = \bold{t} R, P = R/\bold{t} R$ and $(I/I^2)^* = \Hom_P(I/I^2, P)$ then there is a canonical map \cite[(1.8.3)]{Lipman87}
\begin{equation}\label{eq:lip_hh_first_map}
\otimes^n_R (I/I^2)^* \lto \HH^n( R, \Hom_S(P, P) ).
\end{equation}
There is also the well-known map
\begin{equation}\label{eq:lip_hh_second_map}
\theta: \Omega^n_{R/S} \lto \HH_n(R,R).
\end{equation}
The trace map $\Hom_S(P,P) \lto S$ factors through $\HH_0(R, \Hom_S(P,P)) = R \otimes_{R^e} \Hom_S(P,P)$, and together with $\rho^n$ these maps give rise to the \emph{residue homomorphism}
\[
\res: \otimes^n_R (I/I^2)^* \otimes_R \Omega^n_{R/S} \lto \HH_0(R, \Hom_S(P,P)) \lto S.
\]
Since $\bold{t}$ is quasi-regular the $t_i$ form a $P$-basis for $I/I^2$ and we write $t_i^*$ for the dual basis. If we denote by $[ t_1, \ldots, t_n ]$ the image under (\ref{eq:lip_hh_first_map}) of the tensor $t_1^* \otimes \cdots \otimes t_n^*$ then the residue homomorphism applied to $[t_1,\ldots,t_n] \otimes \theta(s \ud r_1 \wedge \cdots \wedge \ud r_n)$ is precisely the residue symbol (\ref{eq:residuesymbol1}). In particular the numerator in (\ref{eq:residuesymbol1}) can be manipulated as an element of $\Omega^n_{R/S}$.

\subsection{Contracting the Koszul complex}\label{section:contractingkoszul}

We keep the above setting, and construct from $\nabla$ a $S$-homotopy $H$ contracting the Koszul complex $\skos(\bold{t})$ onto its cohomology $R/\bold{t}R$. By this we mean that there is a $S$-linear morphism of complexes $\sigma: R/\bold{t} R \lto \skos(\bold{t})$ such that $\pi \sigma = 1$, and that $H$ is a $S$-linear homotopy between $\sigma \pi$ and $1$.

We grade the exterior algebra $R \otimes_{S[\bold{t}]} \Omega_{S[\bold{t}]/S}$ by $|\ud t_i| = -1$, and we define an $R$-linear degree $+1$ map $\delta$ on this graded algebra by contraction $\delta = (\sum_{i=1}^n t_i (\ud t_i)^*) \,\neg\, (-)$. This gives the differential in the Koszul complex $\skos(\bold{t}) = (R \otimes_{S[\bold{t}]} \Omega_{S[\bold{t}]/S}, \delta)$. We write $\pi$ both for the map $R \lto R/\bold{t} R$ and for the morphism of complexes $\skos(\bold{t}) \lto R/\bold{t} R$ and we assume that the connection $\nabla^0$ is standard in the following sense:

\begin{definition}\label{defn:standard_connection} A flat $S$-linear connection $\nabla^0: R \lto R \otimes_{S[\bold{t}]} \Omega^1_{S[\bold{t}]/S}$ is \emph{standard} if for $p > 0$ the $S$-linear map $p \cdot 1_R + \delta \nabla^0$ on $R$ is a bijection which identifies $\bold{t} R$ with $\bold{t}R$, and $\Im(\delta \nabla^0) = \bold{t} R$.
\end{definition}

This amounts to an assumption of zero characteristic: if $S$ is a $\mathbb{Q}$-algebra then the constructions of Example \ref{example:globalringnabla} and Example \ref{example:localringnabla} both produce standard flat connections. We state things this way since ``standardness'' is what we actually use, and this way we handle the local and graded cases at the same time. The next lemma follows by a standard argument of differential geometry:

\begin{lemma}\label{lemma:usualderham} For $r \in R$ and $\omega = \ud t_{i_1} \wedge \cdots \wedge \ud t_{i_p}$ with distinct indices $i_j$ we have
\[
(\delta \nabla + \nabla \delta)(r \cdot \omega ) = (p + \delta \nabla^0)(r) \cdot \omega.
\]
\end{lemma}

The de Rham differential gives a homotopy between zero and the identity on the Koszul complex up to some normalisations coming from the degree of polynomials and weights of differential forms. Our assumption that $\nabla^0$ is standard means that in negative degrees $\delta \nabla + \nabla \delta$ is invertible, and we use the inverse to define $H$ from $\nabla$.

\begin{definition} Let $H_\nabla$ be the degree $-1$ $S$-linear map on $R \otimes_{S[\bold{t}]} \Omega_{S[\bold{t}]/S}$ defined by
\[
H_\nabla = (\delta \nabla + \nabla \delta)^{-1} \nabla,
\]
and write $\tau = \delta \nabla + \nabla \delta$. We usually suppress the choice of connection and simply write $H$ for $H_\nabla$.
\end{definition}

\begin{lemma} For $\omega \in R \otimes_{S[\bold{t}]} \Omega_{S[\bold{t}]/S}$ homogeneous of nonzero degree $(H \delta + \delta H)(\omega) = \omega$.
\end{lemma}
\begin{proof}
Since $\delta \tau = \tau \delta$ we have $\delta = \tau \delta \tau^{-1}$ in negative degrees. Hence on $\omega$,
\begin{align*}
H \delta + \delta H &= \tau^{-1} \nabla \delta + \delta \tau^{-1} \nabla\\
&= \tau^{-1} \tau( \tau^{-1} \nabla \delta + \delta \tau^{-1} \nabla )\\
&= \tau^{-1} \nabla \delta + \tau^{-1} \delta \nabla\\
&= \tau^{-1} \tau = 1,
\end{align*}
as claimed.
\end{proof}

\begin{lemma} In degree zero of $R \otimes_{S[\bold{t}]} \Omega_{S[\bold{t}]/S}$ we have $(1- \delta H)(\bold{t}R) = 0$.
\end{lemma}
\begin{proof}
Given $x \in \bold{t} R$ we have $x = \delta \nabla^0(r)$ for some $r \in R$, since $\nabla^0$ is standard. Hence
\[
(1 - \delta H)(x) = (1 - \delta H) \delta \nabla^0(r) = \delta \nabla^0(r) - \delta\tau^{-1} \nabla^0 \delta \nabla^0(r).
\]
But $\nabla \delta = \tau - \delta \nabla$ so $\delta \tau^{-1} \nabla^0 \delta \nabla^0 = \delta \tau^{-1}(\tau - \delta \nabla^1) \nabla^0 = \delta \nabla^0$, so $(1 - \delta H)(x) = 0$.
\end{proof}

\begin{lemma} $H^2 = 0$ on $R \otimes_{S[\bold{t}]} \Omega_{S[\bold{t}]/S}$.
\end{lemma}
\begin{proof}
It suffices to show that $\nabla \tau^{-1} \nabla = 0$. But $\nabla \tau = \tau \nabla$ so $\tau \nabla \tau^{-1} \nabla = \nabla \tau \tau^{-1} \nabla = \nabla^2 = 0$ and hence $\nabla \tau^{-1} \nabla = 0$.
\end{proof}

It is now clear that there is an induced splitting of $\pi$.

\begin{lemma} There is a unique $S$-linear map $\sigma: R/\bold{t} R \lto R$ such that $\delta H + \sigma \pi = 1$. For this map $\sigma$ we have $\nabla^0 \sigma = 0$.
\end{lemma}

Notice that even if we did not assume that $\Ker(\nabla^0) + \bold{t} R = R$, this would follow from the fact that $\nabla^0$ is standard. Viewing $\sigma$ as a morphism of $S$-complexes $R/\bold{t} R \lto \skos(\bold{t})$ we have $\pi \sigma = 1$ and $H \delta + \delta H = 1 - \sigma \pi$, so indeed $H$ is a $S$-homotopy contracting the Koszul complex onto its cohomology. We refer to $H$ as the \emph{de Rham contraction}. We can summarise the situation by stating that $H$ and $\sigma$ give a deformation retract datum
\begin{equation}\label{eq:formula_idempotent_kaoz_first}
\xymatrix@C+2pc{
(R/\bold{t} R,0) \ar@<-0.8ex>[r]_(0.4)\sigma & (R \otimes_{S[\bold{t}]} \Omega_{S[\bold{t}]/S},\delta), \ar@<-0.8ex>[l]_(0.6){\pi}
} \quad -H
\end{equation}
of $\mathbb{Z}$-graded complexes over $S$. Moreover it is clear that $H^2 = 0, H \sigma = 0$ and $\pi H = 0$.

\section{The Atiyah class}\label{section:atiyah_class}

Let $\varphi: S \lto R$ be a ring morphism and $(X,d)$ a linear factorisation over $R$ of a potential $W \in S$. Following \cite{atiyahconn, illusie} we introduce the Atiyah class of $X$ relative to this ring morphism. Let us assume that $X$ admits an $S$-linear connection (throughout, any connection on a $\mathbb{Z}/2$-graded module is a degree zero map of the underlying graded modules)
\begin{equation}\label{eq:nablatilde}
\nabla^0: X \lto X \otimes_R \Omega^1_{R/S}.
\end{equation}
Such a connection will exist, for example, if $X$ is a projective $R$-module. We make $X \otimes_R \Omega^1_{R/S}$ into a linear factorisation of $W$ with differential $d \otimes 1$.

\begin{proposition} The commutator $[d,\nabla^0]$ is an $R$-linear morphism of linear factorisations
\[
X \lto X \otimes_R \Omega^1_{R/S}[1]
\]
which is independent, up to homotopy, of the choice of connection.
\end{proposition}
\begin{proof}
One checks by explicit calculation that $[d,\nabla^0]$ is $R$-linear, and independence of the homotopy is proved as in the case of $\mathbb{Z}$-graded complexes (see e.g. \cite[(3.2)]{buchweitz03}). To prove that this is a morphism, we use the fact that the potential belongs to $S$ and therefore commutes with $\nabla^0$:
\begin{align*}
d [d,\nabla^0] &= W \cdot \nabla^0 - d \nabla^0 d\\
&= \nabla^0 \cdot W - d \nabla^0 d\\
&= \nabla^0 d d - d \nabla^0 d\\
&= -[d,\nabla^0] d.
\end{align*}
\end{proof}

\begin{definition}\label{defn:atiyah_class} The morphism $\At_{R/S}(X) = [d,\nabla^0]$ is called the \emph{Atiyah class} of $X$ relative to $R/S$, or more precisely relative to $\varphi$. Iterating and using the maps $\Omega^i \otimes \Omega^j \lto \Omega^{i+j}$ we have an $R$-linear morphism of factorisations
\begin{equation}\label{eq:atiyah_class_powers}
\At_{R/S}(X)^k := [d,\nabla^0]^k: X \lto X \otimes_R \Omega^k_{R/S}[k]
\end{equation}
for $1 \le k \le n$, again independent of the choice of connection up to homotopy.
\end{definition}

We will see that it is actually the additive inverse of the Atiyah class which naturally appears in the context of residues; nonetheless the above sign convention is the standard one, for the reasons explained in \cite[(3.17)]{buchweitz03}.

\begin{remark}\label{remark:finite_rank_is_derivative} Let $\ud_{R/S}: R \lto \Omega^1_{R/S}$ be the K\"ahler differential and suppose that $X$ is a finite rank matrix factorisation. We fix a homogeneous $R$-basis $\{ \xi_i \}_{i \in I}$ for $X$, in terms of which we define a homogeneous $S$-linear connection on $X$ by $\nabla^0(r \xi_i) = \xi_i \otimes \ud_{R/S} r$. The matrix of $\At_{R/S}(X) = [d,\nabla^0]$ as a map
\[
\oplus_i R \xi_i = X \lto X \otimes_R \Omega^1_{R/S} = \oplus_i (R \xi_i \otimes \Omega^1_{R/S})
\]
is given by $-\ud_{R/S}(d_X)$ evaluated entrywise.
\end{remark}

In the context of perturbations it is also natural to define a \emph{$\mathbb{Z}/2$-graded $S$-linear connection} to be a degree zero $S$-linear map $\nablagr^0: X \lto X \otimes_R \Omega^1_{R/S}$ such that for $r \in R$ and homogeneous $x \in X$,
\[
\nablagr^0(rx) = r \nablagr^0(x) + (-1)^{|x|} x \otimes \ud r.
\]
We can uniquely extend such a map to a degree one $S$-linear map $\nablagr$ on $X \otimes_R \Omega_{R/S}$ satisfying
\begin{equation}\label{eq:extending_connections}
\nablagr( x \cdot \omega ) = \nablagr(x) \cdot \omega + (-1)^{|x|} x \otimes \ud \omega
\end{equation}
for homogeneous $x \in X \otimes_S \Omega_{R/S}$ and $\omega \in \Omega_{R/S}$. If $G$ denotes the grading operator $x \mapsto (-1)^{|x|}x$ on $X$ then the map $\nablagr^0 \lto \nablagr^0 G$ is a bijection between $\mathbb{Z}/2$-graded connections and ordinary connections, and we say that $\nablagr^0$ is \emph{flat} if $\nablagr^0 G$ is flat, or equivalently if $\nablagr^2 = 0$.

Let $\nablagr^0$ be a $\mathbb{Z}/2$-graded connection with associated ordinary connection $\nabla^0 = \nablagr^0 G$. There is a natural isomorphism which moves the suspension to the second tensor factor
\[
(X \otimes_R \Omega^1_{R/S})[1] \cong X \otimes_R (\Omega^1_{R/S}[1]), \quad x \otimes \omega \mapsto (-1)^{|x|} x \otimes \omega
\]
and it is readily checked that $\{ d, \nablagr^0 \} = d \nablagr^0 + \nablagr^0 d$ is a morphism and that

\begin{lemma}\label{lemma:diagram_two_atiyahs} The diagram
\[
\xymatrix{
& X \ar[dl]_{\At_{R/S}(X)} \ar[dr]^{-\{ d, \nablagr^0 \}}\\
X \otimes_R \Omega^1_{R/S}[1] \ar[rr]_{\cong} & & X \otimes_R (\Omega^1_{R/S}[1])
}
\]
commutes.
\end{lemma}

\subsection{Residues and the Atiyah class}\label{section:residues_on_objects}

Let $\varphi: S \lto R$ be a ring morphism, $\bold{t} = \{ t_1,\ldots,t_n \}$ a quasi-regular sequence in $R$ such that $R/\bold{t} R$ is a finitely generated projective $S$-module, and set $S[\bold{t}] = S[t_1,\ldots,t_n]$. We assume that $R$ admits a flat $S$-linear connection $\nabla^0: R \lto R \otimes_{S[\bold{t}]} \Omega^1_{S[\bold{t}]/S}$ with the property that $\Ker(\nabla^0) + \bold{t} R = R$. 

Let $(X,d_X)$ be a finite rank matrix factorisation over $R$ of a potential $W \in S$. The $n$-th power of the Atiyah class of $X$ relative to $S[\bold{t}]/S$ is a $S[\bold{t}]$-linear morphism of factorisations
\[
\At_{S[\bold{t}]/S}(X)^n: X \lto X \otimes_{S[\bold{t}]} \Omega^n_{S[\bold{t}]/S}[n] \cong X[n]
\]
where we trivialise $\Omega^n_{S[\bold{t}]/S}$ using the $n$-form $\ud t_1 \wedge \cdots \wedge \ud t_n$. With this trivialisation incorporated into the notation there is an induced $k$-linear map of degree $n$ on $X/\bold{t} X$, and

\begin{proposition}\label{prop:residues_on_objects} Let $\alpha$ be a degree $n$ $R$-linear map on $X$ with the property that the induced map $X/\bold{t} X \lto X/\bold{t} X[n]$ is a morphism of linear factorisations. Then
\begin{equation}\label{eq:residues_on_objects}
\str_S\big( \alpha \circ \At_{S[\bold{t}]/S}(X)^n \big) = (-1)^n \Ress{R/S}\begin{bmatrix} \str_R\!\big( \alpha \cdot \ud_{R/S}(d_X)^{\wedge n} \big) \\ t_1,\ldots,t_n \end{bmatrix}
\end{equation}
where the supertrace on the left is of a $S$-linear endomorphism of $X/\bold{t} X$.
\end{proposition}

The proof will be given at the end of this section. Observe that the Atiyah class is relative to $S[\bold{t}]/S$ rather than $R/S$. If $R$ were the polynomial ring $S[x_1,\ldots,x_n]$ this distinction would reflect the difference between residues with denominator $t_1 \cdots t_n$ and residues with denominator $x_1 \cdots x_n$. 

We begin with some comments on the notation. Given square matrices $A, R_1, \ldots, R_n$ over $R$ the product $A \cdot \ud_{R/S}(R_1) \wedge \cdots \wedge \ud_{R/S}(R_n)$ is a matrix of $n$-forms, and the supertrace has a residue
\[
\Ress{R/S}\begin{bmatrix} \str_R\!\big( A \cdot \ud_{R/S}(R_1) \wedge \cdots \wedge \ud_{R/S}(R_n) \big) \\ t_1,\ldots,t_n \end{bmatrix} \in S.
\]
Choosing a homogeneous $R$-basis $\{ \xi_i \}_{i \in I}$ for $X$ and writing $\alpha, d_X$ as matrices gives meaning to the right hand side of (\ref{eq:residues_on_objects}), and for simplicity we fix this homogeneous basis in what follows.

Let $\nabla^0: X \lto X \otimes_{S[\bold{t}]} \Omega^1_{S[\bold{t}]/S}$ be the flat $S$-linear connection on $X$ defined by $\nabla^0( r \xi_i ) = \xi_i \cdot \nabla^0 r$ and let $\nablagr^0 = \nabla^0 G$ be the associated $\mathbb{Z}/2$-graded connection, which extends to a map $\nablagr$ on $X \otimes_{S[\bold{t}]} \Omega_{S[\bold{t}]/S}$. As our representative of the Atiyah class, we take
\[
\At_{S[\bold{t}]/S}(X) = [d_X, \nabla^0].
\]
The anticommutator $\{ d_X, \nablagr \}$ is a $S[\bold{t}]$-linear map on $X \otimes_{S[\bold{t}]} \Omega_{S[\bold{t}]/S}$. Iterating and composing with the projection $\varepsilon$ of Section \ref{eq:augmentation_epsilon} we have a degree $n$ $S[\bold{t}]$-linear map $\varepsilon \{ d_X, \nablagr \}^n: X \lto X$ and it is clear from Lemma \ref{lemma:diagram_two_atiyahs} and the definition of $\varepsilon$ in Section \ref{section:koszulcpxs} that

\begin{lemma}\label{lemma:atiyah_vs_commutator} $\At_{S[\bold{t}]/S}(X)^n = (-1)^{\binom{n+1}{2}} \varepsilon \{ d_X, \nablagr \}^n$.
\end{lemma}

It will be useful to expand $\{ d_X, \nablagr \}$ in terms of the maps $\partial_j = \partial/\partial t_j$ and $\ud t_j$ on $X \otimes_{S[\bold{t}]} \Omega_{S[\bold{t}]/S}$, extended using the basis $\{ \xi_i\}_{i \in I}$ for $X$. Here $\ud t_j$ denotes left multiplication, which involves a sign
\[
\ud t_j \wedge ( \xi_i \otimes r \ud t_{j_1} \cdots \ud t_{j_p} ) = (-1)^{|\xi_i|} \otimes r \ud t_j \ud t_{j_1} \cdots \ud t_{j_p}.
\]
Clearly $\nablagr = \sum_{j=1}^n \partial_j \circ \ud t_j$, and $\ud t_j$ anticommutes with $d_X$. Hence $\{ \partial_j \circ \ud t_j, d_X \} = \ud t_j \circ [\partial_j, d_X]$~and
\begin{align*}
\{ d_X, \nablagr \}^n &= \sum_{\tau \in S_n} \ud t_{\tau(1)} [\partial_{\tau(1)}, d_X] \ud t_{\tau(2)}  [\partial_{\tau(2)}, d_X] \cdots \ud t_{\tau(n)}  [\partial_{\tau(n)}, d_X]\\
&= (-1)^{\binom{n+1}{2}} \sum_{\tau \in S_n} \sgn(\tau) [\partial_{\tau(1)}, d_X] [\partial_{\tau(2)}, d_X] \cdots  [\partial_{\tau(n)}, d_X] \ud t_1 \cdots \ud t_n.
\end{align*}
Evaluated on a homogeneous element $x \in X$ this has $X$-degree $|x| + n$, and the initial map wedging with $\ud t_1 \cdots \ud t_n$ contributes a sign factor $(-1)^{n|x|}$, so
\begin{equation}
\varepsilon \{ d_X, \nablagr \}^n = (-1)^{\binom{n}{2}} \sum_{\tau \in S_n} \sgn(\tau) [\partial_{\tau(1)}, d_X] [\partial_{\tau(2)}, d_X] \cdots  [\partial_{\tau(n)}, d_X]
\end{equation}
and hence finally
\begin{equation}\label{eq:atiyahnexpanded}
\At_{S[\bold{t}]/S}(X)^n = (-1)^{\binom{n+1}{2}} \varepsilon \{ d_X, \nablagr \}^n = (-1)^n \sum_{\tau \in S_n} \sgn(\tau) [\partial_{\tau(1)}, d_X] [\partial_{\tau(2)}, d_X] \cdots  [\partial_{\tau(n)}, d_X]
\end{equation}
as degree $n$ $S[\bold{t}]$-linear maps on $X$.

\begin{proof}[Proof of Proposition \ref{prop:residues_on_objects}] For simplicity we assume that $R/\bold{t} R$ is free over $S$ on some basis $\{ e_z \}_{1 \le z \le m}$, the modifications when $R/\bold{t} R$ is just projective are obvious. The $\mathbb{Z}/2$-graded $S$-module $X/\bold{t} X$ has homogeneous basis $\{ \xi_i \otimes e_z \}_{i,z}$, and the coefficient matrix with respect to this basis of the $S$-linear map $[\partial_j, d_X]$ on $X/\bold{t} X$ is
\[
[\partial_j, d_X]_{iz, i'z'} = [\partial_j, d_{X,ii'}]_{zz'}
\]
where $[\partial_j, d_{X,ii'}]$ denotes a commutator of $S$-linear maps on $R/\bold{t} R$, with $d_{X,ii'}$ acting by multiplication, and $[ - ]_{zz'}$ denotes coefficients in the $S$-basis $\{ e_z \}_z$. Similarly we view $\alpha$ as a supermatrix over $R$, and as a $S$-linear operator on $X/\bold{t} X$ the matrix is given by $\alpha_{iz,i'z'} = [ \alpha_{ii'} ]_{zz'}$, the coefficients of the operator on $R/\bold{t} R$ which multiplies by $\alpha_{ii'}$. Then using (\ref{eq:atiyahnexpanded}) and Lemma \ref{lemma:reformulateressym} we have
\begin{align*}
\str_S\!\big( \alpha \At_{S[\bold{t}]/S}(X)^n \big) &= \sum_{\tau \in S_n} \sgn(\tau)(-1)^n \str_S\big( \alpha [\partial_{\tau(1)}, d_X] [\partial_{\tau(2)}, d_X] \cdots  [\partial_{\tau(n)}, d_X]\big)\\
&= \sum_{i,i_1,\ldots} \sum_{z,z_1,\ldots} \sum_{\tau \in S_n} \sgn(\tau) (-1)^{n+|\xi_i|}\alpha_{iz,i_1z_1} [\partial_{\tau(1)}, d_{X,i_1i_2}]_{z_1z_2} \cdots [\partial_{\tau(n)}, d_{X,i_n i}]_{z_n z}\\
&= \sum_{i,i_1,\ldots} \sum_{\tau \in S_n} \sgn(\tau) (-1)^{n+|\xi_i|} \tr_S\big( \alpha_{ii_1} [\partial_{\tau(1)}, d_{X,i_1 i_2}] \cdots [\partial_{\tau(n)}, d_{X,i_n i}]\big)\\
&= \sum_{i,i_1,\ldots} (-1)^{n+|\xi_i|} \Ress{R/S}\begin{bmatrix} \alpha_{ii_1} \ud(d_{X,i_1i_2}) \wedge \cdots \wedge \ud(d_{X,i_ni}) \\ t_1,\ldots,t_n \end{bmatrix}\\
&= (-1)^n\Ress{R/S}\begin{bmatrix} \str_R\!\big( \alpha \cdot \ud(d_X)^{\wedge n} \big) \\ t_1,\ldots,t_n \end{bmatrix},
\end{align*}
as claimed.
\end{proof}


\section{A formula for the idempotent}\label{section:idempotentformula}

Using the de Rham contraction of a Koszul complex onto its cohomology we express the idempotent from Section \ref{section:pertfrompoint} in terms of the Atiyah class. Throughout $S$ is a $\mathbb{Q}$-algebra, $\varphi: S \lto R$ is a morphism of rings, $W \in S$ is a potential, and $(X,d_X)$ is a finite rank matrix factorisation of $W$ over $R$. We are given a quasi-regular sequence $\bold{t} = \{t_1,\ldots,t_n\}$ in $R$ such that $R/\bold{t} R$ is a finitely generated projective $S$-module and we fix a homogeneous $R$-basis $\{ \xi_i \}_{i \in I}$ for $X$. We assume that $t_j \cdot 1_X$ is null-homotopic for $1 \le j \le n$ and let $\lambda_j$ be an $R$-linear homotopy on $X$ such that $\{ \lambda_j, d_X \} = t_j \cdot 1_X$.

Let $S[\bold{t}] = S[t_1,\ldots,t_n]$ be a polynomial ring in the $t_i$ and assume that $R$ admits a flat $S$-linear connection $\nabla^0$ as a $S[\bold{t}]$-module which is standard in the sense of Definition \ref{defn:standard_connection} (by Appendix \ref{section:derhamsplit} that this is automatic if $R$ is separated and complete in the $\bold{t} R$-adic topology). Let $\nabla$ denote the induced $S$-linear map on $R \otimes_{S[\bold{t}]} \Omega_{S[\bold{t}]/S}$. In Section \ref{section:contractingkoszul} we constructed an $S$-linear homotopy $H = H_\nabla$ on $R \otimes_{S[\bold{t}]} \Omega_{S[\bold{t}]/S}$ and a $S$-linear section $\sigma: R/\bold{t} R \lto R$ of the map $\pi: R \lto R/\bold{t} R$, such that
\begin{equation}\label{eq:formula_idempotent_kaoz}
\xymatrix@C+2pc{
(R/\bold{t} R,0) \ar@<-0.8ex>[r]_(0.4)\sigma & (R \otimes_{S[\bold{t}]} \Omega_{S[\bold{t}]/S},\delta), \ar@<-0.8ex>[l]_(0.6){\pi}
} \quad -H
\end{equation}
is a deformation retract datum of $\mathbb{Z}$-graded complexes over $S$ (notice the sign on $H$). Recall that $\tau = \delta \nabla + \nabla \delta$ is invertible on $R \otimes_{S[\bold{t}]} \Omega_{S[\bold{t}]/S}$ in negative degrees and that $H = \tau^{-1} \nabla$. At this point the hypotheses of Section \ref{section:pertfrompoint} are satisfied, with (\ref{eq:formula_idempotent_kaoz}) in place of (\ref{section:pertfrompoint}) and we conclude that the augmentation $\pi: X \otimes\skos(\bold{t}) \lto X/\bold{t} X$ is a homotopy equivalence over $S$ with inverse
\[
\sigma_\infty = \sum_{m \ge 0} ( (-H) d_X )^m \sigma = \sum_{m \ge 0} (-1)^m ( H d_X )^m \sigma.
\]
In this expression for $\sigma_\infty$ the $H$ is more properly written as $1 \otimes H$, since it is defined by extending $H$ on $R \otimes_{S[\bold{t}]} \Omega_{S[\bold{t}]/S}$ across our fixed homogeneous basis for $X$ to obtain an $S$-linear map on the tensor product $X \otimes_{S[\bold{t}]} \Omega_{S[\bold{t}]/k}$, and similarly for $\sigma$. For example, the degree one $S$-linear map $\nablagr = 1 \otimes \nabla$ on $X \otimes_{S[\bold{t}]} \Omega_{S[\bold{t}]/S}$ is defined by
\[
\nablagr( r\xi_i \otimes \ud t_{i_1} \cdots \ud t_{i_p} ) = \sum_{j=1}^n (-1)^{|\xi_i|} \xi_i \otimes \partial/\partial t_j(r) \cdot \ud t_j \ud t_{i_1} \cdots \ud t_{i_p},
\]
and in this notation the $H$ appearing in $\sigma_\infty$ is $\tau^{-1} \nablagr$. With $\psi = \varepsilon \circ \sigma_\infty$ there is a diagram in the homotopy category of linear factorisations over $S$
\[
\xymatrix@+3pc{
X/\bold{t} X \ar@<-0.8ex>[r]_{\psi} & X[n] \ar@<-0.8ex>[l]_{\vartheta}
}
\]
with $\psi \circ \vartheta = 1$ and $e = \vartheta \circ \psi$ an idempotent. Moreover by Theorem \ref{theorem:first_attempt_e},
\begin{equation}\label{eq:first_e_form}
e = \lambda_1 \cdots \lambda_n \varepsilon (H d_X)^n \sigma.
\end{equation}
We proceed to rewrite $e$ in terms of the Atiyah class $\At_{S[\bold{t}]/S}(X)$. First we need to insert $H = \tau^{-1} \nablagr$ and get rid of the $\tau's$. Observe that in the terminology of Section \ref{section:residues_on_objects}, $\nablagr$ is the $\mathbb{Z}/2$-graded flat $S$-linear connection on $X$ induced by $\nabla^0$. Recall also that the commutator $\{ d_X, \nablagr \}$ is an $S[\bold{t}]$-linear map on $X \otimes_{S[\bold{t}]} \Omega_{S[\bold{t}]/k}$. Let us take $(\ref{eq:first_e_form})$ and using $\nablagr^2 = 0$ and $\nablagr \sigma = 0$ rewrite $e$ as
\begin{equation}
\begin{split}
e &= \lambda_1 \cdots \lambda_n \varepsilon (\tau^{-1} \nablagr d_X) \cdots (\tau^{-1} \nablagr d_X)\sigma\\
&= \lambda_1 \cdots \lambda_n \varepsilon \tau^{-1} \{d_X, \nablagr\} \cdots \tau^{-1} \{d_X, \nablagr\}\sigma
\end{split}
\end{equation}
Each of these factors of $\tau^{-1}$ multiplies by $1/p$ where $p$ is the weight of the differential form on which the $\tau^{-1}$ is evaluated; since each $\{ d_X, \nablagr \}$ increases the weight by one, it follows that
\begin{equation}\label{eq:almost_there_e}
e = \frac{1}{n!} \lambda_1 \cdots \lambda_n \varepsilon \{d_X, \nablagr\}^n \sigma.
\end{equation}
Moreover each of the maps involved in $e$ are $S[\bold{t}]$-linear, hence pass to $S$-linear maps on $X/\bold{t} X$, and with this understood we can drop the $\sigma$ from (\ref{eq:almost_there_e}). Finally, we know from Lemma \ref{lemma:atiyah_vs_commutator} that
\[
\At_{S[\bold{t}]/S}(X)^n = (-1)^{\binom{n+1}{2}} \varepsilon \{ d_X, \nablagr \}^n,
\]
which completes the proof of the main theorem:

\begin{theorem}\label{theorem:main} There is a diagram in the homotopy category of linear factorisations of $W$ over $S$
\[
\xymatrix@+3pc{
X/\bold{t} X \ar@<-0.8ex>[r]_{\psi} & \varphi_*(X)[n] \ar@<-0.8ex>[l]_{\vartheta}
}
\]
with $\psi \circ \vartheta = 1$ and the idempotent $e = \vartheta \circ \psi$ given by the formula
\[
e = \frac{1}{n!} (-1)^{\binom{n+1}{2}} \lambda_1 \cdots \lambda_n \At_{S[\bold{t}]/S}(X)^n.
\]
\end{theorem}

\begin{remark} If $R/\bold{t} R$ is free of finite rank over $S$ then $e$ is an idempotent in $\hmf(S,W)$, and if $R$ is free over $S$ then the pushforward of $X$ is an object of $\HMF(S,W)$.
\end{remark}

\begin{remark} Sign factors of the form $(-1)^{\binom{n}{2}}$ are essentially unavoidable when one starts making abstract dualities explicit in terms of residues; for a good discussion of this point we refer the reader to \cite[Appendix A]{Conrad00}.
\end{remark}

Next we give a presentation of $e$ better suited to computing in examples. For $1 \le j \le n$ there is an $S$-linear map $\partial_j = \partial/\partial t_j$ defined on $X$ by $\partial_j(r \xi_i) = \partial_j(r) \xi_i$. Both $\lambda_i$ and $[\partial_j, d_X]$ are $S[\bold{t}]$-linear maps of degree one on $X$, and therefore induce $S$-linear maps on $X/\bold{t} X$.

\begin{corollary}\label{cor:simplified_e} As an $S$-linear endomorphism of $X/\bold{t} X$,
\[
e = \frac{1}{n!} (-1)^{\binom{n}{2}} \sum_{\tau \in S_n} \sgn(\tau) \cdot \lambda_1 \cdots \lambda_n \cdot [\partial_{\tau(1)}, d_X] \cdots [\partial_{\tau(n)}, d_X].
\]
\end{corollary}
\begin{proof}
This is immediate from (\ref{eq:atiyahnexpanded}) and the theorem.
\end{proof}

\begin{remark} Suppose that $R$ is a free $S[\bold{t}]$-module of finite rank on the basis $e_1,\ldots,e_q$. Then $X$ is a finite rank matrix factorisation over $S[\bold{t}]$ with basis $\{ e_j \cdot \xi_i \}_{j,i}$, and to compute the Atiyah class of $X$ over $S[\bold{t}]/S$ we use the flat connection $\nabla^0(a \cdot e_j \cdot \xi_i) = e_j \cdot \xi_i \otimes \ud a$ extending $\ud: S[\bold{t}] \lto \Omega^1_{S[\bold{t}]/S}$. Then by Remark \ref{remark:finite_rank_is_derivative} the Atiyah class is, as a matrix of $1$-forms, the entrywise differential
\[
\At_{S[\bold{t}]/S}(X) = [d_X, \nabla^0] = -\ud_{S[\bold{t}]/S}(d_X)
\]
of the matrix of $d_X$ over $S[\bold{t}]$. 
\end{remark}

In simple cases one can split the idempotent ``by hand'' to find a finite rank matrix factorisation homotopy equivalent to the pushforward (what we call the \emph{finite model}). A nontrivial example is Kn\"orrer periodicity, studied in Section \ref{section:knorrer}. In more complicated examples the computer can do the splitting, and with Nils Carqueville this process has been implemented in Singular \cite{carqmurf}. We do not know how to give the differential on the finite model directly, but see Remark \ref{remark:avoiding_idempotents}. 

Our main applications of the theorem involve isolated singularities, but our techniques apply more generally to non-isolated singularities.

\begin{remark} Let $W \in \mathbb{C}[\bold{x}] = \mathbb{C}[x_1,\ldots,x_n]$ be a potential whose singular locus contains the origin. Suppose that the critical locus of $W$ is complete intersection of codimension $c$ at the origin, so $I = (\partial_{x_1} W, \ldots, \partial_{x_n} W) \mathbb{C}\llbracket \bold{x}\rrbracket$ can be generated by some regular sequence $\bold{t} = \{ t_1, \ldots, t_c \}$, and let
\[
\rho: \mathbb{C}\llbracket \bold{y}\rrbracket = \mathbb{C}\llbracket y_1,\ldots,y_{n-c}\rrbracket \lto \mathbb{C}\llbracket \bold{x}\rrbracket/I
\]
be a Noether normalisation. Then since $\mathbb{C}\llbracket \bold{x}\rrbracket/I$ is complete intersection and hence Cohen-Macaulay, $\mathbb{C}\llbracket \bold{x}\rrbracket/I$ is a free $\mathbb{C}\llbracket \bold{y}\rrbracket$-module of finite rank via $\rho$. We can factor this map via a continuous morphism of $\mathbb{C}$-algebras $\varphi$,
\[
\xymatrix{
& \mathbb{C}\llbracket \bold{x}\rrbracket \ar[dr]^{\can}\\
\mathbb{C}\llbracket \bold{y}\rrbracket \ar[ur]^{\varphi}\ar[rr]_{\rho} & & \mathbb{C}\llbracket \bold{x}\rrbracket/I.
}
\]
Then the morphism $\varphi$ satisfies the hypotheses at the beginning of this section, and $\mathbb{C}\llbracket \bold{x}\rrbracket$ admits a standard $\mathbb{C}\llbracket \bold{y}\rrbracket$-linear flat connection as a $\mathbb{C}\llbracket \bold{y}\rrbracket[\bold{t}]$-module by virtue of its completeness. To have a concrete example, consider $W = x^2 - y^2 \in \mathbb{C}[x,y,z]$. In this case we take $\bold{t} = \{ x,y \}$ and the Noether normalisation is the canonical map $\mathbb{C}\llbracket z\rrbracket \lto \mathbb{C}\llbracket x,y,z\rrbracket \lto \mathbb{C}\llbracket x,y,z\rrbracket/(x,y) = \mathbb{C}\llbracket z\rrbracket$.

The connection on $\mathbb{C}\llbracket x,y,z\rrbracket$ produced by Appendix \ref{section:derhamsplit} simply differentiates in the normal directions to the singular locus: 
\begin{gather*}
\nabla^0: \mathbb{C}\llbracket x,y,z\rrbracket \lto \mathbb{C}\llbracket x,y,z\rrbracket \otimes_{\mathbb{C}\llbracket z\rrbracket[x,y]} \Omega^1_{\mathbb{C}\llbracket z\rrbracket[x,y]/\mathbb{C}\llbracket z\rrbracket},\\
\nabla^0(f) = \frac{\partial f}{\partial x} \ud x + \frac{\partial f}{\partial y} \ud y.
\end{gather*}
\end{remark}

\section{The Chern character}\label{section:chern_char_over}

The Chern character of a matrix factorisation is the trace of the identity endomorphism, for some suitable notion of trace \cite{polishchuk,dyckmurf}. The pushforward is the splitting of an idempotent which involves the Atiyah class, and we have seen in Section \ref{section:residues_on_objects} that tracing expressions involving the Atiyah class naturally produces residues; we therefore expect a residue formula for the Chern character of the pushforward, and in this section we give such formulas.

We begin with an alternative approach to the Hirzebruch-Riemann-Roch theorem of Polishchuk and Vaintrob \cite{polishchuk}. Their approach is via the derived Morita theory of dg categories and Hochschild homology, which is conceptually clear but requires some sophisticated machinery. In contrast, our proof is very elementary: besides the main theorem above, we need only basic properties of residues and the nondegeneracy of the residue pairing of local duality.

Let $R$ be the power series ring $k\llbracket x_1,\ldots,x_n\rrbracket$ where $k$ is a field of characteristic zero, and let $\mf{m}$ denote the maximal ideal. Let $W \in \mf{m}^2$ be a polynomial defining an isolated singularity, by which we mean that the Jacobi algebra $J_W = R/\bold{t} R$ is finite-dimensional over $k$, where $\bold{t}$ denotes the regular sequence $\{\partial_{x_1} W, \ldots, \partial_{x_n} W\}$. The partial derivatives $\partial_{x_i} W$ act null-homotopically on every matrix factorisation of $W$ (see Remark \ref{remark:partial_gives_homotopy}) and using this we define the Chern character:

\begin{definition} Let $(X,d_X)$ be a finite rank matrix factorisation of $W$ over $R$. For each $1 \le i \le n$ let $\lambda_i$ be an $R$-linear homotopy on $X$ with $\lambda_i d_X + d_X \lambda_i = \partial_{x_i} W \cdot 1_X$. The \emph{Chern character} of $X$ is defined to be the equivalence class in $J_W$ of the supertrace
\[
\ch(X) = (-1)^{\binom{n}{2}} \str_R( \lambda_1 \cdots \lambda_n ).
\]
By Lemma \ref{lemma:indept_homotopy_chern} this does not depend on the choice of homotopies, and if we fix a homogeneous basis for $X$ and take $\lambda_i = \partial_{x_i}(d_X)$ then we recover the definitions in \cite{polishchuk,dyckmurf}.
\end{definition}

Let $(X,d_X), (Y,d_Y)$ be finite rank matrix factorisations of $W$ over $R$ with chosen homogeneous bases. We apply the pushforward construction to the matrix factorisation $\shom(X,Y)$ of zero over $R$ and the ring morphism $\varphi: k \lto R$. Observe that each $\partial_{i} W$ acts null-homotopically on $\shom(X,Y)$, and a null-homotopy for this endomorphism is given by the map
\[
\partial_{x_i}(d_Y) \circ -: \shom(X,Y) \lto \shom(X,Y)
\]
which post-composes with $\partial_{i}(d_Y)$. Since $R$ is complete the main theorem applies, and tells us that there is a diagram in the homotopy category of matrix factorisations of zero over $k$
\begin{equation}\label{eq:homassummand}
\xymatrix@+3pc{
\shom(X,Y) \otimes_R J_W \ar@<-0.8ex>[r]_(0.6){\psi} & \shom(X,Y)[n] \ar@<-0.8ex>[l]_(0.4){\vartheta}
}
\end{equation}
with $\psi \circ \vartheta = 1$ and the idempotent $e = \vartheta \circ \psi$ given by
\[
e = \frac{1}{n!} (-1)^{\binom{n+1}{2}} \partial_{x_1}(d_Y) \cdots \partial_{x_n}(d_Y) \At_{k[\bold{t}]/k}(\shom(X,Y))^n.
\]

\begin{remark}\label{remark:avoiding_idempotents} The category of finite-dimensional $\mathbb{Z}/2$-graded vector spaces has split idempotents, and it is natural to wonder if one can directly construct the differential of a splitting $Z$ of $e$ in terms of invariants of $X$ and $Y$. We don't know such a construction, but here is an observation: because $W$ is nonzero the differentials on $X,Y$ must be \emph{square} matrices, and the constructions that naturally come to mind will produce a differential $d_Z$ for the splitting which is also a square matrix. But then $\dim_k Z^0 = \dim_k Z^1$ implies that $\dim_k H^0(Z) = \dim_k H^1(Z)$. Since $Z$ is quasi-isomorphic to $\Hom(X,Y)$ we infer that
\[
\dim_k \uHom(X,Y) = \dim_k H^0(Z) = \dim_k H^1(Z) = \dim_k \uHom(X, Y[1])
\]
and it is certainly possible to find $X,Y$ for which this is false. 

What construction can we apply, beginning with square matrices, to produce nonsquare matrix factorisations of zero? The most obvious answer is that we can split idempotents; after all, the matrix factorisations of zero that we miss using square matrices are $k \lto 0 \lto k$ and $0 \lto k \lto 0$, the direct sum of which is square.
\end{remark}

\subsection{The Cardy condition}\label{section:hrr}

Keeping the above notation we use $e$ to give an alternative proof of the Cardy condition for matrix factorisations first proved by Polishchuk and Vaintrob \cite{polishchuk}. Given a $\mathbb{Z}/2$-graded complex $Z$ of $k$-vector spaces we write $H(Z) = H^0(Z) \oplus H^1(Z)$ for the total cohomology.

\begin{theorem}\label{theorem:cardycond} Let $\alpha: X \lto X$ and $\beta: Y \lto Y$ be morphisms. Then the induced map
\[
\shom(\alpha, \beta): \shom(X,Y) \lto \shom(X,Y)
\]
induces a degree zero map of finite-dimensional $\mathbb{Z}/2$-graded vector spaces
\[
H_{\alpha, \beta} = H(\shom(\alpha, \beta)): H(\shom(X,Y)) \lto H(\shom(X,Y))
\]
and (writing $\ud \bold{x} = \ud x_1 \wedge \cdots \wedge \ud x_n$)
\[
\str_k(H_{\alpha, \beta}) = (-1)^{\binom{n}{2}} \Res\begin{bmatrix}\str_R\big( \alpha \partial_{x_1}(d_X) \cdots \partial_{x_n}(d_X) \big) \cdot \str_R\big(\beta \partial_{x_1}(d_Y) \cdots \partial_{x_n}(d_Y)\big) \cdot \ud \bold{x} \\ \partial_{x_1} W, \ldots, \partial_{x_n} W \end{bmatrix}.
\]
\end{theorem}

Taking $\alpha = 1_X$ and $\beta = 1_Y$ yields the Hirzebruch-Riemann-Roch theorem.

\begin{corollary} We have
\[
\chi( \Hom_R(X,Y) ) = (-1)^{\binom{n}{2}} \Res\begin{bmatrix} \ch(X) \cdot \ch(Y) \cdot \ud \bold{x} \\ \partial_{x_1} W, \ldots, \partial_{x_n} W \end{bmatrix}.
\]
\end{corollary}

\begin{proof}[Proof of Theorem \ref{theorem:cardycond}]
Set $\mathscr{H} = \shom(X,Y)$ and consider the commutative diagram
\[
\xymatrix@C+2pc{
\mathscr{H} \otimes J_W \ar[r]^{\shom(\alpha, \beta)} & \mathscr{H} \otimes J_W \ar[d]^{\psi} \\
\mathscr{H}[n] \ar[r]^{\shom(\alpha, \beta)[n]} \ar[u]^{\vartheta} & \mathscr{H}[n]\\
H(\mathscr{H})[n] \ar[u]_{\cong}^f \ar[r]_{H_{\alpha, \beta}[n]} & H(\mathscr{H})[n] \ar[u]^{\cong}_{f},
}
\]
where commutativity of the top square is due to Lemma \ref{lemma:nate}. We deduce that
\begin{align*}
\str_k( H_{\alpha, \beta} ) &= (-1)^n \str_k( H_{\alpha,\beta}[n] )\\
&= (-1)^n \str_k( f^{-1} \psi \shom(\alpha, \beta) \vartheta f)\\
&= (-1)^n \str_k( \shom(\alpha, \beta) \vartheta \psi )\\
&= (-1)^n \str_k( \shom(\alpha, \beta) e ).
\end{align*}
Using Proposition \ref{prop:residues_on_objects} and the formula for $e$,
\begin{align*}
\str_k( H_{\alpha, \beta} ) &= \frac{1}{n!} (-1)^{\binom{n}{2}} \str_k( \shom(\alpha, \beta)  \partial_{x_1}(d_Y) \cdots \partial_{x_n}(d_Y) \At_{k[\bold{t}]/k}(\mathscr{H})^n )\\
&= \frac{1}{n!} (-1)^{\binom{n+1}{2}} \Ress{R/k}\begin{bmatrix} \str_R\big( \shom(\alpha, \beta)  \partial_{x_1}(d_Y) \cdots \partial_{x_n}(d_Y) \ud( d_{\mathscr{H}} )^{\wedge n} \big) \;/\; \bold{t} \end{bmatrix},
\end{align*}
where $\Ress{R/k}[ - / \bold{t} ]$ denotes a residue with denominator $\partial_1 W, \ldots, \partial_n W$. One subtlety is that in this residue $\ud(-)$ denotes the K\"ahler differential $\ud: R \lto \Omega^1_{R/k}$ on power series rather than polynomials. But since this all takes place inside the residue, which is annihilated by all monomials of sufficiently high degree, this agrees with the K\"ahler differential on $k[x_1,\ldots,x_n]$. If we now identify $\mathscr{H}$ with the tensor product $X^{\lor} \otimes Y$ and $\shom(\alpha, \beta)$ with $\alpha^{\lor} \otimes \beta$, then we find that
\begin{equation}\label{eq:chieuler}
\str_k( H_{\alpha, \beta} ) = \frac{1}{n!} (-1)^{\binom{n+1}{2}} \Ress{R/k}\begin{bmatrix} \str_R\big( \alpha^{\lor} \otimes \beta \partial_{x_1}(d_Y) \cdots \partial_{x_n}(d_Y) \circ \ud( d_{X^{\lor}} + d_Y )^{\wedge n} \big) \;/\; \bold{t} \end{bmatrix}.
\end{equation}
We expand $\ud( d_{X^{\lor}} + d_Y)^{\wedge n}$ as a sum over permutations $\tau \in S_n$ of signed terms (writing $\partial_i = \partial_{x_i}$)
\begin{equation}\label{eq:chieuler2}
\partial_{\tau(1)}(d_{X^{\lor}} + d_Y) \cdots \partial_{\tau(n)}(d_{X^{\lor}} + d_Y) \ud \bold{x}
\end{equation}
which itself expands as a sum of signed terms of the form
\begin{equation}\label{eq:chieulermult}
\partial_{i_1}(d_{X^{\lor}}) \cdots \partial_{i_p}(d_{X^{\lor}}) \otimes \partial_{j_1}(d_{Y}) \cdots \partial_{j_q}(d_{Y}).
\end{equation}
But it is a simple exercise using nondegeneracy of the residue pairing on $J_W$ (see Lemma \ref{lemma:nonfullstrvanishes}) to check that in $J_W$ we have
\[
\str_R( \alpha^{\lor} \partial_{i_1}(d_{X^{\lor}}) \cdots \partial_{i_p}(d_{X^{\lor}}) ) = 0
\]
unless $p = n$, from which we deduce that for each permutation $\tau$ only one of the summands coming from the expansion of (\ref{eq:chieuler2}) contributes to the residue, and so using $\str(c \otimes d) = \str(c) \cdot \str(d)$ and with a sign change coming from commuting $\partial_i(d_{X^{\lor}})$'s past $\partial_j(d_Y)$'s,
\begin{multline}\label{eq:hrr_almost_there}
\str_k( H_{\alpha, \beta} ) = \frac{1}{n!} (-1)^{\binom{n}{2}}\sum_{\tau \in S_n} \sgn(\tau) \Ress{R/k}\Big[ \str_R\big( \alpha^{\lor} \partial_{\tau(1)}(d_{X^{\lor}}) \cdots \partial_{\tau(n)}(d_{X^{\lor}}) )\\ \cdot \str_R( \beta \partial_{1}(d_Y) \cdots \partial_{n}(d_Y)) \cdot \ud \bold{x} \big) \;/\; \bold{t} \Big].
\end{multline}
It remains to observe that by Lemma \ref{eq:lemmaordchern}
\begin{align*}
\str_R\big( \alpha^{\lor} \partial_{\tau(1)}(d_{X^{\lor}}) \cdots \partial_{\tau(n)}(d_{X^{\lor}}) ) &= \str_R\big( \alpha \partial_{\tau(1)}(d_{X}) \cdots \partial_{\tau(n)}(d_{X}) )\\
&= \sgn(\tau) \str_R\big( \alpha \partial_{1}(d_{X}) \cdots \partial_{n}(d_{X}) ).
\end{align*}
Inserting this into (\ref{eq:hrr_almost_there}) completes the proof.
\end{proof}

\subsection{The general pushforward}\label{section:chern_character}

In this section we give the general formula for the Chern character of the pushforward. For this to make sense the base ring should be a ring of power series, so we take $k$ to be a field of characteristic zero and set $S = k\llbracket y_1,\ldots,y_m\rrbracket$, with $\mf{m}$ denoting the maximal ideal. Let $W$ be a polynomial in $\mf{m}^2$ with $J_W = S/(\partial_{y_1} W, \ldots, \partial_{y_m} W)$ a finite-dimensional $k$-algebra and $\varphi: S \lto R$ a ring morphism.

Let $(X,d_X)$ denote a finite rank matrix factorisation of $W$ over $R$ and $\bold{t} = \{ t_1,\ldots,t_n \}$ a quasi-regular sequence in $R$ such that the following conditions are satisfied:
\begin{itemize}
\item the morphism $t_j \cdot 1_X$ is null-homotopic over $R$ for $1 \le j \le n$,
\item the morphism $\partial_{y_i} W \cdot 1_X$ is null-homotopic over $R$ for $1 \le i \le m$,
\item $R$ admits a standard flat $S$-linear connection $R \lto R \otimes_{S[\bold{t}]} \Omega^1_{S[\bold{t}]/S}$,
\item and $R/\bold{t} R$ is free of finite rank over $S$.
\end{itemize}
Then we can apply the pushforward construction to see that the pushforward $\varphi_*(X)$ is homotopy equivalent to a finite matrix factorisation of $W$ over $S$, and

\begin{proposition}\label{prop:chern_char_psh} If for each $1 \le j \le m$ and $1\le i \le n$ we choose $R$-linear homotopies $\lambda_j$ and $\mu_i$ on $X$ such that $\lambda_j d_X + d_X \lambda_j = t_j \cdot 1_X$ and $\mu_i d_X + d_X \mu_i = \partial_{y_i} W \cdot 1_X$ then as elements of $J_W$,
\begin{equation}\label{eq:chern_char_psh_fom}
\ch(\varphi_*(X)) = \frac{1}{n!} (-1)^{\binom{n+1}{2} + \binom{m}{2}} \Ress{R/S} \Bigg[ \begin{matrix} \str_R\big( \mu_1 \cdots \mu_m \lambda_1 \cdots \lambda_n \ud_{R/S}(d_X)^{\wedge n} \big) \\ t_1, \ldots, t_n \end{matrix} \Bigg].
\end{equation}
\end{proposition}

Before giving the proof, let us consider the typical example.

\begin{example} Suppose our ring morphism is of the form
\[
\varphi: k\llbracket y_1,\ldots,y_m\rrbracket \lto R = k\llbracket x_1,\ldots,x_n,y_1,\ldots,y_m\rrbracket.
\]
It is automatic by Appendix \ref{section:derhamsplit} that $R$ admits a standard flat $S$-linear connection as an $S[\bold{t}]$-module, and by the usual argument $\partial_{y_i} W$ acts null-homotopically on $X$, with homotopy $\mu_i = \partial_{y_i}(d_X)$.
\end{example}

With the notation of the proposition, the results of Section \ref{section:idempotentformula} tell us that there is a diagram in the homotopy category of linear factorisations of $W$ over $S$
\begin{equation}\label{eq:push_chern_diagram1}
\xymatrix@+3pc{
X/\bold{t} X \ar@<-0.8ex>[r]_(0.6){\psi} & X[n] \ar@<-0.8ex>[l]_(0.4){\vartheta}
}
\end{equation}
with $\psi \circ \vartheta = 1$ and
\begin{equation}\label{eq:push_chern_diagram2}
e = \vartheta \psi = \frac{1}{n!} (-1)^{\binom{n+1}{2}} \lambda_1 \cdots \lambda_n \At_{S[\bold{t}]/S}(X)^n.
\end{equation}
Notice that $X/\bold{t} X$ is a finite rank matrix factorisation. The residue formula for the Chern character follows, as in the previous section, from taking the ``trace'' of $1_{X[n]} = \psi \vartheta$. Previously the appropriate trace was just the supertrace, but now we need to use the trace internal to the Calabi-Yau category $\hmf(S,W)$, which is the boundary-bulk map of \cite{kapli,murfet-2009,dyckmurf}.

Suppose for the moment that $m$ is even. For any finite rank matrix factorisation $Y$ of $W$ over $S$, it follows from \cite[Remark 5.18]{murfet-2009} that there is a canonically defined $J_W$-linear map
\begin{equation}\label{eq:beta_formula}
\begin{split}
&\beta_Y: \uHom(Y, Y) \lto J_W,\\
\beta_Y(f) = \;&(-1)^{\binom{m}{2}} \str_S( f \partial_{y_1}(d_Y) \cdots \partial_{y_m}(d_Y) )
\end{split}
\end{equation}
so that the Chern character of $Y$ is $\ch(Y) = \beta_Y(1_Y) \in J_W$. Below we will make use of the fact that $\beta_Y(f)$ can be computed using, in place of the partial derivative $\partial_{y_i}(d_Y)$, any $R$-linear homotopy $\lambda_i$ with $\lambda_i d_Y + d_Y \lambda_i = \partial_{y_i} W \cdot 1_Y$ (see Lemma \ref{lemma:indept_homotopy_chern}). In the approach of \cite{murfet-2009} one derives properties of $\beta_Y$ from the fact that it composes with the embedding $\iota$ of local duality
\begin{gather*}
\iota: J_W \lto H^m_{\mf{m}}(S),\\
r \mapsto \big[ r \;/\; \partial_{y_1} W, \ldots, \partial_{y_m} W \big]
\end{gather*}
(with $[ - / - ]$ denoting a generalised fraction) to give the trace map $\Tr_Y$ of \cite[Definition 5.11]{murfet-2009}. That is, there is a commutative diagram
\begin{equation}\label{eq:commute_iota}
\xymatrix{
\uHom(Y,Y) \ar[dr]_{\beta_Y} \ar[rr]^{\Tr_Y} & & H^m_{\mf{m}}(S).\\
& J_W \ar[ur]_{\iota}
}
\end{equation}
In \cite[Theorem 4.11]{murfet-2009} it was shown that for morphisms $f: Y \lto Z$ and $g: Z \lto Y$ we have $\Tr_Z( fg ) = \Tr_Y( gf )$, hence $\beta_Z(fg) = \beta_Y(gf)$, and further for a morphism $f: Y \lto Y$ we have $\Tr_Y(f) = - \Tr_{Y[1]}(f[1])$, hence $\beta_Y(f) = - \beta_{Y[1]}(f[1])$. Given a matrix factorisation $Y$ of infinite rank homotopy equivalent to a matrix factorisation of finite rank, any such finite model has the same Chern character and we define this to be the Chern character of $Y$.

\begin{proof}[Proof of Proposition \ref{prop:chern_char_psh}]
If $m$ is odd then the Chern character is zero and the right hand side involves a supertrace of an odd supermatrix, hence is also zero, so we may assume that $m$ is even. Let us return to the diagram (\ref{eq:push_chern_diagram1}) and the idempotent $e$. In the category $\hmf(S,W)$ all idempotents split, so we can find a finite rank matrix factorisation $Z$ and morphisms $f,g$ as in the diagram
\[
\xymatrix@+3pc{
X/\bold{t} X \ar@<-0.8ex>[r]_(0.6){g} & Z \ar@<-0.8ex>[l]_(0.4){f}
}
\]
such that $g \circ f = 1_Z$ and $f \circ g = e$. But then $Z[n] \cong \varphi_*(X)$ in $\HMF(S,W)$ and hence
\begin{align*}
\ch(\varphi_*(X)) &= \beta_{Z[n]}(1_{Z[n]}) = (-1)^n \beta_Z(1_Z)\\
&= (-1)^n \beta_Z(g \circ f)\\
&= (-1)^n \beta_{X/\bold{t} X}( f \circ g )\\
&= (-1)^{n + \binom{m}{2}} \str_S( \mu_1 \cdots \mu_m e )\\
&= \frac{1}{n!} (-1)^{\binom{n}{2} + \binom{m}{2}} \str_S( \mu_1 \cdots \mu_m \lambda_1 \cdots \lambda_n \At_{S[\bold{t}]/S}(X)^n )
\end{align*}
which by Proposition \ref{prop:residues_on_objects} is equal to the right hand side of (\ref{eq:chern_char_psh_fom}).
\end{proof}

\section{Convolution}\label{section:fusingtop}

Our main application of the finite pushforward is to give a description of the convolution of matrix factorisation kernels. We begin by explaining exactly what we mean by the convolution; this amounts to adapting standard material on integral functors, for which we refer the reader to \cite{huybrechts}. 

Let $k$ be a field of characteristic zero and $\bold{x} = \{ x_1,\ldots,x_n \}$ and $\bold{y} = \{ y_1, \ldots, y_m \}$ sets of variables. Given potentials $W \in k[\bold{x}]$ and $V \in k[\bold{y}]$ let $\hf(k[\bold{x}],W)$ and $\hf(k[\bold{y}],V)$ denote the homotopy categories of linear factorisations, $\pi_{\bold{x}}: k[\bold{x}] \lto k[\bold{x},\bold{y}]$ and $\pi_{\bold{y}}: k[\bold{y}] \lto k[\bold{x},\bold{y}]$ the canonical ring morphisms and write $\pi^*_{\bold{x}} = (-) \otimes_{k[\bold{x}]} k[\bold{x},\bold{y}] = (-) \otimes_k k[\bold{y}], (\pi_{\bold{x}})_*$ for the extension and restriction of scalars functors respectively.

\begin{definition}\label{defn:kernel_functor} Given a linear factorisation $E$ of $V - W$ over $k[\bold{x},\bold{y}]$ there is a $k$-linear functor
\[
\xymatrix@C+1pc{
\hf(k[\bold{x}],W) \ar[r]^(0.45){\pi^*_{\bold{x}}} & \hf(k[\bold{x},\bold{y}], W) \ar[r]^{E \otimes -} & \hf(k[\bold{x},\bold{y}], V) \ar[r]^(0.55){(\pi_{\bold{y}})_*} & \hf(k[\bold{y}],V)
}
\]
which we denote by $\Phi_E$. This is called the \emph{integral functor} with \emph{kernel} $E$.
\end{definition}

Given a third set of variables $\bold{z} = \{ z_1, \ldots, z_p \}$, consider the following commutative diagram,~with all morphisms canonical (there are duplicate labels, but it should be clear what we mean)
\begin{equation}\label{eq:diagram_ring_maps}
\xymatrix@C+1pc{
k[\bold{x}] \ar[dr]^{\pi_{\bold{x}}} \ar@/_3pc/[dddrr]_{\pi_{\bold{x}}} & & k[\bold{y}] \ar[dl]_{\pi_{\bold{y}}}\ar[dr]^{\pi_{\bold{y}}} & & k[\bold{z}] \ar[dl]_{\pi_{\bold{z}}}\ar@/^2pc/[dddll]^{\pi_{\bold{z}}}\\
& k[\bold{x},\bold{y}]\ar[dr]_{\pi_{\bold{x}\bold{y}}} & & k[\bold{y},\bold{z}]\ar[dl]^{\pi_{\bold{y}\bold{z}}}\\
& & k[\bold{x},\bold{y},\bold{z}] \\
& & k[\bold{x},\bold{z}] \ar[u]_{\pi_{\bold{x}\bold{z}}}
}
\end{equation}
Suppose that we are given linear factorisations $E \in \hf(k[\bold{x},\bold{y}], V - W)$ and $F \in \hf( k[\bold{y},\bold{z}], U - V)$ for some potential $U \in k[\bold{z}]$. Then $F \otimes_{k[\bold{y}]} E \cong \pi_{\bold{y}\bold{z}}^*(F) \otimes_{k[\bold{x},\bold{y},\bold{z}]} \pi_{\bold{x}\bold{y}}^*(E)$ is the external tensor product of $F$ and $E$. This is a linear factorisation of $U - W$ over $k[\bold{x},\bold{y},\bold{z}]$, and we define:

\begin{definition} The \emph{convolution} of $F$ and $E$ is the pushforward
\[
F \star E = (\pi_{\bold{x}\bold{z}})_*( F \otimes_{k[\bold{y}]} E) \in \hf( k[\bold{x},\bold{z}], U - W).
\]
\end{definition}

The kernels so far defined give rise to functors
\[
\xymatrix@C+1pc{
\hf(k[\bold{x}],W) \ar@/_2pc/[rr]_{\Phi_{F \star E}}\ar[r]^{\Phi_E} & \hf(k[\bold{y}],V) \ar[r]^{\Phi_F} & \hf(k[\bold{z}],U)
}
\]
and it is a simple exercise to check that composition corresponds to the convolution of kernels.

\begin{lemma}\label{lemma:convolution_is_composition} There is a natural equivalence $\Phi_F \circ \Phi_E \cong \Phi_{F \star E}$.
\end{lemma}

Our aim is to understand the conditions under which $F \star E$ is homotopy equivalent to a finite rank matrix factorisation, and to describe the finite model in terms of an idempotent. If we want to work over power series rings instead of polynomial rings, then we should use completed tensor products and make some other adjustments; we give the details in Section \ref{section:convolution_powerseries}.

From now on the potentials $U \in k[\bold{z}]$ and $W \in k[\bold{x}]$ are arbitrary, but we assume that $V \in (\bold{y})^2$ and that the critical locus of $V$ consists of the origin. Equivalently, the Jacobi algebra
\[
J_V = k[\bold{y}]/(\partial_{y_1} V, \ldots, \partial_{y_m} V)
\]
is finite-dimensional over $k$. This implies that the sequence $\bold{t} = \{ \partial_{y_1} V, \ldots, \partial_{y_m} V \}$ is regular in $k[\bold{y}]$. We give ourselves \emph{finite rank} matrix factorisations
\[
E \in \hmf(k[\bold{x},\bold{y}], V - W), \qquad F \in \hmf(k[\bold{y},\bold{z}], U - V)
\]
with fixed homogeneous bases $\{ \xi_i \}_{i \in I}$ and $\{ \eta_j \}_{j \in J}$ respectively. We are interested in the convolution $F \star E = (\pi_{\bold{x}\bold{z}})_*( F \otimes_{k[\bold{y}]} E)$. This is a matrix factorisation of infinite rank, and we want to apply the main theorem to this pushforward. Observe that $d^2_E = V - W$ so we have for $1 \le i \le m$
\[
\partial_{y_i}(d_E) \cdot d_E + d_E \cdot \partial_{y_i}(d_E) = \partial_{y_i} V \cdot 1_{E}.
\]
Hence the sequence $\bold{t}$ acts null-homotopically on $F \otimes_{k[\bold{y}]} E$ with null-homotopies $\lambda_i = 1 \otimes \partial_{y_i}(d_E)$. Since $k[\bold{x},\bold{y},\bold{z}]$ and its quotient by the ideal generated by $\bold{t}$ are both free over $k[\bold{x},\bold{z}]$, Theorem \ref{theorem:first_attempt_e} applies to the ring morphism $\pi_{\bold{x}\bold{z}}: k[\bold{x},\bold{z}] \lto k[\bold{x},\bold{y},\bold{z}]$ and there is a diagram
\[
\xymatrix@+3pc{
F/\bold{t} F \otimes_{k[\bold{y}]} E/\bold{t} E = (F \otimes_{k[\bold{y}]} E) \otimes_{k[\bold{y}]} J_V \ar@<-0.8ex>[r]_(0.7){\psi} & (F \star E)[m] \ar@<-0.8ex>[l]_(0.3){\vartheta}
}
\]
in $\HMF(k[\bold{x},\bold{z}], U - W)$ with $\psi \circ \vartheta = 1$ and $F/\bold{t} F \otimes_{k[\bold{y}]} E/\bold{t} E$ finite rank over $k[\bold{x},\bold{z}]$. To describe the idempotent $e = \vartheta \circ \psi$ in terms of the Atiyah class we need a standard flat $k[\bold{x},\bold{z}]$-linear connection $\nabla^0$ on $k[\bold{x},\bold{y},\bold{z}]$ as a module over $k[\bold{x},\bold{z},\bold{t}]$. It is enough to produce a standard, flat $k$-linear connection
\begin{equation}\label{eq:nablaboldy}
\nabla^0_{\bold{y}}: k[\bold{y}] \lto k[\bold{y}] \otimes_{k[\bold{t}]} \Omega^1_{k[\bold{t}]/k},
\end{equation}
since tensoring $\nabla^0_{\bold{y}}$ with $k[\bold{x},\bold{z}]$ produces the desired connection $\nabla^0$ on $k[\bold{x},\bold{y},\bold{z}]$. If $V$ is homogeneous the existence of such a connection is automatic (see Remark \ref{remark:homogeneous_get_connection} below) but in general it is easier to pass from $k[\bold{y}]$ to its completion $k\llbracket \bold{y}\rrbracket$ where a connection is always guaranteed to exist. To see how this works, consider the commutative diagram of canonical ring morphisms 
\[
\xymatrix@C+1pc@R-0.5pc{
& k[\bold{x},\bold{y},\bold{z}] \ar[dd]^{\omega}\\
k[\bold{x},\bold{z}] \ar[ur]^{\pi_{\bold{xz}}} \ar[dr]_{\rho}\\
& k\llbracket \bold{y}\rrbracket[\bold{x},\bold{z}].
}
\]
Completing $E$ and $F$ yields finite rank matrix factorisations $\widehat{E},\widehat{F}$ over $k\llbracket \bold{y}\rrbracket[\bold{x}]$ and $k\llbracket \bold{y}\rrbracket[\bold{z}]$ respectively, and $\omega^*(F \otimes_{k[\bold{y}]} E) = \widehat{F} \otimes_{k\llbracket \bold{y}\rrbracket} \widehat{E}$. It follows from Remark \ref{remark:completionandpush} that this completion commutes with the pushforward: the natural transformation $\pi_* \lto \rho_* \circ \omega^*$ yields a homotopy equivalence
\[
F \star E = \pi_*(F \otimes_{k[\bold{y}]} E) \lto \rho_*( \widehat{F} \otimes_{k\llbracket \bold{y}\rrbracket} \widehat{E} ) =: \widehat{F} \star \widehat{E}
\]
over $k[\bold{x},\bold{z}]$. We therefore obtain a description of $e$ in terms of an Atiyah class:

\begin{theorem}\label{theorem:defectfusion} There is a diagram in the homotopy category $\HMF(k[\bold{x},\bold{z}], U - W)$
\begin{equation}\label{eq:push_fuse1}
\xymatrix@+3pc{
F/\bold{t} F \otimes_{k[\bold{y}]} E/\bold{t} E \ar@<-0.8ex>[r]_(0.6){\psi} & (F \star E)[m] \ar@<-0.8ex>[l]_(0.4){\vartheta}
}
\end{equation}
in which $\psi \circ \vartheta = 1$ and the idempotent $e = \vartheta \circ \psi$ is given by the formula
\begin{equation}\label{eq:idempotent_for_fusion}
e = \frac{1}{m!} (-1)^{\binom{m+1}{2}} \cdot \partial_{y_1}(d_E) \cdots \partial_{y_m}(d_E) \cdot \At(\widehat{F} \otimes_{k\llbracket \bold{y}\rrbracket} \widehat{E})^m
\end{equation}
where the Atiyah class is relative to $k[\bold{x},\bold{z},\bold{t}]/k[\bold{x},\bold{z}].$
\end{theorem}
\begin{proof}
There is a standard flat $k$-linear connection $k\llbracket \bold{y}\rrbracket \lto k\llbracket \bold{y}\rrbracket \otimes_{k[\bold{t}]} \Omega^1_{k[\bold{t}]/k}$ by Appendix \ref{section:derhamsplit}, and this extends to a standard flat $k[\bold{x},\bold{z}]$-linear connection on $k\llbracket \bold{y} \rrbracket [\bold{x},\bold{z}]$ as a $k[\bold{x},\bold{z},\bold{t}]$-module. It follows from Theorem \ref{theorem:main} applied to the morphism $\rho: k[\bold{x},\bold{z}] \lto k\llbracket \bold{y} \rrbracket[\bold{x},\bold{z}]$ and the factorisation $\widehat{F} \otimes_{k\llbracket \bold{y} \rrbracket} \widehat{E}$, with the homotopies $\lambda_i$ as above, that there is a diagram
\[
\xymatrix@+3pc{
(\widehat{F} \otimes_{k\llbracket \bold{y}\rrbracket} \widehat{E}) \otimes_{k\llbracket \bold{y}\rrbracket} J_V \ar@<-0.8ex>[r]_(0.6){\psi_1} & (\widehat{F} \star \widehat{E})[m] \ar@<-0.8ex>[l]_(0.4){\vartheta_1}
}
\]
in $\HMF(k[\bold{x},\bold{z}], U - W)$ with $\psi_1 \circ \vartheta_1 = 1$ and the idempotent $\vartheta_1 \circ \psi_1$ given by the right hand side of (\ref{eq:idempotent_for_fusion}). But by Remark \ref{remark:completionandpush} this diagram is homotopy equivalent to (\ref{eq:push_fuse1}), completing the proof.
\end{proof}

Evaluating the kernel functor $\Phi_F$ on an object is a special case of convolution: if we take $W = 0$ and the set of variables $\bold{x}$ to be empty, then for an object $E \in \hmf(k[\bold{y}], V)$ we have $\Phi_F(E) = F \star E$. Hence if $J_V$ is finite-dimensional as above, we infer from the theorem that $\Phi_F(E)$ is a direct summand in $\HMF(k[\bold{z}], U)$ of a finite rank matrix factorisation, namely, the $m$-th suspension of $F/\bold{t} F \otimes_{k[\bold{y}]} E/\bold{t} E$.

\begin{remark} By Corollary \ref{cor:simplified_e} we can rewrite the idempotent as
\[
e = \frac{1}{m!} (-1)^{\binom{m}{2}} \sum_{\tau \in S_n} \sgn(\tau) \cdot \partial_{y_1}(d_E) \cdots \partial_{y_m}(d_E) \cdot [\partial_{\tau(1)}, d_F + d_E] \cdots [\partial_{\tau(n)}, d_F + d_E]
\]
where $\partial_i = \partial/\partial t_i$ is the operator on $\widehat{F} \otimes_{k\llbracket \bold{y} \rrbracket} \widehat{E}$ induced by the connection on $k\llbracket \bold{y} \rrbracket[\bold{x},\bold{z}]$ and $d_F + d_E$ denotes the differential. The commutator $[\partial_i, d_F + d_E]$ is $k[\bold{x},\bold{z},\bold{t}]$-linear, and passes to a $k[\bold{x},\bold{z}]$-linear operator on $(\widehat{F} \otimes_{k\llbracket \bold{y} \rrbracket} \widehat{E}) \otimes_{k\llbracket \bold{y} \rrbracket} J_V \cong F/\bold{t} F \otimes_{k[\bold{y}]} E/\bold{t} E$.
\end{remark}

\begin{remark}\label{remark:homogeneous_get_connection} If $V$ is homogeneous, then $\bold{t}$ is a homogeneous system of parameters and consequently $k[\bold{x}]$ is free of finite rank over $k[\bold{t}]$ (see Remark \ref{example:globalringnabla}). Hence a standard flat $k$-linear connection as in (\ref{eq:nablaboldy}) exists, and we do not have to pass to the completion; by this we mean that the statement of the theorem holds with $\At(\widehat{F} \otimes_{k\llbracket \bold{y}\rrbracket} \widehat{E})$ replaced by $\At(F \otimes_{k[\bold{y}]} E)$ in the formula for $e$.
\end{remark}

\begin{remark} In the notation of the theorem:
\begin{itemize}
\item[(i)] Taking the partial derivatives of $d_F^2 = U - V$ we learn that $-\partial_{y_i}(d_F)$ gives a null-homotopy for the action of $\partial_{y_i} V$, and replacing $\partial_{y_i}(d_E)$ by $-\partial_{y_i}(d_F)$ in (\ref{eq:idempotent_for_fusion}) we obtain another idempotent splitting to give $(F \star E)[m]$.
\item[(ii)] It follows from remarks in Section \ref{section:koszulcpxs} that there is a homotopy equivalence over $k[\bold{x},\bold{z}]$ between $F/\bold{t} F \otimes_{k[\bold{y}]} E/\bold{t} E$ and $((F \star E) \oplus (F \star E)[1])^{\oplus 2^{m-1}}$.
\end{itemize}
\end{remark}

\subsection{Chern character of the convolution}

Next we compute the Chern character of $F \star E$. Observe that the idempotent $e$ on $F/\bold{t} F \otimes_{k[\bold{y}]} E/\bold{t} E$ need not split in the homotopy category of finite rank matrix factorisations over $k[\bold{x},\bold{z}]$. Passing to $k\llbracket \bold{x},\bold{z}\rrbracket$ we can split $e$ and derive a formula for the Chern character. To this end, let us now assume that the critical loci of $U$ and $W$ are equal to the origin in their respective affine spaces, i.e. that the Jacobi algebras
\begin{align*}
J_W &= k[\bold{x}]/(\partial_{x_1} W, \ldots, \partial_{x_n} W),\\
J_U &= k[\bold{z}]/(\partial_{z_1} U, \ldots, \partial_{z_p} U)
\end{align*}
are finite-dimensional over $k$. The kernels $E$ and $F$ have Chern characters
\begin{align*}
\ch(E) &= (-1)^{\binom{n+m}{2}} \str_{k[\bold{x},\bold{y}]}\big( \partial_{\bold{x}}(d_E) \cdot \partial_{\bold{y}}(d_E) \big),\\
\ch(F) &= (-1)^{\binom{m+p}{2}} \str_{k[\bold{y},\bold{z}]}\big( \partial_{\bold{y}}(d_F) \cdot \partial_{\bold{z}}(d_F) \big),
\end{align*}
where the abbreviations are hopefully self-explanatory: for example, $\partial_{\bold{x}}(d_E) = \partial_{x_1}(d_E) \cdots \partial_{x_n}(d_E)$. Strictly speaking the Chern character $\ch(E)$ should be viewed as an element of the Jacobi algebra $J_W \otimes_k J_V$, and similarly $\ch(F)$ is defined as an element of $J_V \otimes_k J_U$.

\begin{corollary}\label{corollary:chern_char_convolution} The Chern character of the convolution is
\[
\ch( F \star E ) = (-1)^{\binom{m}{2}} \Ress{k[\bold{x},\bold{y},\bold{z}]/k[\bold{x},\bold{z}]}\begin{bmatrix} \ch(F) \cdot \ch(E) \cdot \ud \bold{y} \\ \partial_{y_1} V, \ldots, \partial_{y_m} V \end{bmatrix},
\]
where this equality is understood as elements of $J_W \otimes_k J_U$.
\end{corollary}
\begin{proof}
Here $\ch(F \star E)$ denotes the Chern character of an object of $\hmf(k\llbracket \bold{x},\bold{z}\rrbracket, U - W)$ homotopy equivalent to the completion $X = (F \star E) \otimes_{k[\bold{x},\bold{z}]} k\llbracket \bold{x},\bold{z}\rrbracket$. Such an object exists because this homotopy category is idempotent complete, and by the previous theorem $X$ is a direct summand in $\HMF$ of
\[
Y = (F/\bold{t} F \otimes_{k[\bold{y}]} E/\bold{t} E) \otimes_{k[\bold{x},\bold{z}]} k\llbracket \bold{x},\bold{z}\rrbracket.
\]
If $Z$ is a finite rank matrix factorisation splitting $e$ then $X$ is homotopy equivalent to $Z[m]$, and by the argument given in the proof of Proposition \ref{prop:chern_char_psh} we have 
\begin{align*}
\ch(F \star E) &= \ch( Z[m] )\\
&= (-1)^{m+\binom{n+p}{2}} \str\big( \partial_{\bold{x}}(d_Y) \partial_{\bold{z}}(d_Y) e \big)\\
&= \frac{1}{m!} (-1)^{\binom{m}{2} + \binom{n+p}{2}} \str\big( \partial_{\bold{x}}(d_Y) \partial_{\bold{z}}(d_Y) \partial_{\bold{y}}(d_E) \At(\widehat{F} \otimes_{k\llbracket \bold{y}\rrbracket} \widehat{E})^m \big)\\
&= \frac{1}{m!} (-1)^{\binom{m+1}{2} + \binom{n+p}{2}} \Ress{k[\bold{x},\bold{y},\bold{z}]/k[\bold{x},\bold{z}]} \begin{bmatrix} \str\big( \partial_{\bold{x}}(d ) \partial_{\bold{z}}(d) \partial_{\bold{y}}(d_E) \cdot \ud_{k[\bold{y}]/k}( d )^{\wedge m}\big) \ud \bold{y} \\ \partial_{y_1} V, \ldots, \partial_{y_m} V \end{bmatrix}.
\end{align*}
In the last line $d$ stands for the differential $d_F + d_E$ on $F \otimes_{k[\bold{y}]} E$ and the supertrace is of $k[\bold{x},\bold{y},\bold{z}]$-linear maps on this module. We have implicitly used the functoriality of the residue symbol (see \cite[Proposition 2.3]{Lipman87}) to go from a residue relative to $k\llbracket \bold{y}\rrbracket[\bold{x},\bold{z}]/k[\bold{x},\bold{z}]$, which is what Proposition \ref{prop:chern_char_psh} provides us with, to a residue relative to $k[\bold{x},\bold{y},\bold{z}]/k[\bold{x},\bold{z}]$. Clearly
\[
\partial_{\bold{x}}(d ) \partial_{\bold{z}}(d) \partial_{\bold{y}}(d_E) = (-1)^{np} \partial_{\bold{z}}(d_F) \otimes \partial_{\bold{x}}(d_E) \partial_{\bold{y}}(d_E).
\]
and
\[
\ud_{k[\bold{y}]/k}( d ) = \ud_{k[\bold{y}]/k}( d_F + d_E ) = \sum_{i=1}^m( \partial_{y_i}(d_F) + \partial_{y_i}(d_E) ) \ud y_i
\]
so the expression for $\ch(F \star E)$ expands as a linear combination of terms of the following form, for $\{i_1,\ldots,i_a\}$ a subset of $\{1,\ldots,m\}$ and $\gamma \in k[\bold{x},\bold{y}]$ a supertrace over $E$ which we suppress
\begin{equation}\label{eq:chern_convole1}
\Ress{k[\bold{x},\bold{y},\bold{z}]/k[\bold{x},\bold{z}]} \begin{bmatrix} \str\big(\partial_{\bold{z}}(d_F) \partial_{y_{i_1}}(d_F) \cdots \partial_{y_{i_a}}(d_F)\big) \cdot \gamma \cdot \ud \bold{y} \\ \partial_{y_1} V, \ldots, \partial_{y_m} V \end{bmatrix}.
\end{equation}
We claim that such a residue vanishes in $J_U \otimes_k J_W$ unless $a = m$. Using the nondegenerate residue pairing $\langle -,- \rangle$ on $J_W \otimes_k J_U$, it is enough to prove that $g$ is orthogonal to (\ref{eq:chern_convole1}) for every $g \in k[\bold{x},\bold{z}]$. But this particular pairing is
\begin{align*}
\Ress{k[\bold{x},\bold{z}]/k} \begin{bmatrix} g \cdot \Ress{k[\bold{x},\bold{y},\bold{z}]/k[\bold{x},\bold{z}]} \begin{bmatrix} \str\big(\partial_{\bold{z}}(d_F) \partial_{y_{i_1}}(d_F) \cdots \partial_{y_{i_a}}(d_F)\big) \cdot \gamma \cdot \ud \bold{y} \\ \partial_{y_1} V, \ldots, \partial_{y_m} V \end{bmatrix} \cdot \ud \bold{x} \wedge \ud \bold{z} \\ \partial_{x_1} W, \ldots, \partial_{x_n} W, \partial_{z_1} U, \ldots, \partial_{z_p} U\end{bmatrix}.
\end{align*}
Here we need to apply transitivity of the residue symbol \cite{lipman92} (see also \cite[p.244]{Conrad00}) which tells us that the above residue is equal to
\[
\Ress{k[\bold{x},\bold{y},\bold{z}]/k} \begin{bmatrix} g \cdot \str\big(\partial_{\bold{z}}(d_F) \partial_{y_{i_1}}(d_F) \cdots \partial_{y_{i_a}}(d_F)\big) \cdot \gamma \cdot \ud \bold{y}\wedge \ud \bold{x} \wedge \ud \bold{z} \;/\; \partial_{\bold{y}} V, \partial_{\bold{x}} W, \partial_{\bold{z}} U \end{bmatrix}
\]
with the obvious abbreviations. By the trick of Lemma \ref{lemma:nonfullstrvanishes} this vanishes unless $a = m$, whence
\[
\ch(F \star E) = \frac{1}{m!} \cdot A \cdot \Res \begin{bmatrix} \sum_{\tau \in S_m} \sgn(\tau) \str( \partial_{\bold{y}_\tau}(d_F) \partial_{\bold{z}}(d_F) ) \cdot \str( \partial_{\bold{x}}(d_E) \partial_{\bold{y}}(d_E) ) \ud \bold{y} \\ \partial_{y_1} V, \ldots, \partial_{y_m} V \end{bmatrix}
\]
where $A = (-1)^{\binom{m}{2} + \binom{m+p}{2} + \binom{n+m}{2}}$ and $\partial_{\bold{y}_\tau}(d_F) = \partial_{y_{\tau(1)}}(d_F) \cdots \partial_{y_{\tau(m)}}(d_F)$. Using again the transitivity of residues and the technique of Lemma \ref{eq:lemmaordchern} one checks that each of these permutations, with their signs, contributes equally to the sum, which is therefore equal to
\[
= (-1)^{\binom{m}{2} + \binom{m+p}{2} + \binom{n+m}{2}} \cdot \Res \begin{bmatrix} \str( \partial_{\bold{y}}(d_F) \partial_{\bold{z}}(d_F) ) \cdot \str( \partial_{\bold{x}}(d_E) \partial_{\bold{y}}(d_E) ) \ud \bold{y} \\ \partial_{y_1} V, \ldots, \partial_{y_m} V \end{bmatrix},
\]
and this is what we wanted to show.
\end{proof}

\begin{remark} We remark that the result in this section can also be deduced directly from the non-commutative
Riemann-Roch formalism developped in \cite{shklyarov} in combination with the
explicit formula for the Chern character and Mukai pairing given in
\cite{polishchuk}. The proof is similar to the arguments in the
proof of \cite[Theorem 3.1.4]{shklyarov}. The actual statement of \emph{loc.
cit.} corresponds to the case $n = 0$, so that $E$ is a finite rank matrix
factorisation of $V$ over $k[\bold{y}]$ and $F \star E$ is the image of $E$ under the functor $\Phi_F$. The corollary states that the diagram
\[
\xymatrix@C+4pc{
K_0( \hmf(k[\bold{y}], V )) \ar[d]_{\ch} \ar[r]^{\Phi_F} & K_0( \hmf(k[\bold{z}],U) ) \ar[d]^{\ch}\\
J_V \ar[r]_{(-1)^{\binom{m}{2}} \Res\big[ \ch(F) \cdot (-) \cdot \ud \bold{y} \;/\; \partial_{\bold{y}} V \big]} & J_U
}
\]
commutes.
\end{remark}

\subsection{Convolution over power series rings}\label{section:convolution_powerseries}

We begin with some remarks on completed tensor products, for details see \cite[Chapitre $0$,~\S~7.7]{EGA1}. Let $A \lto R$ and $A \lto S$ be local morphisms of local noetherian rings which induce isomorphisms between residue fields. Then $\mf{m} = \mf{m}_R \otimes_A S + R \otimes_A \mf{m}_S$ is a maximal ideal in $R \otimes_A S$ and we define the \emph{completed tensor product} $R \cotimes{A} S$ to be the $\mf{m}$-adic completion of $R \otimes_A S$. If $M$ is an $R$-module and $N$ an $S$-module then $M \cotimes{A} N$ denotes the $\mf{m}$-adic completion of $M \otimes_A N$.

Let $k$ be a field. If $R,S$ are noetherian $k$-algebras (not necessarily local) with respective maximal ideals $\mf{m}_R,\mf{m}_S$ such that $k \lto R/\mf{m}_R$ and $k \lto S/\mf{m}_S$ are bijective then $\mf{m} = \mf{m}_R \otimes S + R \otimes \mf{m}_S$ is a maximal ideal in $R \otimes_k S$ and there is a unique continuous morphism of $k$-algebras $\alpha$ making the diagram (in which the vertical completions are $\mf{m}$-adic)
\[
\xymatrix@C+1pc{
R \otimes_k S \ar[d]_{\can} \ar[r]^{\can} & \widehat{R} \otimes_k \widehat{S} \ar[d]^{\can}\\
(R \otimes_k S)\;\widehat{} \ar[r]_{\alpha} & \widehat{R} \cotimes{k} \widehat{S}
}
\]
commute, and this $\alpha$ is an isomorphism. In particular, for sets of variables $\bold{x},\bold{y}$ we have a canonical isomorphism of $k$-algebras $k\llbracket \bold{x}\rrbracket \cotimes{k} k\llbracket \bold{y}\rrbracket \cong k\llbracket \bold{x},\bold{y}\rrbracket$ and the ring morphism $\pi_{\bold{x}}: k\llbracket \bold{x}\rrbracket \lto k\llbracket \bold{x},\bold{y}\rrbracket$ factors as the following composite
\[
k\llbracket \bold{x}\rrbracket \lto k\llbracket \bold{x}\rrbracket \otimes_k k\llbracket \bold{y}\rrbracket \lto k\llbracket \bold{x}\rrbracket \cotimes{k} k\llbracket \bold{y}\rrbracket \cong k\llbracket \bold{x},\bold{y}\rrbracket.
\]

\begin{definition} We define the functor $\pi_{\bold{x}}^\bullet: \Modd k\llbracket \bold{x}\rrbracket \lto \Modd k\llbracket \bold{x},\bold{y}\rrbracket$ by extending scalars to $k\llbracket \bold{x}\rrbracket \otimes_k k\llbracket \bold{y}\rrbracket$ and then completing in the adic topology for the ideal $(\bold{x}) \otimes k\llbracket \bold{y}\rrbracket + k\llbracket \bold{x}\rrbracket \otimes (\bold{y})$,
\[
\pi_{\bold{x}}^\bullet(M) = M \cotimes{k} k\llbracket \bold{y}\rrbracket.
\]
\end{definition}

Notice that on finitely presented modules $\pi_{\bold{x}}^\bullet$ is canonically isomorphic to the extension of scalars functor $\pi_{\bold{x}}^*$, but in what follows we will need to also consider infinite modules. From now on $k$ denotes a field of characteristic zero and $\bold{x} = \{ x_1,\ldots,x_n \}, \bold{y} = \{ y_1, \ldots, y_m \}, \bold{z} = \{ z_1,\ldots,z_p \}$. We can draw an analogue of (\ref{eq:diagram_ring_maps}) and we use the same notation for the $\pi$'s, so for example $\pi_{\bold{y}\bold{z}}$ will denote some canonical map out of $k\llbracket \bold{y},\bold{z}\rrbracket$ (to a ring which should be clear from the context).

\begin{definition}\label{defn:kernel_functor_complete} Given $W \in k\llbracket \bold{x}\rrbracket, V \in k\llbracket \bold{y}\rrbracket$ and a linear factorisation $E$ of $V - W$ over $k\llbracket \bold{x},\bold{y}\rrbracket$ there is a $k$-linear functor
\[
\xymatrix@C+1pc{
\hf(k\llbracket \bold{x}\rrbracket,W) \ar[r]^(0.45){\pi^\bullet_{\bold{x}}} & \hf(k\llbracket \bold{x},\bold{y}\rrbracket, W) \ar[r]^{E \otimes -} & \hf(k\llbracket \bold{x},\bold{y}\rrbracket, V) \ar[r]^(0.55){(\pi_{\bold{y}})_*} & \hf(k\llbracket \bold{y}\rrbracket,V)
}
\]
which we denote by $\Psi_E$. If $U \in k\llbracket \bold{z}\rrbracket$ and $F$ is a linear factorisation of $U - V$ over $k\llbracket \bold{y},\bold{z}\rrbracket$ then we define the \emph{external tensor product} and \emph{convolution} of $F$ and $E$ respectively by 
\[
F \boxtimes E = \pi_{\bold{x}\bold{z}}^\bullet(F) \otimes_{k\llbracket \bold{x},\bold{y},\bold{z}\rrbracket} \pi_{\bold{y}\bold{z}}^\bullet(E), \qquad F \star E = (\pi_{\bold{x}\bold{z}})_*(F \boxtimes E).
\]
These are linear factorisations of $U - W$ over $k\llbracket \bold{x},\bold{y},\bold{z}\rrbracket$ and $k\llbracket \bold{x},\bold{z}\rrbracket$ respectively.
\end{definition}

It follows from the next lemma that $F \boxtimes E \cong F \cotimes{k\llbracket \bold{y}\rrbracket} E$ if both $E$ and $F$ are finitely generated.

\begin{lemma}\label{lemma:two_bifunc_completions} If $M$ is a finitely generated $k\llbracket \bold{x},\bold{y}\rrbracket$-module and $N$ is a finitely generated $k\llbracket \bold{y},\bold{z}\rrbracket$-module then there is a natural isomorphism of $k\llbracket \bold{x},\bold{y},\bold{z}\rrbracket$-modules
\begin{equation}\label{eq:two_bifunc_compl}
M \otimes_{k\llbracket \bold{x},\bold{y}\rrbracket} k\llbracket \bold{x},\bold{y},\bold{z}\rrbracket \otimes_{k\llbracket \bold{y},\bold{z}\rrbracket} N \cong M \cotimes{k\llbracket \bold{y}\rrbracket} N
\end{equation}
where the completion is with respect to the topology defined by $(\bold{x},\bold{y}) \otimes k\llbracket \bold{y},\bold{z}\rrbracket + k\llbracket \bold{x},\bold{y}\rrbracket \otimes (\bold{y},\bold{z})$.
\end{lemma}
\begin{proof}
There is a natural morphism
\begin{align*}
\gamma: M \otimes_{k\llbracket \bold{x},\bold{y}\rrbracket} k\llbracket \bold{x},\bold{y},\bold{z}\rrbracket \otimes_{k\llbracket \bold{y},\bold{z}\rrbracket} N &= M \otimes_{k\llbracket \bold{x},\bold{y}\rrbracket} \varprojlim_j \left\{ \frac{k\llbracket \bold{x},\bold{y}\rrbracket}{(\bold{x},\bold{y})^j} \otimes_{k\llbracket \bold{y}\rrbracket} \frac{k\llbracket \bold{y},\bold{z}\rrbracket}{(\bold{y},\bold{z})^j}\right\} \otimes_{k\llbracket \bold{y},\bold{z}\rrbracket} N\\
&\lto \varprojlim_j \left\{ \frac{M}{(\bold{x},\bold{y})^jM} \otimes_{k\llbracket \bold{y}\rrbracket} \frac{N}{(\bold{y},\bold{z})^j N}\right\}
\end{align*}
which is an isomorphism if $M, N$ are finite rank free modules over their respective rings. Since both bifunctors in (\ref{eq:two_bifunc_compl}) are right exact in each variable, we infer that $\gamma$ is an isomorphism when both $M$ and $N$ are finitely generated.
\end{proof}

To prove that the convolution represents $\Psi_F \circ \Psi_E$ we need a technical lemma.

\begin{lemma} Consider the diagram of functors
\[
\xymatrix@C+1pc{
\Modd k\llbracket \bold{x},\bold{y}\rrbracket \ar[d]_{\pi_{\bold{x}\bold{y}}^\bullet} \ar[r]^{(\pi_{\bold{y}})_*} & \Modd k\llbracket \bold{y}\rrbracket \ar[d]^{ \pi_{\bold{y}}^\bullet}\\
\Modd k\llbracket \bold{x},\bold{y},\bold{z}\rrbracket \ar[r]_{(\pi_{\bold{y}\bold{z}})_*} & \Modd k\llbracket \bold{y},\bold{z}\rrbracket.
}
\]
There is a natural transformation $\theta: \pi_{\bold{y}}^\bullet \circ (\pi_{\bold{y}})_* \lto (\pi_{\bold{y}\bold{z}})_* \circ \pi_{\bold{x}\bold{y}}^\bullet$ with the property that $\theta_M$ is an isomorphism whenever $M$ is a $k\llbracket \bold{x},\bold{y}\rrbracket$-module separated and complete in its $(\bold{x},\bold{y})$-adic topology.
\end{lemma}
\begin{proof}
For a $k\llbracket \bold{x},\bold{y}\rrbracket$-module $M$ there is an isomorphism
\begin{align*}
\pi_{\bold{y}}^\bullet (\pi_{\bold{y}})_*(M) &\cong \varprojlim_j \left\{ \frac{M \otimes_k k\llbracket \bold{z}\rrbracket}{(\bold{y})^j \otimes k\llbracket \bold{z}\rrbracket + M \otimes (\bold{z})^j} \right\}\\
&\cong\varprojlim_j \big\{ M/(\bold{y})^j M \otimes_k k[\bold{z}]/(\bold{z})^j \big\}\\
&\cong \varprojlim_j \big\{ ( \varprojlim_\ell M/(\bold{y})^\ell M ) \otimes_k k[\bold{z}]/(\bold{z})^j \big\}
\end{align*}
and similarly
\[
(\pi_{\bold{y}\bold{z}})_* \pi_{\bold{x}\bold{y}}^\bullet(M) \cong \varprojlim_j \big\{ ( \varprojlim_\ell M/(\bold{x},\bold{y})^\ell M ) \otimes_k k[\bold{z}]/(\bold{z})^j \big\},
\]
so the maps $M/(\bold{y})^j M \lto M/(\bold{x},\bold{y})^j M$ give rise to the desired natural transformation $\theta$, which is an isomorphism if $M$ is separated and complete in the adic topology.
\end{proof}

We fix potentials $W \in k\llbracket \bold{x}\rrbracket, V \in k\llbracket \bold{y}\rrbracket$ and $U \in k\llbracket \bold{z}\rrbracket$ in what follows.

\begin{proposition} Given $E \in \hmf( k\llbracket \bold{x},\bold{y}\rrbracket, V - W)$ and $F \in \hmf( k\llbracket \bold{y},\bold{z}\rrbracket, U - V)$ the diagram
\[
\xymatrix@C+1pc{
\hmf(k\llbracket \bold{x}\rrbracket,W) \ar@/_2pc/[rr]_{\Psi_{F \star E}}\ar[r]^{\Psi_E} & \hf(k\llbracket \bold{y}\rrbracket,V) \ar[r]^{\Psi_F} & \hf(k\llbracket \bold{z}\rrbracket,U)
}
\]
commutes up to natural isomorphism.
\end{proposition}
\begin{proof}
For $X \in \hmf(k\llbracket \bold{x}\rrbracket, W)$ we have by the lemma a chain of natural isomorphisms
\begin{align*}
(\Psi_F \circ \Psi_E)(X) &\cong (\pi_{\bold{z}})_*\left(F \otimes \pi_{\bold{y}}^\bullet (\pi_{\bold{y}})_*(E \otimes \pi_{\bold{x}}^* X)\right)\\
&\cong (\pi_{\bold{z}})_*\left(F \otimes (\pi_{\bold{y}\bold{z}})_*\pi_{\bold{x}\bold{y}}^\bullet(E \otimes \pi_{\bold{x}}^* X)\right)\\
&\cong (\pi_{\bold{z}})_*\left(F \otimes (\pi_{\bold{y}\bold{z}})_*\pi_{\bold{x}\bold{y}}^*(E \otimes \pi_{\bold{x}}^* X)\right)\\
&\cong (\pi_{\bold{z}})_*\left( \pi_{\bold{y}\bold{z}}^* F \otimes \pi_{\bold{x}\bold{y}}^* E \otimes \pi_{\bold{x}}^* X \right)\\
&\cong (\pi_{\bold{z}})_*( (\pi_{\bold{x}\bold{z}})_*( F \boxtimes E ) \otimes \pi_{\bold{x}}^* X )\\
&= \Psi_{F \star E}(X),
\end{align*}
as claimed.
\end{proof}

We conclude by describing $F \star E$ as the splitting of an explicit idempotent. We keep the notation of the previous proposition, so that $E,F$ are finite rank kernels. We also write $\bold{t} = \{ \partial_{y_1} V, \ldots, \partial_{y_m} V \}$ and $J_V = k\llbracket \bold{y}\rrbracket/\bold{t} k\llbracket \bold{y}\rrbracket$.

\begin{theorem} Suppose that $\dim_k J_V < \infty$. Then there is a diagram
\begin{equation}\label{eq:idempotent_formula_diagram_power}
\xymatrix@+3pc{
F/\bold{t} F \cotimes{k\llbracket \bold{y}\rrbracket} E/\bold{t} E \ar@<-0.8ex>[r]_(0.6){\psi} & (F \star E)[m] \ar@<-0.8ex>[l]_(0.4){\vartheta}
}
\end{equation}
in $\HMF(k\llbracket \bold{x},\bold{z}\rrbracket, U - W)$, where $\psi \circ \vartheta = 1$ and the idempotent $e = \vartheta \circ \psi$ is given by the formula
\begin{equation}\label{eq:idempotent_formula_last_power}
e = \frac{1}{m!} (-1)^{\binom{m+1}{2}} \cdot \partial_{y_1}(d_E) \cdots \partial_{y_m}(d_E) \cdot \At(F \boxtimes E)^m
\end{equation}
where the Atiyah class is relative to $k\llbracket \bold{x},\bold{z}\rrbracket[\bold{t}]/k\llbracket \bold{x},\bold{z}\rrbracket$.
\end{theorem}
\begin{proof}
Observe that
\[
k\llbracket \bold{x},\bold{y},\bold{z}\rrbracket/\bold{t} k\llbracket \bold{x},\bold{y},\bold{z}\rrbracket \cong k\llbracket \bold{x},\bold{z}\rrbracket \otimes_k J_V
\]
is a finite rank free $k\llbracket \bold{x},\bold{z}\rrbracket$-module, each $\partial_{y_i} V$ acts null-homotopically on $F \boxtimes E$ with null-homotopy $\lambda_i = 1 \otimes \partial_{y_i}(d_E)$, and $k\llbracket \bold{x},\bold{y},\bold{z}\rrbracket$ admits a flat $k\llbracket \bold{x},\bold{z}\rrbracket$-linear connection as a $k\llbracket \bold{x},\bold{z}\rrbracket[\bold{t}]$-module by Appendix \ref{section:derhamsplit}. We can therefore apply Theorem \ref{theorem:main} to the morphism $\pi_{\bold{x}\bold{z}}: k\llbracket \bold{x},\bold{z}\rrbracket \lto k\llbracket \bold{x},\bold{y},\bold{z}\rrbracket$, the matrix factorisation $F \boxtimes E$ and the regular sequence $\bold{t}$. This gives us the diagram (\ref{eq:idempotent_formula_diagram_power}) and associated idempotent $e$, but defined on $(F \boxtimes E) \otimes_{k\llbracket \bold{y}\rrbracket} J_V$. To conclude the proof we recall that by Lemma \ref{lemma:two_bifunc_completions} this factorisation is isomorphic to $F/\bold{t} F \cotimes{k\llbracket \bold{y}\rrbracket} E/\bold{t} E$.
\end{proof}

\begin{remark} Suppose that $J_U, J_V$ and $J_W$ are finite-dimensional over $k$. Then $\hmf(k\llbracket \bold{x},\bold{z}\rrbracket, U - W)$ is idempotent complete and it follows from the theorem that $F \star E$ is homotopy equivalent to a finite rank matrix factorisation. We deduce a residue formula for the Chern character of the convolution, identical to the formula given in Theorem \ref{corollary:chern_char_convolution}. If we take $\bold{x}$ to be empty and $W = 0$, we learn that the integral functor $\Phi_F$ restricts to give a functor $\hmf(k\llbracket \bold{y}\rrbracket, V) \lto \hmf(k\llbracket \bold{z}\rrbracket, U)$ on the full subcategories of finite rank factorisations.

The fact that $F \star E$ is homotopy equivalent to a finite rank matrix factorisation is well-known both in the physics literature \cite[\S 4.2]{brunnerdefect} and in the mathematics literature, see either \cite[Prop. 13]{khovanov} or combine the results of \cite{dyck4} with \cite[Lemma 2.8]{toen.morita}. But in both cases the arguments are indirect; our presentation of $F \star E$ in terms of the finite rank matrix factorisation $F/\bold{t} F \otimes_{k[\bold{y}]} E/\bold{t} E$ with an explicit idempotent seems to be new.
\end{remark}

\section{Kn\"orrer periodicity}\label{section:knorrer}

A classical theorem of Kn\" orrer \cite{knorrer} states that for an algebraically closed field $k$ of characteristic zero, and a potential $W \in k\llbracket \bold{x}\rrbracket = k\llbracket x_1,\ldots,x_n\rrbracket$ with an isolated singularity, there is an equivalence
\begin{equation}\label{eq:knorrerequiv}
\Phi: \hmf(k\llbracket \bold{x}\rrbracket, W) \xlto{\cong} \hmf(k\llbracket \bold{x},u,v\rrbracket, W + uv)
\end{equation}
This phenomenon is known as \emph{Kn\"orrer periodicity}, and was generalised to schemes by Orlov \cite{Orlov04}. The functor $\Phi$ is defined by tensoring with the Koszul factorisation of $uv$, but the inverse $\Psi$ is more subtle and involves a pushforward. In order to have a concrete example of our main theorem, we construct $\Psi$ and define a natural isomorphism $\kappa: \Psi \circ \Phi \lto 1$ in terms of residues; this is in contrast to \cite{knorrer} which does not directly construct the inverse. The same construction of $\Psi$ appears in the work of Herbst-Hori-Page \cite[\S 3.4]{herbsthoripage}, but there the disposal of contractible summands in the pushforward is done in an \emph{ad hoc} fashion, whereas in our case $\kappa$ takes care of this canonically.

Let $S$ be a $\mathbb{Q}$-algebra and $W \in S$ a potential. Consider the Koszul matrix factorisation $K$ of $uv$ over $R = S\llbracket u,v\rrbracket$ with differential
\[
d_K = \begin{pmatrix} 0 & u \\ v & 0 \end{pmatrix}, \text{ and its dual } d_{K^{\lor}} = \begin{pmatrix} 0 & v \\ -u & 0 \end{pmatrix}.
\]
The tensor product defines a functor
\[
\Phi: \hmf(S,W) \lto \hmf(S\llbracket u,v\rrbracket, W + uv), \quad \Phi(X) = X \otimes_S K.
\]
We are interested in the inverse construction: given a finite rank matrix factorisation $Y$ of $W + uv$ over $S\llbracket u,v\rrbracket$ we tensor $Y$ with $K^{\lor}$ to yield a finite rank matrix factorisation of $W$, which we then pushforward along the ring morphism $\varphi: S \lto S\llbracket u,v\rrbracket$, to define the functor
\[
\Psi: \hmf(S\llbracket u,v\rrbracket, W + uv) \lto \HMF(S, W), \quad \Psi(Y) = \varphi_*(Y \otimes K^{\lor}).
\]
Now $\bold{t} = \{ u, v \}$ is a regular sequence in $S\llbracket u,v\rrbracket$ and $\nabla^0(r) = \partial_u(r) \otimes \ud u + \partial_v( r ) \otimes \ud v$ is a standard flat $S$-linear connection on $S\llbracket u,v \rrbracket$. Since in this case $R/\bold{t} R = S$ we may apply Theorem \ref{theorem:main} to see that in $\HMF(S,W)$ there is a diagram
\[
\xymatrix@+3pc{
(Y \otimes K^{\lor})|_{u=v=0} \ar@<-0.8ex>[r]_(0.6){\psi} & \Psi(Y) \ar@<-0.8ex>[l]_(0.4){\vartheta}
}
\]
where $(-)|_{u=v=0}$ denotes the functor $(-) \otimes_{S\llbracket u,v\rrbracket} S$, we have $\psi \circ \vartheta = 1$ and
\begin{align*}
e = \vartheta \circ \psi &= -\frac{1}{2} \partial_{v}(d_{K^{\lor}})\partial_u(d_{K^{\lor}}) \At_{S[u,v]/S}(Y \otimes K^{\lor})^2.
\end{align*}
Here we have used $d_{K^{\lor}}^2 = -uv$ and the Leibniz rule to produce null-homotopies $\lambda_1 = - 1 \otimes \partial_v(d_{K^{\lor}})$ and $\lambda_2 = -1 \otimes \partial_u(d_{K^{\lor}})$ for the action of $u$ and $v$. We have also implicitly fixed a homogeneous basis of $Y$ and $K$. As operators on $(Y \otimes K^{\lor})|_{u=v=0}$ we have
\begin{align*}
[\partial_u, d_{Y \otimes K^{\lor}}] &= \partial_u(d_Y)|_{u=v=0} \otimes 1_{K^{\lor}} + 1 \otimes \begin{pmatrix} 0 & 0 \\ - 1 & 0 \end{pmatrix}\\
[\partial_v, d_{Y \otimes K^{\lor}}] &= \partial_v(d_Y)|_{u=v=0} \otimes 1_{K^{\lor}} + 1 \otimes \begin{pmatrix} 0 & 1 \\ 0 & 0 \end{pmatrix}.
\end{align*}
Combining this with the simplified form of $e$ in Corollary \ref{cor:simplified_e} we find that
\[
e = \frac{1}{2} [ \partial_v(d_Y)|_{u=v=0}, \partial_u(d_Y)|_{u=v=0} ] \otimes \begin{pmatrix} -1 & 0 \\ 0 & 0 \end{pmatrix} + \frac{1}{2}1_Y \otimes \begin{pmatrix} 1 & 0 \\ 0 & 0 \end{pmatrix} - \partial_u(d_Y)|_{u=v=0} \otimes \begin{pmatrix} 0 & -1 \\ 0 & 0 \end{pmatrix}.
\]
So the pushforward $\varphi_*(Y \otimes K^{\lor})$ is the splitting of this idempotent in $\HMF(S,W)$. To prove that this construction gives a left inverse to $\Phi$, suppose that $Y = \Phi(X)$ for some $X \in \hmf(S,W)$. Then
\[
\partial_u(d_Y) = 1_X \otimes \begin{pmatrix} 0 & 1 \\ 0 & 0 \end{pmatrix}, \quad \partial_v(d_Y) = 1_X \otimes \begin{pmatrix} 0 & 0 \\ 1 & 0 \end{pmatrix},
\]
and substituting this in the above one finds that, as an endomorphism of $( X \otimes K \otimes K^{\lor} )|_{u=v=0}$,
\[
e = 1_X \otimes \left(\begin{smallmatrix} 1 & 0 & 0 &1 \\ 0 & 0 & 0 &0 \\ 0 & 0 &0 &0\\ 0 &0 &0 &0 \end{smallmatrix}\right)
\]
the splitting of which is obviously $X$, via the maps
\[
\xymatrix@+5pc{
(X \otimes K \otimes K^{\lor})|_{u=v=0} \ar@<-0.8ex>[r]_(0.6){f = \left(\begin{smallmatrix} 1 & 0 & 0 & 1 \end{smallmatrix}\right)} & X \ar@<-0.8ex>[l]_(0.4){g = \left(\begin{smallmatrix} 1 & 0 & 0 & 0 \end{smallmatrix}\right)^t}.
}
\]
Hence $\Psi \Phi(X) \cong X$. Note that this calculation serves as a good check on the signs and normalisations on $e$ in the main theorem, and that by composing $f$ with $\vartheta$ we have an explicit homotopy equivalence between $X$ and $\Psi \Phi(X)$ in $\HMF(S,W)$. To present this it is best to use the canonical isomorphism
\[
\Psi \Phi(X) = X \otimes K \otimes K^{\lor} \cong X \otimes K^{\lor} \otimes K \cong X \otimes \shom(K,K).
\]
Under this isomorphism $f$ corresponds to the supertrace, so taking $\kappa = f \circ \vartheta$ we have

\begin{proposition} There is a natural isomorphism in $\HMF(S,W)$
\[
\kappa: \Psi \Phi(X) = X \otimes \shom(K,K) \lto X,
\]
given for $x \in X$ and $\alpha \in \shom(K,K)$ homogeneous by
\[
\kappa( x \otimes \alpha ) = \Ress{S\llbracket u,v\rrbracket/S} \begin{bmatrix} \str( \alpha \cdot \partial_v(d_K) \partial_u(d_K) ) \cdot \ud u \wedge \ud v \\ u,v \end{bmatrix} \cdot x.
\]
\end{proposition}

\appendix

\section{Properties of the Chern character}\label{section:ordering_chernchar}

Let $k$ be a field of characteristic zero and set $R = k\llbracket x_1,\ldots,x_n\rrbracket$. Denoting by $\mf{m}$ the maximal~ideal in $R$, let $W \in \mf{m}^2$ be a polynomial defining an isolated singularity, by which we mean that the Jacobi algebra $J_W = R/(\partial_{x_1} W, \ldots, \partial_{x_n} W)$ is finite-dimensional. Recall that $J_W$ is a Frobenius algebra with the nondegenerate residue pairing of local duality
\[
\langle f, g \rangle = \Res \Big[ \begin{matrix} f g \cdot \ud \bold{x} \;/\; \partial_1 W,\ldots,\partial_n W \end{matrix} \Big]
\]
where we write $\ud \bold{x}$ for $\ud x_1 \wedge \cdots \wedge \ud x_n$ and $\partial_i = \partial/\partial x_i$. Let $(X,d_X)$ be a finite rank matrix factorisation of $W$ over $R$. The Chern character of $X$ is the supertrace
\begin{equation}\label{eq:chern_char_formula}
\ch(X) = (-1)^{\binom{n}{2}} \str( \partial_1 d_X \cdots \partial_n d_X )
\end{equation}
considered as an element of $J_W$. If $n$ is odd then this supertrace is zero and all the statements below are trivial, so let us assume that $n$ is even. We begin by clarifying the role of the odd matrices $\partial_i d_X$. Recall from Remark \ref{remark:partial_gives_homotopy} that these act as null-homotopies for the action of the $\partial_i W$, i.e.
\begin{equation}\label{eq:app_chern_char_partials}
\partial_i(d_X) \cdot d_X + d_X \cdot \partial_i(d_X) = \partial_i W \cdot 1_X.
\end{equation}
In fact any family of null-homotopies will serve just as well in the definition of the Chern character.

\begin{lemma}\label{lemma:indept_homotopy_chern} Let $\{ \lambda_i \}_{1\le i \le n}$ be a family of $R$-linear homotopies on $X$ with $d_X \cdot \lambda_i + \lambda_i \cdot d_X = \partial_i W \cdot 1_X$ for $1 \le i \le n$. Then we have $\str( \lambda_1 \cdots \lambda_n ) = \str( \partial_1 d_X \cdots \partial_n d_X )$ as elements of $J_W$.
\end{lemma}
\begin{proof}
By induction it suffices to prove that interchanging $\partial_1 d_X$ and $\lambda_1$ does not change the equivalence class of the supertrace. We do this by proving that both supertraces have the same pairing with every $g \in J_W$, so that nondegeneracy of the pairing gives the desired result.

The denominators of our residues will always be either $\partial_1 W$ or $(\partial_1 W)^2$ followed by $\partial_2 W, \ldots, \partial_n W$, and we drop the latter $n-1$ terms from the notation. If some $\partial_i W$ appears in the denominator of a residue then we can apply (\ref{eq:app_chern_char_partials}) to see that $d_X$ and $\partial_i d_X$ anticommute inside the residue. Hence
\begin{align*}
\langle \str( \partial_1 d_X \cdots \partial_n d_X ), g \rangle &= \Res \Big[ \begin{matrix} g \cdot \str( \partial_1 d_X \cdots \partial_n d_X ) \cdot \ud \bold{x} \;/\; \partial_1 W,\ldots \end{matrix} \Big]\\
&= \Res \Big[ \begin{matrix} g \cdot \str( \partial_1 W \cdot \partial_1 d_X \cdots \partial_n d_X ) \cdot \ud \bold{x} \;/\; (\partial_1 W)^2, \ldots \end{matrix} \Big]\\
&= \Res \Big[ \begin{matrix} g \cdot \str( (\lambda_1 d_X + d_X \lambda_1 ) \partial_1 d_X \cdots \partial_n d_X ) \cdot \ud \bold{x} \;/\; (\partial_1 W)^2, \ldots \end{matrix} \Big]
\end{align*}
and after anticommuting the $d_X$ in $d_X \lambda_1$ around the supertrace and through each $\partial_i d_X$ with $i \neq 1$,
\begin{align*}
\quad &= \Res \Big[ \begin{matrix} g \cdot \str( \lambda_1 \cdot d_X \cdot \partial_1 d_X \cdot \partial_2 d_X \cdots \partial_n d_X ) \cdot \ud \bold{x} \;/\; (\partial_1 W)^2, \ldots \end{matrix} \Big]\\
&+ (-1)^n \Res \Big[ \begin{matrix} g \cdot \str( \lambda_1 \cdot \partial_1 d_X \cdot d_X \cdot \partial_2 d_X \cdots \partial_n d_X ) \cdot \ud \bold{x} \;/\; (\partial_1 W)^2, \ldots \end{matrix} \Big]\\
&= \Res \Big[ \begin{matrix} g \cdot \str( \lambda_1 ( d_X \cdot \partial_1 d_X + \partial_1 d_X \cdot d_X ) \partial_2 d_X \cdots \partial_n d_X ) \cdot \ud \bold{x} \;/\; (\partial_1 W)^2, \ldots \end{matrix} \Big]\\
&= \Res \Big[ \begin{matrix} g \cdot \str( \lambda_1 \partial_2 d_X \cdots \partial_n d_X ) \cdot \ud \bold{x} \;/\; \partial_1 W, \ldots \end{matrix} \Big]\\
&= \langle \str( \lambda_1 \partial_2 d_X \cdots \partial_n d_X ), g \rangle
\end{align*}
which is what we needed to show.
\end{proof}

From now on $\{ \lambda_i \}_{1\le i \le n}$ denotes a family of $R$-linear homotopies with $d_X \cdot \lambda_i + \lambda_i \cdot d_X = \partial_i W \cdot 1_X$ for $1 \le i \le n$. It follows indirectly from the derivations of the Chern character in \cite{polishchuk} and \cite{dyckmurf} that the matrices $\partial_i d_X$ anticommute in (\ref{eq:chern_char_formula}). Let us give a more direct proof.

\begin{lemma}\label{eq:lemmaordchern} For $\sigma \in S_n$ we have
\[
\str( \lambda_{\sigma(1)} \cdots \lambda_{\sigma(n)} ) = \sgn(\sigma) \cdot \str( \lambda_1 \cdots \lambda_n )
\]
as elements of $J_W$.
\end{lemma}
\begin{proof}
It is enough to prove that in $J_W$, $\str( \lambda_1 \lambda_2 \lambda_3 \cdots \lambda_n ) = - \str( \lambda_2 \lambda_1 \lambda_3 \cdots \lambda_n )$ which we do using the technique of the previous lemma:
\begin{align*}
\langle \str( \lambda_1 \lambda_2 \lambda_3 \cdots \lambda_n ), g \rangle &= \Res \Big[ \begin{matrix} g \cdot \str( \lambda_1 \lambda_2 \lambda_3 \cdots \lambda_n ) \cdot \ud \bold{x} \;/\; \partial_1 W,\ldots \end{matrix} \Big]\\
&= \Res \Big[ \begin{matrix} g \cdot \str( \lambda_1 \lambda_2 \cdot \partial_1 W \cdot \lambda_3 \cdots \lambda_n ) \cdot \ud \bold{x} \;/\; (\partial_1 W)^2, \ldots \end{matrix} \Big]\\
&= \Res \Big[ \begin{matrix} g \cdot \str( \lambda_1 \lambda_2 ( \lambda_1 d_X + d_X \lambda_1 )  \lambda_3 \cdots \lambda_n ) \cdot \ud \bold{x} \;/\; (\partial_1 W)^2, \ldots \end{matrix} \Big].
\end{align*}
Anticommuting the $d_X$ in $\lambda_1 d_X$ to the right through various $\lambda_i$ with $i \neq 1$ and then cyclically to the start of the supertrace, and the $d_X$ in $d_X \lambda_1$ one spot to the left, this is
\begin{align*}
\quad &= - \Res \Big[ \begin{matrix} g \cdot \str( d_X \lambda_1 \lambda_2 \lambda_1 \lambda_3 \cdots \lambda_n ) \cdot \ud \bold{x} \;/\; (\partial_1 W)^2, \ldots \end{matrix} \Big]\\
&- \Res \Big[ \begin{matrix} g \cdot \str( \lambda_1 d_X \lambda_2 \lambda_1 \lambda_3 \cdots \lambda_n ) \cdot \ud \bold{x} \;/\; (\partial_1 W)^2, \ldots \end{matrix} \Big]\\
&= - \Res \Big[ \begin{matrix} g \cdot \str( (d_X \lambda_1 + \lambda_1 d_X) \cdot \lambda_2 \lambda_1 \lambda_3 \cdots \lambda_n ) \cdot \ud \bold{x} \;/\; (\partial_1 W)^2, \ldots \end{matrix} \Big]\\
&= - \Res \Big[ \begin{matrix} g \cdot \str( \lambda_2 \lambda_1 \lambda_3 \cdots \lambda_n ) \cdot \ud \bold{x} \;/\; \partial_1 W, \ldots \end{matrix} \Big]\\
&= - \langle \str( \lambda_2 \lambda_1 \lambda_3 \cdots \lambda_n ), g \rangle
\end{align*}
as claimed.
\end{proof}

The supertrace of a product of null-homotopies which is not ``full'' must be zero.

\begin{lemma}\label{lemma:nonfullstrvanishes} Let $\{ i_1, \ldots, i_p \}$ be a proper subset of $\{ 1, \ldots, n \}$. Then $\str( \lambda_{i_1} \cdots \lambda_{i_p} ) = 0$ in $J_W$.
\end{lemma}
\begin{proof}
Without loss of generality assume $1 \notin \{ i_1, \ldots, i_p \}$. Given $g \in J_W$ we have
\begin{align*}
\langle \str(\lambda_{i_1} \cdots \lambda_{i_p}), g\rangle &= \Res \Big[ \begin{matrix} g \cdot \str( \lambda_{i_1} \cdots \lambda_{i_p} ) \cdot \ud \bold{x} \;/\; \partial_1 W, \ldots \end{matrix} \Big]\\
&= \Res \Big[ \begin{matrix} g \cdot \str( \lambda_{i_1} \cdots \lambda_{i_p} \cdot \partial_1 W) \cdot \ud \bold{x} \;/\; (\partial_1 W)^2, \ldots \end{matrix} \Big]\\
&= \Res \Big[ \begin{matrix} g \cdot \str( \lambda_{i_1} \cdots \lambda_{i_p} (\lambda_1 d_X + d_X \lambda_1) ) \cdot \ud \bold{x} \;/\; (\partial_1 W)^2, \ldots \end{matrix} \Big],
\end{align*}
and if we commute the $d_X$ in $\lambda_1 d_X$ around the supertrace in the fashion of the proof of the previous lemma (using the fact that $\lambda_1$ does not appear among the $\lambda_i$'s) we find that
\begin{align*}
&= \Res \Big[ \begin{matrix} g \cdot \str( \lambda_{i_1} \cdots \lambda_{i_p} \cdot d_X \lambda_1 ) \cdot \ud \bold{x} \;/\; (\partial_1 W)^2, \ldots \end{matrix} \Big]\\
&\qquad + (-1)^{p+1} \Res \Big[ \begin{matrix} g \cdot \str( \lambda_{i_1} \cdots \lambda_{i_p} \cdot d_X \lambda_1 ) \cdot \ud \bold{x} \;/\; (\partial_1 W)^2, \ldots \end{matrix} \Big]
\end{align*}
which is zero (we can assume that $p$ is even).
\end{proof}

\section{Connections and residues}\label{section:derhamsplit}

In this appendix we construct connections on complete rings, as a complement to our discussion of residues in Section \ref{section:residuesandtraces}. Let $S$ be a ring, $R$ a $S$-algebra and $\bold{t} = \{t_1,\ldots,t_n\}$ a quasi-regular sequence in $R$ such that $R/\bold{t} R$ is a finitely generated projective $S$-module and $R$ is separated and complete in the $\bold{t} R$-adic topology. We denote by $S[\bold{t}] = S[t_1,\ldots,t_n]$ the formal polynomial ring in the $t_i$, and make $R$ into a $S[\bold{t}]$-algebra in the natural way. Let us fix a $S$-linear section
\[
\sigma: P = R/\bold{t} R \lto R
\]
of the quotient map $\pi: R \lto P$. By extension of scalars there is a $S\llbracket \bold{t}\rrbracket$-linear map
\[
\sigma^*: P \otimes_S S\llbracket \bold{t}\rrbracket \lto R
\]
and one can show that this is a bijection \cite[(3.3.2)]{Lipman87}. This follows from quasi-regularity of $\bold{t}$ and the fact that $R$ is separated and complete in its $\bold{t}R$-adic topology, since together these imply that there exists for every $M \in \mathbb{N}^n$ and $r \in R$ unique residue classes $r_M \in R/\bold{t} R$ with
\begin{equation}\label{eq:expansion_r_sigma}
r = \sum_M \sigma(r_M) t^M.
\end{equation}
Hence we can identify $R$ as a $S$-module with $\prod_M P$ and from this one deduces that $\sigma^*$ is a bijection. It follows that $R$ is a finitely generated projective $S\llbracket \bold{t}\rrbracket$-module. If we denote by $\omega$ the canonical flat $S$-linear connection on $S\llbracket \bold{t}\rrbracket$ as a $S[\bold{t}]$-module
\begin{gather*}
\omega: S\llbracket \bold{t}\rrbracket \lto S\llbracket \bold{t}\rrbracket \otimes_{S[\bold{t}]} \Omega^1_{S[\bold{t}]/S},\\
\omega( f ) = \sum_{j=1}^n \frac{\partial f}{\partial t_j} \otimes \ud t_j
\end{gather*}
then there is an induced $S$-linear connection on $R$. It is given explicitly by
\begin{gather}
\nabla^0: R \lto R \otimes_{S[\bold{t}]} \Omega^1_{S[\bold{t}]/S},\\
\nabla^0(r) = \sum_{j=1}^n \sum_{M \in \mathbb{N}^n} M_j \sigma(r_M) \cdot t^{M-e_j} \otimes \ud t_j \label{eq:nabla_complete_ring_app}
\end{gather} 
where we write $t^M = t_1^{M_1} \cdots t_n^{M_n}$ and for a tuple $M$ involving negative integers we set $t^M = 0$. It is routine using (\ref{eq:nabla_complete_ring_app}) to check that $\nabla^0$ is flat, and since $\nabla^0 \sigma = 0$ it is also clear that $\Ker(\nabla^0) + \bold{t} R = R$. Note that we do not assume that $\sigma(1) = 1$, although such a section can be found \cite[(3.3.3)]{Lipman87}. For the next lemma, let $\delta$ denote the Koszul differential on the exterior algebra $R \otimes_{S[\bold{t}]} \Omega_{S[\bold{t}]/S}$ as defined in Section \ref{section:contractingkoszul}. 

\begin{lemma} If $S$ is a $\mathbb{Q}$-algebra then $\Im(\delta \nabla^0) = \bold{t} R$, and for any integer $p > 0$ the map $p \cdot 1_R + \delta \nabla^0$ is a $S$-linear bijection on $R$ identifying $\bold{t} R$ with $\bold{t} R$. Hence in this case, $\nabla^0$ is standard.
\end{lemma}
\begin{proof}
It is clear that $\Im(\delta \nabla^0) \subseteq \bold{t} R$. To see that $\delta \nabla^0$ is onto $\bold{t} R$ let $r \in \bold{t} R$ be given. Since $(t_i)_0 = 0$ we have $r_0 = 0$, and hence $\sum_{M \neq 0} |M|^{-1} \sigma(r_M) t^M$ maps to $r$ under $\delta \nabla^0$. A similar argument proves the other claims.
\end{proof}

\subsection{Residues}\label{section:lipmanres}

We keep the above notation, but we drop the assumption that $R$ is separated and complete in the $\bold{t} R$-adic topology and let $\nabla^0$ denote any flat $S$-linear connection on $R$ as a $S[\bold{t}]$-module satisfying $\Ker(\nabla^0) + \bold{t} R = R$. It follows from this hypothesis that there is a $S$-linear section $\sigma$ of $\pi: R \lto R/\bold{t} R$ such that $\nabla^0 \sigma = 0$, and we fix such a section. With $I = \bold{t} R$ it is clear that for $t > 0$
\[
\nabla^0( I^{t+1} ) \subseteq I^{t} \cdot (R \otimes_{S[\bold{t}]} \Omega^1_{S[\bold{t}]/S}),
\]
and hence $\nabla^0$ naturally extends to a flat $S$-linear connection on the $I$-adic completion $\widehat{R}$
\begin{equation}\label{eq:nabla_0_hat}
\widehat{\nabla}^0: \widehat{R} \lto \widehat{R} \otimes_{S[\bold{t}]} \Omega^1_{S[\bold{t}]/S}
\end{equation}
making the following diagram with canonical columns commute:
\[
\xymatrix@C+1pc{
R \ar[d]\ar[r]^(0.35){\nabla^0} & R \otimes_{S[\bold{t}]} \Omega^1_{S[\bold{t}]/S} \ar[d]\\
\widehat{R} \ar[r]_(0.35){\widehat{\nabla}^0} & \widehat{R} \otimes_{S[\bold{t}]} \Omega^1_{S[\bold{t}]/S}.
}
\]
The section $\sigma$ composes with the canonical map $R \lto \widehat{R}$ to give a section of $\widehat{R} \lto \widehat{R}/\bold{t} \widehat{R} \cong R/\bold{t} R$, which determines a flat $S$-linear connection on the completion, as in (\ref{eq:nabla_complete_ring_app}). We claim that this agrees with the connection (\ref{eq:nabla_0_hat}), so that after passing to the completion a connection on $R$ is completely determined by the associated section of $\pi$.

\begin{lemma}\label{lemma:lipmanresdiagram} For $r \in \widehat{R}$ we have
\begin{equation}\label{eq:lipmanresdiagram}
\widehat{\nabla}^0(r) = \sum_{j=1}^n \sum_{M \in \mathbb{N}^n} M_j \sigma(r_M) \cdot t^{M-e_j} \otimes \ud t_j
\end{equation}
\end{lemma}
\begin{proof}
This follows by applying $\widehat{\nabla}^0$ to the expansion (\ref{eq:expansion_r_sigma}) of $r$ in terms of $\sigma$ and using the Leibniz rule, together with $\nabla^0 \sigma = 0$.
\end{proof}


We are now prepared to explain how the residue of Definition \ref{defn:residuesymbol} fits into the theory developed in \cite{Lipman87}. Set $P = R/\bold{t} R$ and $E = \Hom_{S}(P, P)$. Given $r \in R$, Lipman defines in \cite[(3.7)]{Lipman87} an operator-valued power series 
\[
r^{\#} = \sum_M r^{\#}_M t^M \in E\llbracket \bold{t}\rrbracket
\]
where $r^{\#}_M \in E$ sends a residue class $p \in P$ to $( r \sigma(p) )_M$, using the notation of (\ref{eq:expansion_r_sigma}). Recall that we have a $S$-linear operator $\partial_i = \partial/\partial t_i$ on $R$ determined by $\nabla^0$. Applying the projection $(\ud t_i)^*$ to (\ref{eq:lipmanresdiagram}) yields the following lemma.

\begin{lemma}\label{lemma:lipmanagreeminor} For $1 \le i \le n$ we have $r^{\#}_{e_i} = \pi \circ \partial_i \circ r \circ \sigma$.
\end{lemma}

\begin{proposition}\label{prop:agreementlipman} The residue symbol of Definition \ref{defn:residuesymbol} agrees with the residue symbol of Lipman \cite[(3.7)]{Lipman87}, that is, for $s, r_1, \ldots, r_n \in R$,
\[
\sum_{\tau \in S_n} \sgn(\tau) \tr_S\big( s (r_1)^{\#}_{e_{\tau(1)}} \cdots (r_n)^{\#}_{e_{\tau(n)}} \big) = \tr_S\big( s [\nabla, r_1] \cdots [\nabla, r_n] \big).
\]
\end{proposition}
\begin{proof}
In view of Lemma \ref{lemma:reformulateressym} this follows from Lemma \ref{lemma:lipmanagreeminor}.
\end{proof}

It is clear from the development in \cite{Lipman87} that Lipman's residue symbol does not depend on the choice of section $\sigma$, and it follows that the residue symbol of Definition \ref{defn:residuesymbol} does not depend on the choice of connection $\nabla^0$.

\bibliographystyle{amsalpha}
\providecommand{\bysame}{\leavevmode\hbox to3em{\hrulefill}\thinspace}
\providecommand{\href}[2]{#2}

\end{document}